\documentclass[12pt,leqno]{article}
\usepackage{amsmath, amsthm, amsfonts, amssymb, color}
\usepackage{mathrsfs}
\newtheorem{thm}{Theorem}[section]
\newtheorem{cor}[thm]{Corollary}
\newtheorem{lem}[thm]{Lemma}
\newtheorem{prp}[thm]{Proposition}

\newtheorem{Remark}[thm]{Remark}
\theoremstyle{definition}

\newcommand{\scr}[1]{\mathscr #1}
\definecolor{wco}{rgb}{0.5,0.2,0.3}

\numberwithin{equation}{section}

\newcommand{\ua}{\uparrow}

\title{{\bf Stochastic Heat Equations for infinite strings  with Values in a  Manifold}\footnote{Supported in
 part by  NSFC (11371099, 11501361, 11671035, 11771037). Financial support by the DFG through
the CRC 1283 " Taming uncertainty and profiting from randomness and low regularity in analysis, stochastics
and their applications" is acknowledged.}}
\author{
{\bf    Xin Chen$^{a)}$, Bo Wu$^{b)}$, Rongchan Zhu$^{c,e)}$, Xiangchan Zhu$^{d,e,f)}$}
\thanks{E-mail address: chenxin217@sjtu.edu.cn (X. Chen), wubo@fudan.edu.cn(B.Wu),
zhurongchan@126.com(R.C.Zhu), zhuxiangchan@126.com(X.C.Zhu)}
\\
\footnotesize {$^{a)}$ Department of Mathematics, Shanghai Jiaotong University, Shanghai 200240, China}\\
\footnotesize {$^{b)}$ School  of Mathematical Sciences, Fudan
University, Shanghai 200433, China}\\
 \footnotesize{ $^{c)}$Department of Mathematics, Beijing Institute of Technology, Beijing 100081,  China}\\
\footnotesize{ $^{d)}$School of Science, Beijing Jiaotong University, Beijing 100044, China}
\\
\footnotesize{  $^{e)}$ Department of Mathematics, University of Bielefeld, D-33615 Bielefeld, Germany}\\
\footnotesize{ $^{f)}$ Academy of Mathematics and Systems Science,
Chinese Academy of Sciences, Beijing 100190, China}
}

\date{}

\begin{document}
\maketitle

\def\paral{/\kern-0.55ex/}
\def\parals_#1{/\kern-0.55ex/_{\!#1}}
\def\R{\mathbb R} \def\EE{\mathbb E} \def\P{\mathbb P}\def\Z{\mathbb Z} \def\ff{\frac} \def\ss{\sqrt}
\def\H{\mathbb H}
\def\HH{\mathbf{H}}
\def\DD{\Delta} \def\vv{\varepsilon} \def\rr{\rho}
\def\<{\langle} \def\>{\rangle} \def\GG{\Gamma} \def\gg{\gamma}
\def\ll{\lambda} \def\LL{\Lambda} \def\nn{\nabla} \def\pp{\partial}
\def\dd{\text{\rm{d}}}
\def\Id{\text{\rm{Id}}}\def\loc{\text{\rm{loc}}} \def\bb{\beta} \def\aa{\alpha} \def\D{\scr D}
\def\P{\mathbb{P}}
\def\E{\scr E} \def\si{\sigma} \def\ess{\text{\rm{ess}}}
\def\beg{\begin} \def\beq{\beg}  \def\F{\scr F}
\def\Ric{\text{{\bf Ric}}}
\def\Vol{\text{{\bf Vol}}}
\def\Var{\text{{\bf Var}}}
\def\Ent{\text{{\bf Ent}}}
\def\Scal{\text{{\bf Scal}}}
\def\Hess{\text{\rm{Hess}}}\def\B{\scr B}
\def\e{\text{\rm{e}}} \def\ua{\underline a} \def\OO{\Omega} \def\b{\mathbf b}
\def\oo{\omega}     \def\tt{\tilde} \def\Ric{\text{\rm{Ric}}}
\def\cut{\text{\rm{cut}}} \def\P{\mathbb P} \def\K{\mathbb K}
\def\ifn{I_n(f^{\bigotimes n})}
\def\fff{f(x_1)\dots f(x_n)} \def\ifm{I_m(g^{\bigotimes m})} \def\ee{\varepsilon}
\def\C{\scr C}
\def\PP{\scr P}
\def\M{\scr M}\def\ll{\lambda}
\def\X{\scr X}
\def\T{\scr T}
\def\A{\mathbf A}
\def\LL{\scr L}\def\LLL{\Lambda}
\def\gap{\mathbf{gap}}
\def\div{\text{\rm div}}
\def\Lip{\text{\rm Lip}}
\def\dist{\text{\rm dist}}
\def\cut{\text{\rm cut}}
\def\supp{\text{\rm supp}}
\def\Cov{\text{\rm Cov}}
\def\Dom{\text{\rm Dom}}
\def\Cap{\text{\rm Cap}}\def\II{{\mathbb I}}\def\beq{\beg{equation}}
\def\sect{\text{\rm sect}}\def\H{\mathbb H}

\begin{abstract}In the paper, we construct conservative Markov processes corresponding to the martingale solutions to the stochastic heat equation on $\mathbb{R}^+$ or $\mathbb{R}$ with values in  a general Riemannian maifold, which is only assumed to be complete and stochastic complete. This work is an extension of the previous paper \cite{RWZZ17} on finite volume case.

Moveover, we also obtain some functional inequalities associated to these Markov processes. This implies that on infinite volume case, the exponential ergodicity of the solution if the Ricci curvature is strictly positive and the  non-ergodicity of the process if the sectional curvature is negative.
\end{abstract}

\vskip 2cm

\section{Introduction}\label{sect1}

This work is the continuity of \cite{RWZZ17}, which is motivated by Tadahisa Funaki's pioneering work \cite{Fun92} and Martin Hairer's recent work \cite{Hai16}. Let $M$ be a $n$-dimensional compact Riemannian manifold. In  \cite{Hai16} Hairer considered the stochastic heat equation 
associated to the energy
$$E(u)=\frac{1}{2}\int_{S^1} g_{u(x)}(\partial_xu(x),\partial_xu(x))\dd x,$$
for smooth functions $u:S^1\rightarrow M$, and wrote the equation in the local coordinates formally:
\begin{equation}\label{eq1.1}\partial_t{u}^\alpha=\partial_x^2u^\alpha+\Gamma_{\beta\gamma}^\alpha(u)\partial_xu^\beta\partial_xu^\gamma+\sigma_i^\alpha(u)\xi_i,\end{equation}
where Einstein¡¯s convention of summation over repeated indices is implied and
$\Gamma_{\beta\gamma}^\alpha$ are the Christoffel symbols for the Levi-Civita connection of $(M, g)$, $\sigma_i$ are vector fields on $M$ and $\xi_i$ are independent space-time white noise. This equation may be also looked as certain kind of multi-component version of the KPZ equation.
By the theory of regularity structure recently developed in \cite{Hai14, BHZ16, CH16}, local well-posedness of (1.1) has been obtained in \cite{Hai16} (see more recent work \cite{BGHZ19}).

 By the Andersson-Driver's approximation of Wiener measure in \cite{AD99}, we know that there exists an explicit relation between the Langevin energy $E(u)$ and Wiener (Brownian bridge) measure. In particular, it has been obtained in \cite{AD99} that
 Wiener
(Brownian bridge) measure $\mu$ on $C([0,1];M)$ could been interpreted as the limit of a natural approximation of the  measure
$\exp(-E(u))\D u$, where $\D u$ denotes a `Lebesgue' like measure on path space.
Based on the above connection, one may think the solution to the stochastic heat equation (1.1) may have $\mu$ as an invariant (even symmetrizing) measure.

In \cite{RWZZ17}, starting from the Wiener measure (or Brownian bridge measure) $\mu$ on $C([0,1],M)$ we use the theory of Dirichlet forms to construct a natural evolution which admits $\mu$ as an invariant measure. Moreover, the relation between the evolution constructed in \cite{RWZZ17} and (1.1) has also been discussed in \cite{RWZZ17} by using the Andersson-Driver approximation. It is conjectured in \cite{RWZZ17} that the Markov processes constructed by Dirichlet form in \cite{RWZZ17} have the same law as the solution to \eqref{eq1.1}.
Since we consider the Wiener measure on $C([0,1],M)$ in \cite{RWZZ17}, the evolution corresponds to the stochastic heat equation on $[0,1]$ for different boundary conditions  with values in a compact
Riemannian manifold. In the paper, we extend the results in \cite{RWZZ17} from finite volume  $[0,1]$  to the half line $\mathbb{R}^+$ or the real line $\mathbb{R}$.

When $M=\mathbb{R}^n$ it is well-known that the law of Brownian motion on $C([0,\infty);\mathbb{R}^n)$ is an invariant measure of the following stochastic heat equatioin
$$\partial_tX=\frac{1}{2}\partial_x^2X+\xi,\quad  X(t,0)=0, $$
on $[0,\infty)\times [0,\infty)$. Here  $\xi$ is  space-time white noise. By similar calculation as that in \cite{FQ} we easily know that the distribution of a two-sided Brownian motion with a  shift given by Lebesgue measure is invariant under the following stochastic heat equation
$$\partial_tX=\frac{1}{2}\partial_x^2X+\xi, $$
on $[0,\infty)\times \mathbb{R}$. This suggests us to use  the law of Brownian motion on $C([0,\infty);M)$ or the law of two sided Brownian motion on $C(\mathbb{R};M)$ to construct the corresponding stochastic heat equation on $\mathbb{R}^+$ or $\mathbb{R}$ with values in a Riemannian manifold.

Similarly as in \cite{RWZZ17}, we construct the solution to stochastic heat equation  by using the following $L^2$-Dirichlet form  with the  reference measure $\mu=$ the law of Brownian motion on $M$/the law of two sided Brownian motion on $M$:
$$\E(F,G):=\frac{1}{2}\int \langle DF, DG\rangle_{\mathbf{H}}\dd\mu=\frac{1}{2}\sum_{k=1}^\infty\int D_{h_k}F D_{h_k}G  \dd\mu; \quad F,G\in\F C_b,$$
where $\{h_k\}_{k\geq1}$ is an orthonormal basis in ${\mathbf{H}}:=L^2(\mathbb{R}^+;\mathbb{R}^d)/L^2(\mathbb{R};\mathbb{R}^d)$, and  $\F C_b$ and $DF$ are the set of all cylinder functions and $L^2$-gradient respectively(refer to the definitions in Section 2). In this case, the associated Dirichlet-Form $\E$ is called \textbf{{$\mathbf{L^2}$}-Dirichlet form}.

\textbf{For the half line case}:
we consider the reference measure as the law of Brownian motion for the half line $\R^+$ on Riemannian path space $C([0,\infty);M)$
and choose the state space as some weighted $L^2$-space (see Section 2). By using a general integration by parts formula from \cite{CLW17} (see also appendix) we can construct a martingale solution to the stochastic heat equation with values in a general Riemannian manifold, which is complete and stochastic complete.

\textbf{For the whole line case}: we first construct the two sided Brownian motion $\hat{x}$ on $M$ with $\hat{x}(0)=o$ by an independent copy of Brownian motion on $M$. We consider the reference measure given by the law $\mu^o_\R$ of $\hat{x}$.  By this we   derive an integration by parts formula by using the  stochastic horizontal lift for independent copy (see Proposition \ref{p3.1} 
for the reason we choose it in this way). We also emphasize that the $L^2$-Dirichlet form is independent of the stochastic horizontal lift (see Remark \ref{r2-1}), which can be seen as  a tool to obtain the integration by parts formula and the closablity of the associated bilinear form (see Remark 3.1).
Moreover, we also consider the reference measure as $\mu_\R^\nu:=\int \mu^x_\R \nu(\dd x)$ with some Randon measure $\nu$ satisfying \eqref{l4-1-1}, which could be the volume measure on $M$ under 
some mild curvature condition (see Remark \ref{r4-1} below).
As mentioned before, the process corresponds to the stochastic heat equation on $\mathbb{R}$ without any boundary condition for $\nu$  given by the volume measure on $M$.
Here we mainly concentrate on the more complicated case that $\nu$ and the reference measure have infinite mass. We use a cut-off technique to find suitable test functions and prove the quasi-regularity of the $L^2$-Dirichlet form (see Theorem \ref{t 4-1}), which gives a Markov process as a martingale solution to  the stochastic heat equation on $[0,\infty)\times \mathbb{R}$ with values in a Riemannian manifold. It is not easy to obtain that the process is conservative in this case, since $1$ is not in the domain of the Dirichlet form. Under mild curvature condition we find suitable approximation functions in the domain of the $L^2$-Dirichlet form and obtain that the Markov process is conservative in the sense that the life time is infinity (see Theorem \ref{t4-2a}).

We also emphasize that the construction of the conservative Markov processes on general manifold with reference measure having infinite mass still holds for the finite volume case, especially for the free loop case with the reference measure $e^{c\int_0^1\mathrm{Scal}(\gamma(s))\dd s}\tilde \mu^\nu(\gamma)$, which is  conjectured  to be invariant measure for \eqref{eq1.1} in \cite{BGHZ19}.  Here $c\in \mathbb{R}, \tilde \mu^\nu:= \int \tilde \mu^x\nu(\dd x)$, with $\nu=p_1(x,x)\dd x$ and $\tilde \mu^x$ given by  the Brownian bridge
measure and $p_t$ is the heat kernel for $\frac{1}{2}\Delta$ and $\mathrm{Scal}$ denotes the scalar curvature.  For more details we refer to Remark \ref{r4-3} and Remark \ref{loop}.



In the final part of this paper, we  study functional inequalities associated to $L^2$-Dirichlet-Form, which implies the long time behavior  of the solutions to the stochastic heat equations for infinite string.  
In this case, the $L^2$-Dirichlet form  is not comparable with the O-U Dirichlet form  constructed in \cite{DR92}, we refer
 readers to \cite{A96,AE95,ALR93,AD99,CHL97,CLW11,CLW17,CW14,DR92,Dri92,Dri94,ELL97,FM,FW05,FW17,L04,MR00,S00,W99,W04,WW08,
 WW09,WW16,W1}
 and references therein for various study about O-U Dirichlet form on path and loop space.

 As we explained before, this case corresponds to SPDEs on infinite volume. The ergodicity property is different from
 that for the finite volume case (see \cite{RWZZ17}).  For different manifolds we have ergodicity or non-ergodicity for the associated Markov processes.
We establish the log-Sobolev inequality for the corresponding $L^2$-Dirichlet form if $\Ric\geq K>0$ for some constant $K$ and Poincar\'{e} inequality for compact Riemannian manifold
  with some suitable curvature condition (see Theorem 4.1),
  which implies the $L^2$-exponential ergodicity in this case; When $M=\mathbb{R}^n$, ergodicity still holds but the Poincar\'{e} inequality does not hold in this case (see Theorem 4.3); 
  When $M$ is not a Liouville manifold, the associated Dirichlet form $\E$ is reducible, which means that  the solutions to the stochastic heat equation are not ergodic.

  \textbf{Notations}: In this paper we use $C_c^m$ to denote $C^m$-differentiable functions  with compact support. We use $C_b^m$ to denote $C^m$-differentiable functions with bounded derivatives. For Hilbert space $H$ we also use $|\cdot|_H$ to denote the norm of it.

The rest of this paper is as follow:
In Section $2$, we will construct the stochastic heat equation for the half line case on general Riemannian manifold $M$. The stochastic heat equation for the whole line will be established in Section $3$, and the ergodicity or non-ergodicity property of the processes will be obtained in Section 4.

\section{The case of half line $\R^+$}\label{sect2}
Throughout the article, suppose that $M$ is a complete and stochastic complete Riemannian manifold with dimension $n$, and $\rho$ is the Riemannian distance on $M$.
In this section, we will construct the stochastic heat process on half line. We first introduce some notions.
Fix $o\in M$ ,
the  path space over $M$ is defined by
$$W_{\mathbb{R}^+}^{o}(M):=\{\gamma\in C([0,\infty);M):\gamma(0)=o\}.$$
Then $W^o_{\mathbb{R}^+}(M)$ is a Polish (separable metric)
space under the following uniform distance
$$d_\infty(\gamma,\sigma):=\sum_{k=1}^\infty
\frac{1}{2^k}\displaystyle\sup_{t\in [0,k]}\Big(\rho(\gamma(t),\sigma(t))\wedge 1\Big),\quad\gamma,\sigma\in W^o_{\mathbb{R}^+}(M).$$

In order to construct Dirichlet forms associated to stochastic heat equations for infinite strings
 on Riemannian path space,
we also define the following weighted $L^1$-distance:
\begin{equation}\label{eq 2.12-1}
\tilde{d}(\gamma,\eta):={\sum_{k=1}^{\infty}\frac{1}{2^k}\int_{k-1}^k\tilde{\rho}(\gamma(s),\eta(s))\dd s}, \quad \gamma,\eta\in W_{\mathbb{R}^+}^{o}(M),
\end{equation}
where $\tilde{\rho}=\rho\wedge1$. Obviously we have $\tilde d \le d_{\infty}$.
Let ${\bf E}^o_{\R^+}(M)$ be the closure of  $W^o_{\R^+}(M)$
with respect to the distance $\tilde{d}$, then ${\bf E}^o_{\R^+}(M)$ is a Polish space.

Let $O(M)$ be the orthonormal frame bundle over $M$, we
consider the following SDE,
\begin{equation}\label{eq2.1}
\begin{cases}
&\dd U_t=\sum_{i=1}^n H_i(U_t)\circ \dd W_t^i,\quad t\geq0\\
& U_0=u_o,
\end{cases}
\end{equation}
where $\{H_i\}_{i=1}^n$ is a canonical orthonormal basis of horizontal vector fields
$O(M)$, $u_o$ is a fixed orthonormal basis of $T_o M$ and $(W_t^i)_{t\ge 0}$, $1\le i \le n$ is a standard
$\mathbb{R}^n$-valued Brownian motion defined on a probability space $(\Omega, \mathscr{F},\mathbb{P})$.
Note that $M$ is stochastically complete, so $U_t$ is well defined for all $t\ge 0$.
Let $\pi: O(M) \rightarrow M$ denote the canonical projection, then $x_t:=\pi(U_t),\ t\geq0$ is the Brownian motion on
$M$ with initial point $o$, and $U_{\cdot}$ is the (stochastic) horizontal lift along
$x_{\cdot}$. Let $\mu^o_{\R^+}$ be the law of $x_{[0,\infty)}$, then $\mu^o_{\R^+}$
is a probability measure on $W^{o}_{\R^+}(M)$, and  the (stochastic) horizontal lift
$(U_t(\gamma))_{t\in [0,\infty)}$ is well defined for $\mu^o_{\R^+}$-a.s. $\gamma \in W^o_{\R^+}(M)$, (whose distribution
is the same as that of $(U_t)_{t\in [0,\infty)}$ under $\mathbb{P}$). Therefore $\mu^o_{\R^+}$ can be seen as  a probability measure
on ${\bf E}^o_{\R^+}(M)$ with support contained in $W^o_{\R^+}(M)$, and
$(U_t(\gamma))_{t\in [0,\infty)}$ is also well defined for $\mu^o_{\R^+}$-a.s. $\gamma \in {\bf E}^o_{\R^+}(M)$.

Let $\F C_b^1$ be the space of
$C_b^1$ cylinder functions on ${\bf E}^o_{\R^+}(M)$ defined as follows: for every $F\in \F C_b^1$, there exist some $m\geq1, ~m\in \mathbb{N}^+,~ f\in C_b^{1}(\mathbb{R}^m), g_j\in C_b^{0,1}([0,\infty)\times M)$,
$T_j\in[0,\infty)$, $j=1,...,m$,
such that
\begin{equation}\label{eq 3.1}\aligned
F(\gamma)=f\left(\int_0^{T_1}  g_1(s,\gamma(s)) \dd s,\int_0^{T_2}   g_2(s,\gamma(s)) \dd s,...,\int_0^{T_m}   g_m(s,\gamma(s)) \dd s\right),\quad \gamma\in {\bf E}^o_{\R^+}(M).\endaligned\end{equation}
 Here $C_b^{0,1}([0,\infty)\times M)$ denotes the bounded functions which are continuous w.r.t. the first variable and $C^1_b$-
differentiable w.r.t. the second variable.
It is easy to see that $F$ is well defined for  $\gamma \in {\bf E}^o_{\R^+}(M)$,
$\F C_b^1$ is dense in $L^2({\bf E}^o_{\R^+}(M); \mu^o_{\R^+})=L^2(W^o_{\R^+}(M); \mu^o_{\R^+})$.
For any $F\in \F C_b^1$ of the form \eqref{eq 3.1} and $h\in \HH_+:=L^2([0,\infty)\rightarrow \R^n;\dd s)=\{h: [0,\infty)\rightarrow \R^n;
\int_0^{\infty}|h(s)|^2 \dd s<\infty\}$, the directional derivative of $F$ with respect to $h$ is ($\mu^o_{\R^+}$-a.s.) defined by
$$D_hF(\gamma)=\sum_{j=1}^m\hat{\partial}_jf(\gamma)\int_0^{T_j}
 \left\langle U_s^{-1}(\gamma)\nabla g_j(s,\gamma(s)),h(s)
 \right\rangle \dd s,\quad \gamma\in{\bf E}^o_{\R^+}(M),$$
where
$$\hat{\partial}_jf(\gamma):=\partial_jf\bigg(\int_0^{T_1}  g_1(s,\gamma(s)) \dd s,\int_0^{T_2}  g_2(s,\gamma(s)) \dd s,...,\int_0^{T_m} g_m(s,\gamma(s)) \dd s\bigg).$$
and $\nabla g_j$ denotes the gradient w.r.t. the second variable.
By the Riesz representation theorem,
there exists a gradient operator $DF(\gamma)\in \HH_+$ such that $\langle DF(\gamma),h\rangle_{\HH_+}=D_hF(\gamma)$, $\mu^o_{\R^+}$-a.s.$\gamma\in {\bf E}^o_{\R^+},
h\in \HH_+$. In particular, for $\gamma\in W^o_{\R^+}(M)$,
\begin{equation}\label{eq2.3}\aligned DF(\gamma)(s)=\sum_{j=1}^m\hat{\partial}_jf(\gamma)U_s^{-1}(\gamma)\nabla g_j(s,\gamma(s))1_{(0,T_j] }(s).\endaligned\end{equation}

 We define the (Cameron-Martin) subspace $\H^\infty_+$ of $\HH_+$ as follows
\begin{equation}\label{eq2.3-1}
\H^\infty_+:=\left\{h\in C_c^1([0,\infty);\mathbb{R}^d)\Big| h(0)=0,\int^\infty_0|h'(s)|^{2}\dd s<\infty\right\}.
\end{equation}
Fix a sequence of elements $\{h_k\}_{k=1}^{\infty}\subset \H^\infty_+$
such that it is an orthonormal basis in $\HH_+$, we define the following symmetric quadratic form
as follows
\begin{equation}\label{DF}\E^o_{\R^+}(F,G):=\frac{1}{2}\int_{{\bf E}^o_{\R^+}(M)}
\langle DF, DG\rangle_{\HH_+}\dd\mu^o_{\R^+}, \quad F,G\in\F C_b^1.\end{equation}
Then it is obvious that
$$\E^o_{\R^+}(F,G)=\frac{1}{2}\sum_{k=1}^\infty\int_{{\bf E}^o_{\R^+}(M)} D_{h_k}F D_{h_k}G  \dd\mu^o_{\R^+}; \quad F,G\in\F C_b^1.$$

\begin{Remark}\label{r2-1}
Although the stochastic horizontal lift $(U_t(\gamma))_{t\in [0,\infty)}$ is applied
in the definition of $(\E^o_{\R^+}, \F C_b^1)$, the value of $\E^o_{\R^+}(F,F)$ is independent of $(U_t(\gamma))_{t\in [0,\infty)}$.
In particular, by the definition \eqref{eq2.3} of the gradient, we have
$$\E^o_{\R^+}(F,G)=\frac{1}{2}\int_{{\bf E}^o_{\R^+}(M)}
\sum_{i=1}^m\sum_{j=1}^l\hat{\partial}_if_1(\gamma)\hat{\partial}_jf_2(\gamma)\int^{T_i\wedge T_j}_0\langle \nabla g^1_i(s,\gamma(s)),\nabla g^2_j(s,\gamma(s))\rangle \dd s\dd\mu^o_{\mathbb{R}^+}$$
for any $F,G\in\F C_b^1$ with
$$\aligned
&F(\gamma)=f_1\left(\int_0^{T_1}  g^1_1(s,\gamma(s)) \dd s,\int_0^{T_2}   g^1_2(s,\gamma(s)) \dd s,...,\int_0^{T_m}   g^1_m(s,\gamma(s)) \dd s\right)\\
&G(\gamma)=f_2\left(\int_0^{T_1}  g^2_1(s,\gamma(s)) \dd s,\int_0^{T_2}   g^2_2(s,\gamma(s)) \dd s,...,\int_0^{T_l}   g^2_l(s,\gamma(s)) \dd s\right),\quad \gamma\in {\bf E}^o_{{\R^+}}(M),\endaligned$$
for $f_1\in C_b^{1}(\mathbb{R}^m), f_2\in C_b^{1}(\mathbb{R}^l),  g_j^i\in C_b^{0,1}([0,\infty)\times M)$ $i=1,2, j=1,...,m$.
This implies the quadratic form $\E^o_{\R^+}$ is independent of $(U_t(\gamma))_{t\in [0,\infty)}$.

\end{Remark}

\beg{thm}\label{T2.1}  The quadratic form $(\E^o_{\R^+}, \F C_b^1)$
is closable and its closure $(\E^o_{\R^+},\D(\E^o_{\R^+}))$ is a quasi-regular Dirichlet form on $L^2({\bf E}^o_{\R^+}(M);\mu^o_{\R^+})$.
\end{thm}

By using the theory of Dirichlet form (refer to \cite{MR92}), we obtain the following associated diffusion process.

 \beg{thm}\label{T2.2} There exists a conservative (Markov) diffusion process
 $\mathbf{M}=(\Omega,\F,\M_t,$ $(X(t))_{t\geq0},(\mathbf{P}^z)_{z\in {\bf E}^o_{\R^+}(M)})$ on ${\bf E}^o_{\R^+}(M)$ having $\mu^o_{\R^+}$ as an invariant measure and  \emph{properly associated with} $(\E_{\R^+}^o,\D(\E^o_{\R^+}))$, i.e. for $u\in L^2({\bf E}^o_{\R^+}(M);\mu^o_{\R^+})\cap\B_b({\bf E}^o_{\R^+}(M))$, the transition semigroup $P_tu(z):=\EE^z[u(X(t))]$ is an $\E^o_{\R^+}$-quasi-continuous version of $T_tu$ for all $t >0$, where $T_t$ is the semigroup associated with $(\E^o_{\R^+},\D(\E^o_{\R^+}))$.
\end{thm}

  Here for the notion of $\E^o_{\R^+}$-quasi-continuity we refer to \cite[Definition III-3.2]{MR92}. By Fukushima decomposition we have
\beg{thm}\label{T2.4}  There exists a \emph{properly  $\E^o_{\R^+}$-exceptional set} $S\subset {\bf E}^o_{\R^+}(M)$, i.e.
$\mu_{\R^+}^o(S)=0$ and $\mathbf{P}^z[X(t)\in {\bf E}^o_{\R^+}(M)\setminus S, \forall t\geq0]=1$ for
$z\in {\bf E}^o_{\R^+}(M)\backslash S$, such that $\forall z\in {\bf E}^o_{\R^+}(M)\backslash S$ under $\mathbf{P}^z$,  the sample paths of the associated  process $\mathbf{M}=(\Omega,\F,\M_t,$ $(X(t))_{t\geq0},(\mathbf{P}^z)_{z\in {\bf E}^o_{\R^+}(M)})$ on ${\bf E}^o_{\R^+}(M)$ satisfy the following for
$u\in \D(\E^o_{\R^+})$
 \begin{equation}\label{eq2.4}\aligned u(X_t)-u(X_0)=M_t^u+N_t^u\quad \mathbf{P}^z-a.s.,\endaligned\end{equation}
 where $M^u$ is a martingale with quadratic variation process given by $\int_0^t |Du(X_s)|_{\HH_+}^2ds$ and $N_t$ is zero quadratic variation process. In particular, for $u\in D(L)$, $N_t^u=\int_0^tLu(X_s)ds$, where $L$ is the generator of
 $(\E^o_{\R^+},\D(\E^o_{\R^+}))$.
\end{thm}

\beg{Remark}\label{r2.5}
If we choose $u(\gamma)=\int_{r_1}^{r_2}u^\alpha(\gamma(s))ds\in\F C_b^1$, with $u^\alpha$ is a local coordinate on $M$, then the quadratic variation process for $M^u$ is the same as  that for the martingale part in (1.1).
\end{Remark}

To prove Theorem \ref{T2.1}, the crucial ingredient is the local integration by parts formula in \cite{CLW17}.  To do that, we need to introduce some notations.  In the following,  we first introduce another  cylinder functions set, every element in which only depends on finite times:
\begin{equation*}
\begin{split}
\hat \F C_b:=\Big\{&W^o_{\R^+}(M)\ni\gamma\mapsto f(\gamma(t_1),\cdots,\gamma(t_m)):\ m\geq1,\\
&~~~~~~~~~~~~~~~~~~~~~~~~
0<t_1<t_2\cdots<t_m<\infty,f\in C_{b,Lip}(M^m)\Big\},
\end{split}
\end{equation*}
where $C_{b,Lip}(M^m)$ denotes the collection of bounded Lipschitz continuous functions on $M^m$.

For a fixed $o \in M$, since $M$ is complete,
there exists a $C^\infty$ non-negative smooth function $\hat \rho:M\rightarrow\R$ with the property that $0<|\nabla \hat \rho(z)|\le 1$ and
$$\left|\hat \rho(z)-\frac{1}{2}\rho(o,z)\right|<1,\quad z\in M.$$
For every non-negative $m$, define
 \begin{equation}\label{cut}D_m:=\left\{z \in M: \hat \rho(z)<m\right\},\quad \tau_m(\gamma):=\inf\left\{s \ge 0:\ \gamma(s) \notin D_m\right\}.\end{equation}

We first introduce the following two results in \cite{CLW17} and \cite{CLW18}, for convenience of readers we will give
the proof of them in the Appendix

\begin{lem}\label{l2.5}{\bf [Chen-Li-Wu \cite{CLW17}]}
For any $m \in \mathbb{N}^+$ and $T \in \R^+$, there exists a stochastic process(vector fields)
$l_{m,T}: [0,\infty)\times W^o_{\R^+}(M) \rightarrow [0,1]$  such that

\begin{enumerate}

\item [(1)]  $l_{m,T}(t,\gamma)=\left\{ \begin{array}{ll} 1, \qquad &t < \tau_{m-1}(\gamma)\wedge T\\
0, &t > \tau_{m}(\gamma)\end{array}\right..$

\item[(2)]  Given any $o \in D_m$,
$l_{m,T}(t,\gamma)$ is $\mathscr{F}_t^{\gamma}:=\sigma\{\gamma(s);s \in [0,t]\}$-adapted and
$l_{m,T}(\cdot,\gamma)$ is absolutely continuous for $\mu^o_{\R^+}$-a.s. $\gamma \in W^o_{\R^+}(M)$.

\item[(3)] For any positive integers $k,p,m \in \mathbb{Z}_+$ and $t\in \R^+$, we have
\begin{equation}\label{eq2.5}
\sup_{o \in D_m}\int_{W_{\R^+}^o(M)} \int^t_0|l_{k,T}'(s, \gamma)|^p \dd s \mu_{\R^+}^o(\dd\gamma)\le C_1(m,k,p,T)
\end{equation}
for some positive constant $C_1(m,k,p,T)$ (which may depends on $m$, $T$, $p$ and $k$).

\end{enumerate}

\end{lem}

\begin{lem}\label{l2.6}{\bf [Chen-Li-Wu \cite{CLW18}]}
Let $l_{m,T}$ be the cut-off process constructed in Lemma \ref{l2.5}, then for
every $F \in \hat \F C_b$, $m\in \Z^+$, $T \in \R^+$, $h\in \H^{\infty}_+$ (see  \eqref{eq2.3-1}),
the following integration by parts formula holds
\begin{equation}\label{eq2.6}\aligned
&\int_{W^o_{\R^+}(M)}\left(\dd F (U_{\cdot}l_{m,T}(\cdot)h(\cdot) )\right)\mu^o_{\R^+}(\dd \gamma)\\
&=\int_{W^o_{\R^+}(M)}\left( F\int_0^{\infty} \left\langle(l_{m,T} h)'(s)+\frac{1}{2}\Ric_{U_s}\left(l_{m,T}(s)h(s)
\right), \dd \beta_s \right\rangle\right)\mu^o_{\R^+}(\dd \gamma),
\endaligned
\end{equation}
where $\beta_t$ denotes the anti-development of $\gamma(\cdot)$,
whose distribution is a standard $\R^n$-valued Brownian motion under $\mu_{\R^+}^o$.


\end{lem}



Based on the above Lemma \ref{l2.6}, and using an approximation procedure, it is not difficult to obtain the following integration by parts formula.

\begin{lem}\label{l2.7}
Let $l_{m,T}$ be mentioned in Lemma \ref{l2.6},  then for
every $F \in \F C_b^1$, $m \in \mathbb{Z}_+$,
$T\in \R^+$, $h\in \H^{\infty}_+$, the following integration by parts formula holds
\begin{equation}\label{eq2.8}\aligned
&\int_{{\bf E}^o_{\R^+}(M)}\left(\dd F (U_{\cdot}l_{m,T}(\cdot)h(\cdot) )\right)\mu^o_{\R^+}(\dd \gamma)\\
&=\int_{{\bf E}^o_{\R^+}(M)}\left( F\int_0^{\infty} \left\langle(l_{m,T} h)'(s)+\frac{1}{2}\Ric_{U_s}\left(l_{m,T}(s)h(s)
\right), \dd \beta_s \right\rangle\right)\mu^o_{\R^+}(\dd \gamma),
\endaligned
\end{equation}
where $\beta_t$ denotes the anti-development of $\gamma(\cdot)$, which is a Brownian motion under
$\mu_{\R^+}^o$.
Here $\mu^o_{\R^+}$ can be seen as  a probability measure
on ${\bf E}^o_{\R^+}(M)$ with support contained in $W^o_{\R^+}(M)$, and
$l_{m,T}(t,\gamma)$ is also well defined for $\mu^o_{\R^+}$-a.s. $\gamma \in {\bf E}^o_{\R^+}(M)$.
\end{lem}

\beg{proof}  In fact,
it suffices to check the result holds for $F(\gamma)=f(\int_0^t  g(s,\gamma(s)) \dd s)\in \F C_b^1$ with arbitrarily
pre-fixed $t\in \R^+$, and the general case can be handled similarly. For any $k\geq1$, defining
$$F_k(\gamma)= f\left(\frac{1}{k}\sum_{i=1}^{[kt]} g(i/k,\gamma(i/k))\right).$$
Fix a time $T>t>0$, then
\begin{equation}\label{eq2.6-1}\aligned
&\int_{{\bf E}^o_{\R^+}(M)}\left(\dd F_k) (U_{\cdot}l_{m,T}(\cdot)h(\cdot) \right)\mu^o_{\R^+}(\dd \gamma)\\
&=\int_{{\bf E}^o_{\R^+}(M)}\left( F_k\int_0^{\infty} \left\langle(l_{m,T} h)'(s)+\frac{1}{2}\Ric_{U_s}\left(l_{m,T}(s)h(s)
\right), \dd \beta_s \right\rangle\right)\mu^o_{\R^+}(\dd \gamma),
\endaligned
\end{equation}
where we used $\supp(\mu^o_{\R^+})\subset W^o_{\R^+}(M)$ and \eqref{eq2.6}.

By the dominated convergence theorem, it is easy to see that $F_k\rightarrow F$ in
$L^2({\bf E}^o_{\R^+}(M);\mu^o_{\R^+})$ as $k\rightarrow\infty.$ According to the definition of directional derivative, we have
$$\aligned
&\dd F (U_{\cdot}l_{m,T}(\cdot)h(\cdot) )=\langle D F, l_{m,T}h\rangle_{\HH_+}=
\hat{\partial}f(\gamma)\int_0^t
 \left\langle U_s^{-1}(\gamma)\nabla g(s,\gamma(s)),(l_{m,T}h)(s)
 \right\rangle_{\mathbb{R}^d} \dd s\\
 &\dd F_k (U_{\cdot}l_{m,T}(\cdot)h(\cdot) )=\langle D F_k, l_{m,T}h\rangle_{\HH_+}=
 \frac{1}{k}\hat{\partial}f_k(\gamma)\sum_{i=1}^{[kt]}
 \left\langle U_{i/k}^{-1}(\gamma)\nabla g(i/k,\gamma(i/k)),(l_{m,T}h)(i/k)
 \right\rangle_{\mathbb{R}^n},\endaligned$$
 with $\hat{\partial}f=f'\left(\int_0^t  g(s,\gamma(s)) \dd s\right)$  and $\hat{\partial}f_k=f'\left(\frac{1}{k}\sum_{i=1}^k g(i/k,\gamma(i/k))\right).$
 By our assumptions for $f$ and $g$ (especially $\nabla g$ is bounded)
 we know that
 $$\langle D F_k, l_{m,T}h\rangle_{\HH_+}\rightarrow\langle D F, l_{m,T}h\rangle_{\HH_+}\textrm{ in }
 L^2({\bf E}^o_{\R^+}(M);\mu^o_{\R^+}),\quad  k\rightarrow\infty.$$

By using the above argument, we get \eqref{eq2.8} by taking $k \rightarrow \infty$ on both sides of the equation \eqref{eq2.6-1} .

\end{proof}

In the following we  prove Theorem \ref{T2.1} by using  the above integration by parts formula.

\ \newline\emph{\bf Proof of Theorem \ref{T2.1}.}
{\bf$(a)$ Closablity:}
Let $\{F_m\}_{m=1}^{\infty}\subseteq \F C_b^1$ be a sequence of cylinder functions with
\begin{equation}\label{eq2.10}
\begin{split}
\lim_{m \rightarrow \infty}\mu^o_{\R^+}\left(F_m^2\right)=0,\ \
\lim_{k,m \rightarrow \infty}\E^o_{\R^+}\left(F_k-F_m,F_k-F_m\right)=0.
\end{split}
\end{equation}
Thus $\{D F_m\}_{m=1}^{\infty}$ is a Cauchy sequence in
$L^2\left({\bf E}^o_{\R^+}(M)\rightarrow \HH_+;\mu^o_{\R^+}\right)$, for which there exists a limit $\Phi$. It only suffices to prove that $\Phi=0$.
Suppose that $\{h_i\}_{i=1}^{\infty}\subset \H^{\infty}_+\cap C_c^1([0,\infty);\R^n)$ is an orthonormal basis of $\HH_+$.
By Lemma \ref{l2.7},
for each $G \in \F C_b^1$ and any positive integers $k,m,i\geq1$, we have
\begin{equation}\label{eq2.11}
\begin{split}
&\mu^o_{\R^+}\left(\langle D F_k, l_{m,T} h_i\rangle_{\HH_+}G\right)\\
&=\mu^o_{\R^+}\left(\langle D \left(F_kG\right), l_{m,T} h_i\rangle_{\HH_+}\right)
-\mu^o_{\R^+}\left(\langle D G, l_{m,T}h_i\rangle_{\HH_+}F_k\right)\\
&=\mu^o_{\R^+}\left(F_kG \int_0^\infty\left\langle(l_{m,T} h_i)'(s)+\frac{1}{2}\Ric_{U_s}\left(l_{m,T}(s)
h_i(s)\right)
, \dd \beta_s\right\rangle\right)\\
&~~~-\mu^o_{\R^+}\left(\langle D G, l_{m,T}h_i\rangle_{\HH_+}F_k\right).
\end{split}
\end{equation}
In particular, for each $h_i\in C_c^1([0,\infty);\R^n)$, by \eqref{eq2.5} and the compact property of $\supp(h_i)$, we have
$$\int_0^\infty\left\langle(l_{m,T} h_i)'(s)+\frac{1}{2}\Ric_{U_s}\left(l_{m,T}(s)
h_i(s)\right), \dd \beta_s\right\rangle\in L^2({\bf E}^o_{\R^+}(M);\mu^o_{\R^+}).$$

Since $G$ and $D G$ are bounded, and $F_k\rightarrow 0$, $|DF_k-\Phi|_{\HH_+}\rightarrow 0$
in $L^2({\bf E}^o_{\R^+}(M);\mu^o_{\R^+})$, we let
 $k \rightarrow \infty$ in \eqref{eq2.11} and obtain that for every $m,T,i\in \mathbb{N}^+$,
\begin{equation*}
\begin{split}
&\mu^o_{\R^+}\left(\langle \Phi, l_{m,T}h_i\rangle_{\HH_+}G\right)=0,\quad \forall\
G \in \F C_b^1.\
\end{split}
\end{equation*}
Therefore we could find a $\mu^o_{\R^+}$-null set $\Delta_i \subset {\bf E}^o_{\R^+}(M)$, such that
\begin{equation}\label{eq2.12}
\langle \Phi(\gamma), l_{m,T}(\gamma)h_i\rangle_{\HH_+}=0,\quad \ \forall\ m,T\in \mathbb{N}^+,\ \gamma \notin \Delta_i.
\end{equation}

For a fixed $h_i\in \H^{\infty}_+$, there exists a positive integer  $T_i\in \mathbb{N}^+$ (which may depend on $h_i$) such that
$\supp (h_i)\subset [0,T_i]$. Since $\gamma(\cdot)$ is non-explosive, there is a $\mu^o_{\R^+}$-null set $\Delta_0\subset {\bf E}^o_{\R^+}(M)$ such that
for every $\gamma \notin \Delta_0$, there exists  $m_i(\gamma)\in \mathbb{N}^+$ satisfying
$$\gamma(t)\in D_{m_i-1}, \quad\text{ for all}~~ t \in [0,T_i],$$ where $D_{m_i-1}$ is introduced by \eqref{cut}. Hence $l_{m_i,T_i}(t,\gamma)=1$ for all $t \in [0,T_i]$.
Combining this with \eqref{eq2.12} we know
$$\langle \Phi(\gamma), h_i\rangle_{\HH_+}=0,\quad  i\ge 1,
\gamma \notin \Delta_i\cup \Delta_0,$$
which implies that
$\Phi(\gamma)=0$, $\forall\ \gamma \notin \Delta=\cup_{i=0}^{\infty}\Delta_i$. So $\Phi=0$, $\mu_{\R^+}^o$-a.s., and
$(\E_{\R^+}^o,\F C_b^1)$ is closable. By the standard method, we show easily that its closure
$(\E_{\R^+}^o,\D(\E^o_{\R^+}))$ is a Dirichlet form.

{\bf$(b)$  Quasi-Regularity:}
In order to prove the quasi-regularity of $(\E_{\R^+}^o,\D(\E^o_{\R^+}))$,
 we need to verify conditions (i)-(iii) in \cite[Definition IV-3.1]{MR92}.

It is easy to see that each $G\in \F C_b^1$ is continuous in (Polish space) $({\bf E}^o_{\R^+}(M), \tilde d)$, and
$\F C_b^1$ is dense in $\D(\E^o_{\R^+})$ under the $(\E^o_{\R^+,1})^{1/2}$-norm with
$$\E^o_{\R^+,1}(\cdot,\cdot):=\E^o_{\R^+}(\cdot,\cdot)+\|\cdot\|^2_{L^2({\bf E}^o_{\R^+}(M),\mu^o_{\R^+})}.$$
 So (ii) of \cite[Definition IV-3.1]{MR92}
holds.

Since the metric space $({\bf E}^o_{\R^+}(M), \tilde d)$ is separable,
we can choose a fixed countable dense subset $\{\xi_m|m\in\mathbb{N}^+\}\subset W^o_{\R^+}(M)$.
Next, we  prove the tightness of the capacity for $(\E^o_{\R^+},\D(\E^o_{\R^+}))$ which ensures
(i) of \cite[Definition IV-3.1]{MR92}.

Let $\varphi\in C_b^\infty(\mathbb{R})$ be an increasing function  satisfying with
$$\varphi(t)=t,\quad \forall~t\in[-1,1]~~\text{and}~~\|\varphi'\|_{\infty}\leq 1.$$

For each $m\geq1$, the function $v_m:{\bf E}^o_{\R^+}(M)\rightarrow\mathbb{R}$ is given by
$$v_m(\gamma)=\varphi(\tilde{d}(\gamma,\xi_m)),\quad\gamma\in {\bf E}^o_{\R^+}(M),$$
with $\tilde{d}$ defined in \eqref{eq 2.12-1}.
By Lemma \ref{l2.8} below $v_m\in \D(\E^o_{\R^+})$.
We claim that
\begin{equation}\label{eq 2.13}w_k:=\inf_{m\leq k}v_m, k\in \mathbb{N}^+, \textrm{ converges } \E^o_{\R^+} -\text{quasi-uniformly to zero on} ~{\bf E}^o_{\R^+}(M).\end{equation}
Then  for every $i\in\mathbb{N}^+$ there exists a closed set $K_i$ such that $\textrm{Cap}(K_i^c)<\frac{1}{i}$ and $w_k\rightarrow0$ uniformly on $K_i$ as $k\rightarrow \infty$. Here Cap is the capacity associated to $(\E^o_{\R^+},\D(\E^o_{\R^+}))$ 
(see \cite[Section III.2]{MR92}). Hence for every $0<\varepsilon<1$ there exists $k\in\mathbb{N}^+$ such that $w_k<\varepsilon$ on $K_i$, by using the definitions of $v_m$ and $w_k$, we obtain that $K_i\subset \cup_{m=1}^k B(\xi_m,\varepsilon)$, where $B(\xi_m,\varepsilon):=\{\gamma\in {\bf E}^o_{\R^+}(M);
\tilde d(\xi_m,\gamma)<\varepsilon\}$.
   Consequently, for every $i\ge 1$, $K_i$ is totally bounded, hence compact. Combining this with
 the fact $\lim_{i \rightarrow \infty} \textrm{Cap}(K_i^c)=0$ we know the capacity for $(\E^o_{\R^+},\D(\E^o_{\R^+}))$ is tight.

Now it only remains to show the claim \eqref{eq 2.13}. For each fixed $m\geq1$, by \eqref{l2-8-2} in Lemma \ref{l2.8} below we obtain
\begin{equation*}
Dv_m(\gamma)(s)=\varphi'(\tilde{d}(\gamma,\xi_m))\cdot\Big(\sum_{k=1}^{\infty}
\frac{1}{2^k} U_s^{-1}\nabla_1 \tilde{\rho}(\gamma(s),\xi_m(s))1_{(k-1,k]}(s)\Big),
\end{equation*}
where $\nabla_1\tilde \rho$ is the gradient of $\tilde \rho$ with respect to the first variable.
By the definition \eqref{DF} of the quadratic form $\E^o_{\R^+}$, we have
\begin{equation}\label{eq2.14}
\aligned
&\E^o_{\R^+}(v_m,v_m)=\frac{1}{2}\int_{{\bf E}^o_{\R^+}(M)}\big|Dv_m(\gamma)\big|^2_{\HH_+}\dd\mu^o_{\R^+}(\gamma)\\
&=\frac{1}{2}\sum_{k=1}^{\infty}\frac{1}{2^{2k}}\int_{{\bf E}^o_{\R^+}(M)}|\varphi'(\tilde{d}(\gamma,\xi_m))|^2\cdot
\Big(\int_{k-1}^k \big|\nabla_1 \tilde{\rho}(\gamma(s),\xi_m(s))\big|^2_{T_{\gamma(s)}M}\dd s\Big)\dd\mu^o_{\R^+}(\gamma)\\
&\le \|\varphi'\|_{\infty}\cdot \Big(\sum_{k=1}^{\infty}\frac{1}{2^{2k+1}}\Big)\le C,\ \forall\ m \in \mathbb{N}^+,
\endaligned\end{equation}
where $C>0$ is a constant independent of $m$, and in the first inequality above we  applied the property that
$|\nabla_1 \rho|\le 1$.

Since $\{\xi_m|m\in\mathbb{N}\}$ is dense in $({\bf E}^o_{\R^+}(M); \tilde d)$,
it is easy to verify that $w_k\downarrow0$ $\mu^o_{\R^+}$-a.s.
on ${\bf E}^o_{\R^+}(M)$ hence in $L^2({\bf E}^o_{\R^+}(M);\mu^o_{\R^+})$.
By \eqref{eq2.14} we arrive at
$$\E^o_{\R^+}(w_k,w_k)\leq  C,\quad \forall k\in \mathbb{N}^+,$$
 where $C$ is independent of $k$.

 Based on this and \cite[I.2.12, III.3.5]{MR92} we obtain that a subsequence of the Cesaro mean of some subsequence of $w_k$ converges to zero $\E^o_{\R^+}$-quasi-uniformly. But since $\{w_k\}_{k\in\mathbb{N}^+}$ is decreasing, \eqref{eq 2.13} follows. Now tightness in (i) of \cite[Definition IV-3.1]{MR92} follows.

{For any $\gamma, \eta\in {\bf E}^o_{\R^+}(M)$ with $\varepsilon:=\tilde{d}(\gamma,\eta)>0$},  there exists certain $\xi_N$ such that $\tilde{d}(\xi_N,\eta)<\frac{\varepsilon}{4}$ and $\tilde{d}(\xi_N,\gamma)>\frac{\varepsilon}{4}$. Take $\{F_m(\gamma):=\varphi(\tilde{d}(\xi_m,\gamma)),m\in \mathbb{N}\}$ for $\varphi$ as above,
(iii) of \cite[Definition IV-3.1]{MR92} follows.

$\hfill\square$

For a locally Lipschitz continuous function $g:M \rightarrow \R$, by Radamacher's theorem,  it is well known that
the gradient $\nabla g(x)$ of $g$ exists for all $x \in M/S$ with some Lebesgue null set $S\subset M$. For convenience, let us define $\nabla g(x)=0$ for any $x\in S$.
Also note that $\mu^o_{\R^+}\big(\gamma(s)\in S\big)=0$ for each $s>0$, hence $\nabla g(\gamma(s))$
is $\mu^o_{\R^+}$-a.s. well defined for every $s>0$.

Let $C_{b,Lip}([0,\infty)\times M)$ be the set of all functions $f$ on the product space $[0,\infty)\times M$ and each function $g(t,x)$ is bounded and continuous with respect to the first variable $t\in [0,\infty)$, and uniformly Lipschitz continuous with respect to the second variable $x\in M$.

\begin{lem}\label{l2.8}
\begin{itemize}
\item [(1)] 
For each fixed function $F(\gamma):=f(\int_0^t g(s,\gamma(s))\dd s)$ with some fixed $t>0,g\in C_{b,Lip}([0,\infty)\times M)$ and
 $f\in C_b^1(\mathbb{R})$. Then $F \in \D(\E^o_{\R^+})$ and we have
\begin{equation}\label{l2-8-1}
DF(\gamma)(s)=f'\left(\int_0^t g(r,\gamma(r))\dd r\right)\cdot\Big(U_s^{-1}(\gamma)\nabla g(s,\gamma(s))1_{(0,t]}(s)\Big)
\end{equation}
for $\dd s\times \mu^o_{\R^+}-a.s.
  (s,\gamma)\in [0,\infty)\times {\bf E}^o_{\R^+}(M).$

\item[(2)] For a fixed $\sigma\in W^o_{\R^+}(M)$, let $G(\gamma):=f(\tilde d(\gamma,\sigma))$ with  $f\in C_b^1(\R)$ and
$\tilde d$ defined by \eqref{eq 2.12-1}. Then $G \in \D(\E^o_{\R^+})$ and we have
\begin{equation}\label{l2-8-2}
DG(\gamma)=f'(\tilde d(\gamma,\sigma))\cdot\Big(\sum_{k=1}^{\infty}\frac{1}{2^k}
U_s^{-1}(\gamma)\nabla_1 \tilde{\rho}(\gamma(s),\sigma(s))1_{(k-1,k]}(s)\Big)
\end{equation}
for $\dd s\times \mu^o_{\R^+}-a.s.
  (s,\gamma)\in [0,\infty)\times {\bf E}^o_{\R^+}(M),$
where $\nabla_1 \tilde \rho(\cdot,x)$ denotes the gradient with respect to the first variable of
$\tilde \rho(\cdot,\cdot)$.
\end{itemize}
\end{lem}

\begin{proof}
{\bf Step (i)}
First we suppose that 
$g\in C_{b,Lip}([0,\infty)\times M)$ and there exist a constant $L\in \R^+$ and a compact set $K\subset M$ such that
$$\aligned& |g(s,x)-g(s,y)|\leq L \rho(x,y), ~~\forall x,y\in M, s\in [0,\infty)
\\&~\text{and}~
\supp(g(s,\cdot))\subset K,~~\forall s\in [0,\infty).\endaligned$$
Consider a local coordinate system $\{U,\varphi_U\}$ on $M$, i.e. for any $x\in M$, there exists a (bounded) neighborhood $U$ of $x$ and a $C^\infty$ diffeomorphism $\varphi_U:U\rightarrow V$, where $V$ is a (bounded open) subset in $\mathbb{R}^n$. Without loss of generality, we may assume that $\textrm{supp}g\subset[0,\infty)\times K$.
According to the unit decomposition theorem on manifold, there exist $N\in \mathbb{N}^+$, $U_i\in \{U,\varphi_U\}, 1\leq i\leq N$, and non-negative smooth functions $\alpha_i, 1\leq i\leq N$ such that \begin{equation}\label{Cond}
\sum_{i=1}^N\alpha_i\Big|_{K}\equiv1,
 ~K\subset \bigcup_{i=1}^N U_i~\text{and}~\supp(\alpha_i)\subset U_i, 1\leq i\leq N.\end{equation}

Define $g_i:=g\alpha_i$ for each $1\le i \le N$. Then, from \eqref{Cond},
we know $\textrm{supp}(g_i)\subset [0,\infty)\times U_i$.
Let $V_i:=\varphi_{U_i}(U_i)\subset \R^n$ and
$\tilde{g}_i:[0,\infty)\times V_i\rightarrow\mathbb{R}$ denoted by
$\tilde{g}_i(s,y):= g_i(s,\varphi^{-1}_{U_i}(y))$ for $s\in [0,\infty), y\in V_i$. We can easily check that $\tilde{g}_i(s,\cdot)$ is Lipschitz continuous with
support contained in $V_i$ for all $s\in [0,\infty)$.

Let $\phi\in C_c^{\infty}(\R^n)$ be a polishing function satisfying that
$\supp (\phi) \subset B_1(0)$ and $\int_{\R^n}\phi(x)\dd x=1$, where $B_1(0)$ is the $0$-centered unit ball in $\R^n$.
Note that $\supp (\tilde g_i(s,\cdot))\subset V_i$, then for each $1\le i \le N$, there exists a constant $\varepsilon_i>0$ such that
for every $\varepsilon\in (0,\varepsilon_i)$, the following $\tilde g_i^{\varepsilon}(s,\cdot)$ is well defined on $V_i$,
$$\tilde{g}_i^\varepsilon(s,u):=\tilde{g}_i*\phi_\varepsilon(s,u)=
\int_{\R^d} \tilde g_i(s,v)\phi_\varepsilon(u-v)\dd v,\ \ \forall\ (s,u)\in [0,\infty)\times V_i,$$
and $\supp \tilde g_i^{\varepsilon}(s,\cdot)\subset V_i$,
where $\phi_\varepsilon(u):=\varepsilon^{-n}\phi\big(\frac{u}{\varepsilon}\big)$.
It is easy to verify
\begin{equation}\label{l2-8-3a}
\lim_{\varepsilon \downarrow 0}\sup_{v \in V_i}|\tilde g_i^{\varepsilon}(s,v)-\tilde g_i(s,v)|=0,\ \ s\in [0,\infty).
\end{equation}
Since the Lipschitz constant of $\tilde g_i(s,\cdot)$ is independent of $s$, we also have
for any $p>0$,
 \begin{equation}\label{l2-8-3}
 \begin{split}
 &\sup_{\varepsilon\in (0,\varepsilon_i), v\in V_i, s\in [0,\infty)}|\nabla \tilde{g}_i^\varepsilon(s,v)|\leq C_i,\quad 1\leq i\leq N,\\
 & \lim_{\varepsilon \downarrow 0}\int_{V_i}
 |\nabla \tilde g_i^{\varepsilon}(s,v)-\nabla \tilde g_i(s,v)|^p \dd v=0,\ \forall\ s\in [0,\infty)
 \end{split}
 \end{equation}
for some constants $C_i>0,\quad 1\leq i\leq N$, where $\nabla$ is the gradient w.r.t. the second variable.

 Define $g_i^\varepsilon:=\tilde{g}_i^\varepsilon\circ \varphi_{U_i}$, and we extend $g_i^\varepsilon$ to the whole product space $[0,\infty)\times M$ by letting $g_i^\varepsilon|_{[0,\infty)\times U_i^c}=0$. Since $\supp(\tilde{g}_i^\varepsilon)\subset [0,\infty)\times V_i$ for all $\varepsilon\in (0,\varepsilon_i)$
  implies that $\supp g_i^\varepsilon\subset [0,\infty)\times U_i$ for every $\varepsilon\in (0,\varepsilon_i)$,
  it is not difficult to see $g_i^\varepsilon\in C_b^{0,1}([0,\infty)\times M)$.
Taking $\varepsilon_0:=\inf_{1\le i \le N}\varepsilon_i$, then for every
  $\varepsilon\in (0,\varepsilon_0)$ we could define $g^\varepsilon:=\sum_{i=1}^N g_i^\varepsilon$. By \eqref{Cond}, \eqref{l2-8-3a} and
\eqref{l2-8-3} we know
for all $p>0$,
 \begin{equation}\label{eq2.15}
 \begin{split}
 &\lim_{\varepsilon \downarrow 0}\sup_{y \in M}|g^{\varepsilon}(s,y)-g(s,y)|=0,\ \ s\in [0,\infty),\\
 &\sup_{\varepsilon\in (0,\varepsilon_0), y\in M,s\in [0,\infty)}|\nabla {g}^\varepsilon(s,y)|\leq C,\\
 &\lim_{\varepsilon \downarrow 0}\int_{M}
 |\nabla g^{\varepsilon}(s,y)-\nabla g(s,y)|^p \dd y=0,\ \ s\in [0,\infty)
 \end{split}
 \end{equation}
for some constant $C>0$.

 Define $F^\varepsilon(\gamma):=f(\int_0^t g^\varepsilon(s,\gamma(s))\dd s)\in \F C_b^1$, then from \eqref{eq2.3} it is easy to obtain
  \begin{equation*}
  \begin{split}
  & DF^{\varepsilon}(\gamma)(s)=f'\left(\int_0^t g^\varepsilon(r,\gamma(r))\dd s\right)\cdot
  \Big(U_s^{-1}(\gamma)\nabla g^{\varepsilon}(s,\gamma(s))1_{(0,t]}(s)\Big),~~s\in[0,\infty).
  \end{split}
  \end{equation*}
Combining this and \eqref{eq2.15} we have
$$\sup_{\varepsilon\in (0,\varepsilon_0)}\E(F^\varepsilon,F^\varepsilon)<\infty,$$
and
$$\lim_{\varepsilon \downarrow 0}\mu^o_{\R^+}\Big(\big|F^\varepsilon(\gamma)-
F(\gamma)\big|^2\Big)=0.$$
By \cite[Chap. I Lemma 2.12]{MR92} we know that $F\in \D(\E^o_{\R^+})$.
Moreover, \eqref{eq2.15} ensures
\begin{align*}
&\lim_{\varepsilon \downarrow 0}DF^{\varepsilon}(\gamma)(s)=
f'\left(\int_0^t g(r,\gamma(r))\dd s\right)\cdot
  \Big(U_s^{-1}(\gamma)\nabla g(s,\gamma(s))1_{(0,t]}(s)\Big)
\end{align*}
for  $\dd s\times \mu^o_{\R^+}$-a.s.
  $(s,\gamma)\in [0,\infty)\times {\bf E}^o_{\R^+}(M)$.
Combining this with the dominated convergence theorem yields
\begin{equation*}
\lim_{\varepsilon \downarrow 0}
\int_{{\bf E}^o_{\R^+}(M)}\int_0^{\infty}\big|DF^{\varepsilon}(\gamma)(s)-\lim_{\varepsilon \downarrow 0}DF^{\varepsilon}(\gamma)(s)\big|^2\dd s \dd \mu^o_{\R^+}=0,
\end{equation*}
which implies \eqref{l2-8-1} immediately.

{\bf Step (ii)}
Now we consider the general case : $g\in C_{b,Lip}([0,\infty)\times M)$.
By the Greene-Wu approximation theorem in \cite{GWu}, there exists a smooth function $\eta: M \rightarrow \R^+$
such that for every $R>0$, $\{x\in M; \eta(x)\leq R\}$ is compact and $\sup_{x\in M}|\nabla \eta(x)|\leq C$. Choose $h_R:\mathbb{R}^+\rightarrow[0,1]$, $h_R\in C^\infty(\mathbb{R}^+)$ with
$$h_R(x)=1, \forall~x\in[0,R],~ h_R(x)=0, \forall~x>R+1, ~~\text{and}~~\|h_R'\|_{\infty}\leq 2.$$
For each $(s,x)\in [0,\infty)\times M$, define $g_R(s,x):=g(s,x)h_R(\eta(x)), F_R(\gamma):=f(\int_0^t g_R(s,\gamma(s))\dd s)$. Based on
the fact that  $\sup_{x\in M}|\nabla \eta(x)|\leq C$ it is easy to verify that $g_R(s,\cdot): M\rightarrow \R$
is Lipschitz continuous and with uniform compact support and with uniform Lipschitz constant.

From {\bf Step (i)} of the proof we know $F_R\in \D(\E^o_{\R^+})$ and
it is not difficult to show
\begin{align*}
\aligned
&\E^o_{\R^+}(F_R,F_R)\leq C\|f'\|^2_\infty\|\nabla g_R\|^2_\infty
\leq C\|f'\|^2_\infty(\|\nabla g\|_\infty+\|g\|_\infty)^2,\\
&\lim_{R \rightarrow \infty}\mu^o_{\R^+}\Big(\big|F_R(\gamma)-
F(\gamma)\big|^2\Big)=0,\\
&\lim_{R \rightarrow \infty}DF_R(\gamma)(s)=DF(\gamma)(s)\ \textrm{for}\ \dd s\times \mu^o_{\R^+}-a.s.
  (s,\gamma)\in [0,\infty)\times {\bf E}^o_{\R^+}(M).
  \endaligned
\end{align*}
Combining this with the same arguments as in {\bf Step (i)} we know $F \in \D(\E^o_{\R^+})$ with
$DF$ given by \eqref{l2-8-1}.

{\bf Step (iii)} By similar arguments as above we can easily check that for $F$ given as in \eqref{eq 3.1} with $g_i$ as in (1) the results in (1) follow.
Let $G_N(\gamma):=f\Big(\tilde d_N\big(\gamma,\sigma\big)\Big)$, where
$\tilde d_N\big(\gamma,\sigma\big):=\sum_{k=1}^N \frac{1}{2^k}\int_{k-1}^k
\tilde \rho\big(\gamma(s),\sigma(s)\big)\dd s$. Hence according to the conclusion in
{\bf Step (i),(ii)} we obtain $G_N\in \D(\E^o_{\R^+})$ and
\begin{equation*}
DG_N(\gamma)(s)=f'\Big(\tilde d_N\big(\gamma,\sigma\big)\Big)\cdot\Big(\sum_{k=1}^N \frac{1}{2^k} U_s^{-1}(\gamma)
\nabla_1 \tilde \rho\big(\gamma(s),\sigma(s)\big)1_{(k-1,k]}(s)\Big)
\end{equation*}
for $\dd s\times \mu^o_{\R^+}-a.s.
  (s,\gamma)\in [0,\infty)\times {\bf E}^o_{\R^+}(M).$
By this and the same arguments as  in {\bf Step (i)} (by the dominated convergence theorem) it is easy to prove
\begin{equation*}
\begin{split}
& \lim_{N \rightarrow \infty}\mu^o_{\R^+}\big(|G_N(\gamma)-G(\gamma)|^2\big)=0,\\
& \lim_{N \rightarrow \infty}\mu^o_{\R^+}\big(|DG_N(\gamma)-DG(\gamma)|^2_{\HH_+}\big)=0,
\end{split}
\end{equation*}
which implies $G \in \D(\E^o_{\R^+})$ and $DG$ has the expression \eqref{l2-8-2}.
\end{proof}

\beg{Remark}\label{Finite volume}{\bf (Finite Volume Case)} Let $\mu^o_T$ be the distribution of the Brownian motion starting from $o$ on $C([0,T];M)$.
Similar to the above argument,  we can obtain Theorems \ref{T2.1}-\ref{T2.4} and Lemma \ref{l2.7} hold with $\mu^o_{\R^+}$ be replaced by $\mu^o_{T}$. These extend the results in \cite[Section 2]{RWZZ17} to general Riemannian manifold.
\end{Remark}

\section{The case of whole line}\label{sect3}


Fix $o\in M$,
the  path space $W^{o}_{\mathbb{R}}(M)$ over $M$ is defined by
$$W^{o}_{\mathbb{R}}(M):=\{\gamma\in C(\mathbb{R};M):\gamma(0)=o\}.$$
Then $W^o_{\R}(M)$ is a separable metric space with respect to the distance $d_{\infty}$ as follows
\begin{equation}\label{e3-1}
d_{\infty}(\gamma,\sigma):=\sum_{n=1}^\infty \frac{1}{2^n}\sup_{s\in [-n,n]}\big(\tilde{\rho}(\gamma(s),\sigma(s))\big),\ \ \gamma,\sigma\in W^o_{\R}(M).
\end{equation}
where $\tilde{\rho}=\rho\wedge1$.
Similar as in Section 2, we define the following $L^1$-distance:
\begin{equation}\label{eq2.12-a}
\tilde{d}(\gamma,\eta):=
\sum_{k=1}^{\infty}\Big(\frac{1}{2^{k}}\int_{k-1}^k\tilde{\rho}(\gamma(s),\eta(s))\dd s+
\frac{1}{2^{k}}\int_{-k}^{-k+1}\tilde{\rho}(\gamma(s),\eta(s))\dd s\Big), \quad \gamma,\eta\in W_{\mathbb{R}}^{o}(M).
\end{equation}
 Obviously we have $\tilde d \le 2d_{\infty}$. Let ${\bf E}^o_{\R}(M)$ be the closure of  $W^o_{\R}(M)$
with respect to the distance $\tilde{d}$, then ${\bf E}^o_{\R}(M)$ is a Polish space.

Let $\bar{W}$ be an $n$-dimensional Brownian motion independent of $W$ and let $\bar{U}$ be the solution to \eqref{eq2.1} with $W$ replaced by $\bar{W}$. Set $\bar{x}_t:=\pi (\bar{U})$.
Then  $\bar{x}_\cdot$  is a Brownian motion with initial point $o$ on $M$ , and independent of $x$.
Define $$\hat{x}_t:=\left\{ \begin{array}{ll} x_t, \qquad &t \geq0\\
\bar{x}_{-t}, &t <0\end{array}\right..$$

Denote by $\mu^o_{\R}$ the distribution of $\hat{x}$ on $W^{o}_\mathbb{R}(M)$, then $\mu^o_{\R}$
is also a probability measure on ${\bf E}^o_{\R}(M)$ whose support is contained in
$W^{o}_{\mathbb{R}}(M)$.
Moreover, we can easily check that $\mu^o_{\R}$ is the unique probability measure such that for $F(\gamma)=f\big(\gamma
(-\bar{t}_k),...,\gamma(-\bar{t}_1),\gamma(t_1),...,\gamma(t_m)\big)$, $f\in C_b(M^{k+m})$,
 $$\aligned\int_{{\bf E}^o_{\R}(M)} F(\gamma)\dd \mu^o_{\R}=&\int_{M^{k+m}} \prod_{i=1}^kp_{\Delta_i \bar{t}}(\bar{y}_{i-1},\bar{y}_i)\prod_{i=1}^mp_{\Delta_i t}(y_{i-1},y_i)\\&f(\bar{y}_k,...,\bar{y}_1,y_1,...,y_m)\dd\bar{y}_1...\dd\bar{y}_k\dd y_1...\dd y_m,\quad y_0=\bar{y}_0=o,\endaligned$$
 where 
 $p_t$ is the heat kernel corresponding to $\frac{1}{2}\Delta$ and $-\bar{t}_k<...<-\bar{t}_1<\bar{t}_0=0=t_0<t_1<...<t_m$, $\Delta_it=t_i-t_{i-1}$ and $\Delta_i\bar{t}=\bar{t}_i-\bar{t}_{i-1}$.

 Similar to  Section 2, in order to construct Dirichlet forms associated to stochastic heat equations in Riemannian path space,
we consider
the collection $\F C_b$ of all cylinder functions on ${\bf E}^o_{\R}(M)$ as follows:
for every $F\in \F C_b$, there exist some $ ~m, k\in \mathbb{N},~ f\in C_b^{1}(\mathbb{R}^{m+k}), g_i\in C_b^{0,1}([0,\infty)\times M),\bar{g}_j\in C_b^{0,1}((-\infty,0]\times M),  T_i, \bar{T}_j\in[0,\infty)$, $i=1,...,m$, $j=1,...,k$,
such that
\begin{equation}\label{eq3.01}\aligned
F(\gamma)=f\left(\int_0^{T_1}  g_1(s,\gamma(s)) \dd s,...,\int_0^{T_m}   g_m(s,\gamma(s)) \dd s, \int_{-\bar{T}_1}^0
\bar{g}_1(s,\gamma(s)) \dd s,...,\int_{-\bar{T}_k}^0   \bar{g}_k(s,\gamma(s)) \dd s\right).
\endaligned\end{equation}
For $\gamma\in {\bf E}^o_{\R}(M)$, define $\tilde{\gamma}(s):=\gamma(s), s\geq0$ and $\bar{\gamma}(s):=\gamma(-s), s\geq0$ respectively, then $\tilde{\gamma}, \bar{\gamma}\in {\bf E}^o_{\R^+}(M)$.
Thus we could decompose $\gamma=(\tilde \gamma, \bar{\gamma})$, in particular,  under ${\mu}^o_{\R}$,  $\tilde{\gamma}(\cdot)$ and $\bar{\gamma}(\cdot)$ are two independent Brownian motions on $M$. We also define
\begin{equation}\label{eq3.01-1}
U_s(\gamma):=\left\{ \begin{array}{ll} U_s(\tilde{\gamma}), \qquad &s \geq0\\
 U_{-s}(\bar{\gamma}), &s <0,\end{array}\right.
 \end{equation}
where $U_s(\tilde{\gamma}):\R^n \rightarrow T_{\tilde \gamma(s)}M$ is
the stochastic horizontal lift along $\tilde \gamma(\cdot)$ defined via \eqref{eq2.1}.
By the above argument, for $F \in \F C_b$ with form \eqref{eq3.01} we have
\begin{equation}\label{eq3.02}\aligned
&\int_{{\bf E}^o_{\R}(M)} F(\gamma )\dd{\mu}^o_{\R}\\
&=\int_{{\bf E}^o_{\R^+}(M)} \int_{{\bf E}^o_{\R^+}(M)} f\bigg(\int_0^{T_1}  g_1(s,\tilde{\gamma}(s)) \dd s,...,\int_0^{T_m}   g_m(s,\tilde{\gamma}(s)) \dd s,\\
   &~~~~~~~~\int_{-\bar{T}_1}^0  \bar{g}_1(s,\bar{\gamma}(-s)) \dd s,...,\int_{-\bar{T}_k}^0   \bar g_k(s,\bar{\gamma}(-s)) \dd s\bigg)\dd{\mu}^o_{\R^+}(\tilde{\gamma})\dd\mu^o_{\R^+}(\bar{\gamma}),\endaligned\end{equation}
where $\mu^o_{\R^+}$ is 
introduced in Section 2.
It is easy to see that $\F C_b$ is dense in $L^2({\bf E}^o_{\R}(M);\mu^o_{\R})$.

Set:$$\HH:=L^2(\R\rightarrow \R^n;\dd s)=\left\{h: \R \rightarrow \R^n;
\int_{-\infty}^{\infty}|h(s)|^2\dd s<\infty\right\}.$$
For every $h\in \HH$ and each $F\in \F C_b$ of the form \eqref{eq3.01}, the directional derivative of $F$ with respect to $h$ is ($\mu^o_\R$-a.s.) given by
\begin{equation}\label{eq3.03}\aligned D_hF(\gamma)=&\sum_{j=1}^m\hat{\partial}_jf(\gamma)\int_0^{T_j}
 \left\langle U_s^{-1}(\gamma)\nabla g_j(s,\gamma(s)),h(s)
 \right\rangle \dd s\\&+\sum_{j=1}^{k}\hat{\partial}_{m+j}f(\gamma)\int_{-\bar{T}_j}^0
 \left\langle U_s^{-1}(\gamma)\nabla \bar{g}_j(s,\gamma(s)),h(s)
 \right\rangle \dd s,\quad \gamma\in {\bf E}^o_{\R}(M),\ h\in \HH,\endaligned\end{equation}
where
$$\hat{\partial}_jf(\gamma):=\partial_jf\left(\int_0^{T_1}  g_1(s,\gamma(s)) \dd s,...,\int_0^{T_m}   g_m(s,\gamma(s)) \dd s, \int_{-\bar{T}_1}^0  \bar{g}_1(s,\gamma(s)) \dd s,...,\int_{-\bar{T}_k}^0  \bar g_k(s,\gamma(s)) \dd s\right),$$
$U_s(\gamma)$ is defined by \eqref{eq3.01-1}, and $\nabla g_j$ denotes the gradient w.r.t. the second variable.
By the Riesz representation theorem,
there exists a gradient operator $DF(\gamma)\in \HH$ such that $\langle DF(\gamma),h\rangle_{\HH}=D_hF(\gamma)$ for every
$h\in \HH$. In particular, for the above $F$, $\gamma\in W^o_\R(M)$
\begin{equation}\label{eq2.03}\aligned
DF(\gamma)(s)=&\sum_{j=1}^m\hat{\partial}_jf(\gamma)U_s^{-1}(\gamma)\nabla g_j(s,\gamma(s))1_{(0,T_j] }(s)\\
&+\sum_{j=1}^n\hat{\partial}_{j+m}f(\gamma)U_s^{-1}(\gamma)\nabla \bar{g}_j(s,\gamma(s))1_{[-\bar{T}_j,0) }(s).
\endaligned\end{equation}

Set
$$\H^{\infty}:=\left\{h\in C_c^1(\mathbb{R};\mathbb{R}^n)\Big| h(0)=0, 
\int_{\mathbb{R}}|h'(s)|^{2}\dd s<\infty\right\}.$$
Fix a sequence of elements $\{h_k\}\subset \H^{\infty}$ such that it is an orthonormal basis in $\HH$, we define
the following symmetric quadratic form
$$\E^o_{\R}(F,G):=\frac{1}{2}\int_{{\bf E}^o_{\R}(M)}
\langle DF, DG\rangle_{\HH}\dd\mu^o_{\R}
; \quad F,G\in\F C_b.$$


\beg{Remark}  We deduce the integration by parts formula  by using the above stochastic horizontal lift $U$ below. There are other ways to define the stochastic horizontal lift such that it is adapted to the filtration generated by $\gamma$. However, as mentioned in Section 2, the $L^2$-Dirichlet form is independent of the stochastic horizontal lift, which can be seen as  a tool to obtain the integration by parts formula and the closablity of the associated bilinear form.

\end{Remark}

Set $\tilde{\beta}_{\cdot}$, $\bar{\beta}_{\cdot}$ as the anti-development of $\tilde \gamma$ and $\bar{\gamma}$ respectively
  (whose distribution under $\mu^o_{\R}$ are two independent $\R^n$-valued Brownian motions).
Let $l_{m,T}:[0,\infty)\times  W^o_{\R^+}(M)\rightarrow [0,1]$ be the vector fields constructed in
Lemma \ref{l2.5} and we define $\hat l_{m,T}:\R\times W^o_{\R}(M)\rightarrow [0,1]$ as follows,
\begin{equation}\label{eq2.05-1}
\hat l_{m,T}(t,\gamma)=
\begin{cases}
&l_{m,T}(t,\tilde \gamma),\ \ \ \ t\in [0,\infty),\\
&l_{m,T}(-t,\bar{\gamma}),\ \ \ t\in (-\infty,0).
\end{cases}
\end{equation}

\begin{prp}\label{p3.1}For each $F\in \F C_b$ and $h\in \H^\infty$, and for each $\hat l_{m,T}$ defined by \eqref{eq2.05-1}, we have
\begin{equation}\label{eq2.04}\aligned&
\int_{{\bf E}_\mathbb{R}^o(M)} \langle DF, \hat l_{m,T}h\rangle_{\HH} \dd \mu^o_{\R}=
\int_{{\bf E}^o_{\R}(M)} F\Theta_h^{m,T} \dd \mu^o_{\R},
\endaligned\end{equation}
where
\begin{equation}\label{eq2.04-1}
\begin{split}
\quad \quad \Theta_h^{m,T}(\gamma)=\Theta_h^{m,T}(\tilde \gamma, \bar \gamma)&=\int_0^{+\infty}
 \left\langle \frac{1}{2}\Ric_{U_s(\tilde{\gamma})}h_{m,T}(s,\tilde \gamma)+h_{m,T}'(s,\tilde \gamma),\dd \tilde{\beta}_s
 \right\rangle\\
&~~+\int_0^{+\infty}
 \left\langle \frac{1}{2}\Ric_{U_s(\bar{\gamma})}h_{m,T}(s,\bar \gamma)+h_{m,T}'(s,\bar \gamma),\dd \bar{\beta}_s
 \right\rangle.
\end{split}
\end{equation}
Here $h_{m,T}(s,\tilde\gamma):=h(s)l_{m,T}(s,\tilde \gamma)$ and
  $h_{m,T}(s,\bar{\gamma}):=h(-s)l_{m,T}(s,\bar{\gamma})$ for all $s\in [0,\infty)$.
\end{prp}

\begin{proof}By \eqref{eq3.02}, \eqref{eq3.03} we have
\begin{equation}\label{eq2.05}
\begin{split}
&\int_{{\bf E}_\mathbb{R}^o(M)} \langle DF,\hat l_{m,T}h\rangle_{\HH}\dd\mu^o_{\R}\\
&=\int_{{\bf E}^{o}_{\mathbb{R}^+}(M)}\int_{{\bf E}^{o}_{\mathbb{R}^+}(M)}\sum_{j=1}^m\hat{\partial}_jf(\tilde{\gamma},\bar{\gamma})\int_0^{T_j}
 \left\langle U_s^{-1}(\tilde{\gamma})\nabla g_j(s,\tilde{\gamma}(s)),l_{m,T}(s,\tilde \gamma)h(s)
 \right\rangle \dd s\dd{\mu}^o_{\R^+}(\tilde{\gamma})\dd\mu^o_{\R^+}(\bar{\gamma})\\
 &+\int_{{\bf E}^{o}_{\mathbb{R}^+}(M)} \int_{{\bf E}^{o}_{\mathbb{R}^+}(M)}\sum_{j=1}^{k}
 \hat{\partial}_{m+j}f(\tilde{\gamma},\bar{\gamma})\int_0^{\bar{T}_j}
 \left\langle U_{s}^{-1}(\bar{\gamma})\nabla \bar{g}_j(-s,\bar{\gamma}(s)),l_{m,T}(s,\bar{\gamma})h(-s)
 \right\rangle \dd s\dd{\mu}^o_{\R^+}(\tilde{\gamma})\dd\mu^o_{\R^+}(\bar{\gamma})\\&:=I+II,
 \end{split}
 \end{equation}
where $\hat{\partial}_jf(\tilde{\gamma},\bar{\gamma}):=\hat{\partial}_jf({\gamma})$.

According to \eqref{eq2.8} of Lemma \ref{l2.7}, we get
$$I=\int_{{\bf E}^{o}_{\mathbb{R}^+}(M)} \int_{{\bf E}^{o}_{\mathbb{R}^+}(M)}F(\tilde{\gamma},\bar{\gamma})\int_0^{+\infty}
 \left\langle \frac{1}{2}\Ric_{U_s(\tilde{\gamma})}h_{m,T}(s,\tilde \gamma)+h_{m,T}'(s,\tilde \gamma),\dd \tilde{\beta}_s
 \right\rangle \dd{\mu}^o_{\R^+}(\tilde{\gamma})\dd\mu^o_{\R^+}(\bar{\gamma}),$$
 $$II=\int_{{\bf E}^{o}_{\mathbb{R}^+}(M)} \int_{{\bf E}^{o}_{\mathbb{R}^+}(M)}F(\tilde{\gamma},\bar{\gamma})\int_0^{+\infty}
 \left\langle \frac{1}{2}\Ric_{U_s(\bar{\gamma})}h_{m,T}(s,\bar \gamma)+h_{m,T}'(s,\bar \gamma),\dd \bar{\beta}_s
 \right\rangle \dd {\mu}^o_{\R^+}(\tilde{\gamma})\dd {\mu}^o_{\R^+}(\bar{\gamma}),$$
where $F(\tilde{\gamma},\bar{\gamma})=F(\gamma)$.  Combining this and \eqref{eq2.05}, we finish the proof.

\end{proof}

Similar to the arguments as in the proof of Theorem \ref{T2.1} and based on the above integration by parts formula \eqref{eq2.04}, we  obtain the following:

\beg{thm}\label{T 3.1}  The quadratic form $(\E^o_{\R}, \F C_b)$
is closable and its closure $(\E^o_{\R},\D(\E^o_{\R}))$ is a quasi-regular Dirichlet form on $L^2({\bf E}^o_{\R}(M);\mu^o_{\R})$.
\end{thm}

\begin{proof}
{\bf$(a)$ Closablity:}
$(I)$
Suppose that $\{F_k\}_{k=1}^{\infty}\subseteq \F C_b$ is a sequence of cylinder functions with
\begin{equation}\label{eq2.06}
\begin{split}
\lim_{m \rightarrow \infty}\mu^o_{\R}\left( F_m^2\right)=0,\ \
\lim_{k,m \rightarrow \infty}\E^o_{\R}\left(F_k-F_m,F_k-F_m\right)=0.
\end{split}
\end{equation}
Thus $\{D F_m\}_{m=1}^{\infty}$ is a Cauchy sequence in
$L^2\left({\bf E}^o_{\R}(M)\rightarrow \HH;\mu^o_{\R}\right)$ for which there exists a limit $\Phi$. It suffices to prove that $\Phi=0$.
Given an orthonormal basis $\{h_k\}_{k=1}^{\infty}\subset  C^\infty_c(\mathbb{R};\mathbb{R}^n)\cap \H^\infty$ of $\HH$, by the integration by parts formula \eqref{eq2.04},
for every $G \in \F C_b, h_k$ and $k,i,m,T\in \mathbb{N}^+$  we have
\begin{equation}\label{eq2.07}
\begin{split}
\mu^o_{\R}\left(\langle D F_i, \hat l_{m,T}h_k\rangle_{\HH}G \right)&=\mu^o_{\R}\left(\langle D \left(F_i G\right),\hat l_{m,T}h_k\rangle_{\HH}\right)
-\mu^o_{\R}\left(\langle D G, \hat l_{m,T}h_k\rangle_{\HH}F_i\right)\\
&=\mu_{\R}^o\left(F_i G \Theta_{h_k}^{m,T}\right)-\mu^o_{\R}\left(\langle D G, \hat l_{m,T}h_k\rangle_{\HH}F_i\right).
\end{split}
\end{equation}
Since $G$ and $D G$ are bounded and $\Theta_{h_k}^{m,T}\in L^2({\bf E}_{\R}^o(M);\mu^o_{\R})$
(due to \eqref{eq2.5} and the fact $h_k\in C_c^1(\R;\R^d)$), by \eqref{eq2.06} we could take the limit
$i \rightarrow \infty$ under the integral in \eqref{eq2.07} to conclude
\begin{equation*}
\begin{split}
&\mu^o_{\R}\left(\langle \Phi, \hat l_{m,T}h_k\rangle_{\HH}G\right)=0,\quad \forall\
G \in \F C_b,\ k,m,T\in \mathbb{N}^+,
\end{split}
\end{equation*}
therefore we could find a $\mu^o_{\R}$-null set $\Delta_k \subset W_{\R}^o(M)$, such that
\begin{equation}\label{eq2.12-1}
\langle \Phi(\gamma), \hat l_{m,T}(\gamma)h_k\rangle_{\HH}=0,\quad \ \forall\ m,T\in \mathbb{Z}_+,\ \gamma \notin \Delta_k.
\end{equation}

For a fixed $h_k$, we could find a $T_k\in \mathbb{N}^+$ (which may depend on $h_k$) satisfying
$\supp (h_k)\subset [-T_k,T_k]$. Since the coordinate process $\gamma(\cdot)$ is non-explosive,
for every $\gamma \notin \Delta_0$ with some $\mu^o_{\R}$-null set $\Delta_0$, there exists $m_k(\gamma)\in \mathbb{Z}_+$, such that
$\tilde \gamma(t)\in D_{m_k-1}$ and $\bar \gamma(t)\in D_{m_k-1}$ for all $t \in [0,T_k]$,
hence $\hat l_{m_k,T_k}(t,\gamma)=1$ for all $t \in [-T_k,T_k]$. Here $D_{m_k-1}$ is defined in \eqref{cut}.
Combining this with \eqref{eq2.12-1} we know
$$\langle \Phi(\gamma), h_k\rangle_{\HH}=0,\quad  k\ge 1,
\gamma \notin \Delta_0\cup\Delta_k,$$
which implies that
$\Phi(\gamma)=0$, $\forall\ \gamma \notin \Delta:=\cup_{k=0}^{\infty}\Delta_k$. So $\Phi=0$, a.s., and
$(\E_{\R^+}^o,\F C_b)$ is closable.
By standard procedure, it is not difficult to show that its closure $(\E^o_{\R},\D(\E^o_{\R}))$ is a Dirichlet form.


{\bf$(b)$  Quasi-Regularity:}

In order to prove the quasi-regularity, we need to verify conditions (i)-(iii) in \cite[Definition IV-3.1]{MR92}.
By the same arguments as in the proof of Theorem \ref{T2.1}, we could check (ii) and (iii) of
\cite[Definition IV-3.1]{MR92} for $(\E^o_{\R},\D(\E^o_{\R}))$, so we omit the proof here.


Since the metric space $({\bf E}^o_{\R}(M); \tilde d)$ is separable,
we can choose a fixed countable dense subset $\{\xi_m|m\in\mathbb{N}^+\}\subset W^o_{\R}(M)$.
Let $\varphi\in C_b^\infty(\mathbb{R})$ be an increasing function  satisfying with
$$\varphi(t)=t,\quad \forall~t\in[-1,1]~~\text{and}~~\|\varphi'\|_{\infty}\leq 1.$$

For each $m\geq1$, the function $v_m:{\bf E}^o_{\R}(M)\rightarrow\mathbb{R}$ is given by
$$v_m(\gamma)=\varphi(\tilde{d}(\gamma,\xi_m)),\quad\gamma\in {\bf E}^o_{\R}(M)
,$$
where $\tilde d$ is defined by \eqref{eq2.12-a}.
According to the same procedures as in the proof of Lemma \ref{l2.8} we have $v_m\in \D(\E^o_{\R})$ and
\begin{equation*}
\begin{split}
Dv_m(\gamma)(s)=\varphi'(\tilde d(\gamma,\xi_m))\cdot\Bigg(&\sum_{k=1}^{\infty}\frac{1}{2^k}\Big(
 U_s^{-1}(\tilde \gamma)\nabla_1 \tilde{\rho}(\tilde \gamma(s),\xi_m(s))1_{(k-1,k]}(s)\\
&+U_{-s}^{-1}(\bar \gamma)\nabla_1 \tilde{\rho}(\bar \gamma(-s),\xi_m(s))1_{[-k,-k+1)}(s)\Big)\Bigg)
\end{split}
\end{equation*}
for $\dd s\times \mu^o_{\R}-a.s.
  (s,\gamma)\in \R\times {\bf E}^o_{\R}(M),$
where $\nabla_1 \tilde \rho(\cdot,x)$ denotes the gradient with respect to the first variable of
$\tilde \rho(\cdot,\cdot)$.
By such expression we arrive at
\begin{equation*}
\sup_{m\ge 1}\E^o_{\R}(v_m,v_m)<\infty.
\end{equation*}
Then based on this and repeating the arguments as in the proof of Theorem \ref{T2.1} we can show
\begin{equation}\label{eq2.13}w_k:=\inf_{m\leq k}v_m, k\in \mathbb{N}^+, \textrm{ converges } \E^o_{\R} -\text{quasi-uniformly to zero on} ~{\bf E}^o_{\R}(M),
\end{equation}
therefore the capacity associated with $(\E^o_{\R},\D(\E^o_{\R}))$ is tight. So (i) of
\cite[Definition IV-3.1]{MR92} holds. By now we have finished the proof.

\end{proof}

\begin{Remark}By the theory of the Dirichlet form,  for  the case of the whole line, we also derive similarly Theorems \ref{T2.2} and \ref{T2.4} in Section 2.
\end{Remark}


As explained in the introduction, the invariant measure for the stochastic heat equation on the whole line could be the distribution of a two-sided Brownian motion with a  shift given by Lebesgue measure,
which may not be finite measure. 
 So in our setting it is also natural to consider the reference measure given by $\int_M \mu^{x}_{\R}(\dd \gamma)
\nu(\dd x)$ with some Randon measure $\nu$
(which may not be finite measure).  The support of the measure is the paths on $M$ with initial point not fixed.

Let $W_{\mathbb{R}}(M):=C(\mathbb{R};M)$ be the free path space, then $(W_{\R}(M),d_{\infty})$ is also a
separable metric space with $d_{\infty}$ defined by \eqref{e3-1}. Let $\tilde d$ be the $L^1$-distance defined
by \eqref{eq2.12-a}, and let ${\bf E}_{\R}(M)$ be the closure of $W_{\R}(M)$ under $\tilde d$. It is easy to see that
${\bf E}_{\R}(M)$ is a Polish space.

For any fixed Radon measure $\nu$ (not necessarily finite) on $M$, we could introduce a
measure (not necessarily finite) $\mu^{\nu}_{\R}(\dd \gamma):=\int_M \mu^{x}_{\R}(\dd \gamma)
\nu(\dd x)$ on ${\bf E}_{\R}(M)$, where $\mu^x_{\R}$ is the probability measure defined as $\mu^o_{\R}$ with $o$ replaced by $x$. Then we  have that for  $F(\gamma)=f\big(\gamma
(-\bar{t}_k),...,\gamma(-\bar{t}_1),\gamma(t_0),\gamma(t_1),...,\gamma(t_m)\big)$ with $f\in C_c(M^{k+m+1})$,  it holds
 \begin{equation}\label{e 3-1}
 \begin{split}
 \int_{{\bf E}_{\R}(M)} F(\gamma)\dd \mu^{\nu}_{\R}
 =&\int \prod_{i=1}^kp_{\Delta_i \bar{t}}(\bar{y}_{i-1},\bar{y}_i)\prod_{i=1}^mp_{\Delta_i t}(y_{i-1},y_i)\\&f(\bar{y}_k,...,\bar{y}_1,y_0,y_1,...,y_m)\dd\bar{y}_1...\dd\bar{y}_k\dd y_1...\dd y_m \nu(\dd y_0),
\end{split}
\end{equation}
where the variable $y_0=\bar{y}_0$ and $p_t$ is the heat kernel corresponding to $\frac{1}{2}\Delta$ and $-\bar{t}_k<...<-\bar{t}_1<\bar{t}_0=0=t_0<t_1<...<t_m$, $\Delta_it=t_i-t_{i-1}$ and $\Delta_i\bar{t}=\bar{t}_i-\bar{t}_{i-1}$. Here $C_c(M^{k+m+1})$ denote continuous  functions on $M^{k+m+1}$ with compact support.
\begin{Remark}
When $M$ is compact and $\nu$ is the normalized volume measure, then 
$\mu^{\nu}_{\R}$ corresponds to  the distribution
of stationary $M$-valued Brownian motion. 
In the case that $\nu$ is given by the volume measure, the Markov process we construct below corresponds to stochastic heat equation on $\mathbb{R}$ with values in $M$ without any boundary conditions.
\end{Remark}

\begin{Remark}
If $\nu$ is the volume measure ($M$ could be either compact or non-compact), then by expression \eqref{e 3-1} we know that
$\theta_s^{\sharp}\mu_{\R}^\nu=\mu_{\R}^\nu$ for any $s\in \R$, where $\theta_s^{\sharp}\mu_{\R}^\nu$
denotes the push forward measure for $\mu_{\R}^\nu$ by the map $\theta_s:{\bf E}_{\R}(M) \rightarrow  {\bf E}_{\R}(M)$
as $\theta_s(\gamma)(t):=\gamma(t+s)$. This means that $\mu_{\R}^\nu$ is invariant under any translation on $\R$.
\end{Remark}

\begin{Remark}\label{l o}
In \cite{BGHZ19, Hai16}, the authors studied \eqref{eq1.1} with  solutions taking values in free loop space
${\bf L}(M):=\{\gamma\in C([0,1];M); \gamma(0)=\gamma(1)\}$. In this case we could also construct the $\mathbf{L^2}$-Dirichlet form $(\tilde \E_{\R}^\nu, \mathscr{D}(\tilde \E_{R}^\nu))$ as follows
\begin{equation}\label{e 3-2}
\begin{split}
\tilde \E_{\R}^\nu(F,F)&:=\frac{1}{2}\int_{{\bf L}(M)}\langle DF,DF\rangle_{\HH}\tilde \mu^{\nu}(\dd \gamma)\\
&:=\frac{1}{2}\int_M\int_{{\bf L}_x(M)}\langle DF,DF\rangle_{\HH}\tilde \mu^x(\dd\gamma)\nu(\dd x),
\end{split}
\end{equation}
where ${\bf L}_x(M):=\{\gamma\in C([0,1];M); \gamma(0)=\gamma(1)=x\}$, $\tilde \mu^x$ denotes the Brownian bridge
measure on ${\bf L}_x(M)$.

For the free loop measure $\tilde \mu^{\nu}$ above, if $\nu(\dd x)=p_1(x,x)\dd x$, then $\tilde \mu^\nu$ is invariant under
any rotation on $S^1$. In \cite{RWZZ17}, the quasi-regularity of $(\tilde \E_{\R}^\nu, \mathscr{D}(\tilde \E_{R}^\nu))$ has been
studied under the assumption that $M$ is compact. As explained in Remarks \ref{r4-3} and \ref{loop}, by the method of this paper, we could also
obtain the corresponding results for the case that $M$ is non-compact.

The state space ${\bf L}(M)$ corresponds to the spatial variable with values in finite volume, while ${\bf E}_{\R}(M)$ and ${\bf E}_{\R}^o(M)$
correspond to the case that the spatial variable in infinite volume.
\end{Remark}

Here we only consider the case that $\nu$ is an infinite measure,  since when $\nu$ is a finite measure, the case is simpler and it may be handled similar as in Theorem \ref{T3.1}.

Next, we assume that $\nu$ is infinite,  then $\mu^{\nu}_{\R}$ is also an infinite measure on ${\bf E}_{\R}(M)$ with support contained in $W_{\R}(M)$.
In this case $1\notin L^2(\mu^\nu_\R)$ and
 we need to introduce a new class of cut-off functions ${\bf E}_{\R}(M)$. Let $\F C_{Lip}$ be the space of bounded Lipschitz continuous functions on ${\bf E}_{\R}(M)$, i.e. for every $F\in \F C_{Lip}$, there exist some $ ~m, k\in \mathbb{N},~ f\in C_{b}^{1}(\mathbb{R}^{m+k}), g_i\in C_{Lip}^{0,1}([0,\infty)\times M),\bar{g}_j\in C_{Lip}^{0,1}((-\infty,0]\times M),  T_i, \bar{T}_j\in[0,\infty)$, $i=1,...,m$, $j=1,...,k$,
such that
\begin{equation}\label{e4-1}
\aligned
F(\gamma)=f\left(\int_0^{T_1}  g_1(s,\gamma(s)) \dd s,...,\int_0^{T_m}   g_m(s,\gamma(s)) \dd s, \int_{-\bar{T}_1}^0
\bar{g}_1(s,\gamma(s)) \dd s,...,\int_{-\bar{T}_k}^0   \bar g_k(s,\gamma(s)) \dd s\right),
\endaligned
\end{equation}
where $C_{Lip}^{0,1}([0,\infty)\times M)$ denotes the collection of functions
$g:[0,\infty)\times M\rightarrow \R$ such that $g$ is continuous on $[0,\infty)$ and Lipschitz
continuous (not necessarily bounded) on $M$ with the associated Lipschitz constants independent of $s\in [0,\infty)$.
Now we fix a point $o\in M$.
Let
\begin{equation*}
\begin{split}
\F C_{c}:=&\Big\{F\in \F C_{Lip}; \text{there\ exists}\ R>0\ \text{such\ that}\
F(\gamma)=0\ \text{for\ all}\\
&\ \gamma\in {\bf E}_{\R}(M)\ \text{satisfying}\ \int_0^1 \rho\big(o,\gamma(s)\big)\dd s>R\ \Big\}.
\end{split}
\end{equation*}


\begin{lem}\label{l4-1}
Suppose that for every $R>0$, 
 it holds
\begin{equation}\label{l4-1-1}
\int_M \mu_{\R^+}^x\Big(\sup_{s\in [0,1]}\rho\big(x,\gamma(s)\big)>\rho(o,x)-R\Big)\nu(\dd x)<\infty,
\end{equation}
then $\F C_c$ is a dense subset of $L^2({\bf E}_{\R}(M); \mu^\nu_{\R})$.
\end{lem}
\begin{proof}
{\bf Step (i)} We first show $\F C_c \subset L^2({\bf E}_{\R}(M); \mu^\nu_{\R})$. For every $F \in \F C_c$,
without loss of generality we may assume that there exist
$R>0$, such that $F(\gamma)=0$ for all $\gamma\in {\bf E}_{\R}(M)$ satisfying
$\int_0^1 \rho\big(o,\gamma(s)\big)\dd s>R$.
Then we have
\begin{align*}
& \quad \int_{{\bf E}_{\R}(M)}|F(\gamma)|^2\mu_{\R}^{\nu}(\dd \gamma)=
\int_M\int_{{\bf E}_{\R}^x(M)}|F(\gamma)|^2\mu_{\R}^{x}(\dd \gamma)\nu(\dd x)\\
&=\int_{B(o,2R)}\int_{{\bf E}_{\R}^x(M)}|F(\gamma)|^2\mu_{\R}^{x}(\dd \gamma)\nu(\dd x)
+\int_{B(o,2R)^c}\int_{{\bf E}_{\R}^x(M)}|F(\gamma)|^2\mu_{\R}^{x}(\dd \gamma)\nu(\dd x)\\
&\le \|F\|_{\infty}^2\Big(\nu\big(B(o,2R)\big)+\int_{B(o,2R)^c}\mu^x_{\R^+}\Big(\sup_{s\in [0,1]}\rho(x,\gamma(s))>
\rho(o,x)-R\Big)\Big)\nu(\dd x)<\infty,\\
\end{align*}
where the third step is due to the fact when $x\notin B(o,2R)$, $F(\gamma)=0$ for all $\gamma \in {\bf E}_{\R}^x(M)$ with
$\sup_{s\in [0,1]}\rho(\gamma(s),x)\le \rho(o,x)-R$ (if $\sup_{s\in [0,1]}\rho(\gamma(s),x)\le \rho(o,x)-R$, then
$\int_0^1\rho\big(o,\gamma(s)\big)\dd s\ge \inf_{s\in [0,1]}\rho(\gamma(s),o)>R$,
hence $F(\gamma)=0$).

{\bf Step (ii)} Now we are going to show $\F C_c$ is dense in $L^2({\bf E}_{\R}(M); \mu^\nu_{\R})$. It suffices to
prove that for every $G(\gamma):=f\Big(\gamma(t_1),\cdots ,\gamma(t_m)\Big)$ with some $m\in \mathbb{N}^+$,
$t_1<t_2\cdots <t_m$ and $f\in C_c^1(M^m)$, there exists a sequence $\{G_{k,R}\}_{k,R}\subset
\F C_c$ such that $\lim_{k, R \rightarrow \infty}
\mu_{\R}^{\nu}\big(|G_{k,R}-G|^2\big)=0$. Here $C_c^1(M^m)$ denotes the $C^1$ functions on $M^m$ with compact support.

By Nash isometric imbedding theorem, there is a smooth isometric imbedding $\eta:M \rightarrow \mathbb{R}^N$ with some
$N \in \mathbb{N}^+$ and we can extend $f\in C_c^1(M^m)$ to  $\tilde f \in C_c^1(\R^{Nm})$ satisfying
$\tilde f\big(\eta(x)\big)=f(x)$ for all $x\in M$. Choose $\varphi_R\in C_b^1(\mathbb{R},\mathbb{R})$,
$\phi_R \in C_c^1(\mathbb{R},\mathbb{R})$ satisfying
\begin{equation*}
\varphi_R(x)=
\begin{cases}
& x,\ \ \ \ \ \ \ \ \text{if}\ |x|\le R,\\
& R+1,\ \ \ \ \ \ \ \text{if}\ x>R+1,\\
& -R-1,\ \ \ \ \ \ \ \ \text{if}\ x<-R-1,
\end{cases}
\end{equation*}
\begin{equation*}
\phi_R(x)=
\begin{cases}
& 1,\ \ \ \ \ \ \ \ \text{if}\ |x|\le R,\\
& \in (0,1),\ \ \ \ \ \ \ \text{if}\ R<|x|<R+1,\\
& 0,\ \ \ \ \ \ \ \ \text{if}\ |x|>R+1.
\end{cases}
\end{equation*}

 We set $\varphi_{R,N}(x):=\prod_{i=1}^N\varphi_R(x_i)$ for $x=(x_1,...,x_N)$.
$$G_{k,R}(\gamma):=
\phi_R\Big(\int_0^1\rho\big(o,\gamma(s)\big)\dd s\Big)
\tilde f\Big(k\int_{t_1}^{t_1+\frac{1}{k}}\varphi_{R,N}\circ\eta\big(\gamma(s)\big)\dd s,\cdots,
k\int_{t_m}^{t_m+\frac{1}{k}}\varphi_{R,N}\circ\eta\big(\gamma(s)\big)\dd s\Big),$$
then it is easy to verify that $G_{k,R} \in \F C_c$ for all $k>0$ and $R$ large enough,
and  $\lim_{k, R \rightarrow \infty}
\mu_{\R}^{\nu}\big(|G_{k,R}-G|^2\big)=0$ (
(since $\tilde f \in C_c^1(\R^{Nm})$, this could be shown by the dominated convergence theoem). By now we have finished the proof.

\end{proof}

Now we give some sufficient conditions on the curvature of $M$ for \eqref{l4-1-1}.

\begin{lem}\label{r4-1}
Suppose that
\begin{equation}\label{r4-1-1}
{\rm Ric}_x(X,X)\ge -C_1\big(1+\rho(o,x)^{\alpha}\big),\ \ \forall\ x\in M,\ X\in T_xM, |X|=1,
\end{equation}
for some $C_1>0$, $\alpha\in (0,2)$ and $o\in M$, where ${\rm Ric}_x$
denotes the Ricci curvature operator at $x\in M$. Then for every Radon measure $\nu(\dd x)=\nu(x)\dd x$
(here $\dd x$ denotes the volume measure on $M$) such that
\begin{equation}\label{r4-1-2a}
|\nu(x)|\le C_2\exp(C_3\rho(o,x)^{\beta}),\ \ \forall x\in M
\end{equation}
with some $C_2,C_3>0$ and $\beta\in (0,2)$, \eqref{l4-1-1} holds.
\end{lem}
\begin{proof}
Note that \eqref{r4-1-1} implies that
\begin{equation}\label{r4-1-1a}
{\rm Ric}_y(Y,Y)\ge -K_1\big(\rho(o,y)\big),\ \ \forall\ y\in M,\ Y\in T_yM,\ |Y|=1,
\end{equation}
with $K_1(r):=C_1(1+r^{\alpha})$.
It is easy to verify that we could find a $c_1>0$  such that
\begin{equation}\label{r4-1-2b}
c_2:=\sup_{t>0}\big(t\sqrt{(n-1)K_1(t)}-2c_1t^2\big)<\infty.
\end{equation}
Then according to \cite[Lemma 2.2]{W04} we know that
for every $N>0$ and $T>0$,
\begin{equation}\label{r4-1-3a}
\mu_{\R^+}^o\Big(\sup_{s\in [0,T]}\rho\big(o,\gamma(s)\big)>N\Big)\le e^{n+c_2-\kappa(T)N^2},
\end{equation}
where $\kappa(T):=\frac{1}{2T}e^{-1-2c_1T}$.

Also note that \eqref{r4-1-1a} implies for every $x\in M$,
\begin{align*}
{\rm Ric}_y(Y,Y)\ge -2K_1\big(\rho(x,y)\big)-2K_1\big(\rho(o,x)\big),\ \ \forall\ y\in M,\ Y\in T_yM,
\end{align*}
Then taking $c_1=1$ and using $2K_1(t)+2K_1\big(\rho(o,x)\big)$ to replace $K_1(t)$ in \eqref{r4-1-2b},
we have $c_2\le c_3(1+\rho(o,x)^\alpha)$. Therefore according to
\eqref{r4-1-3a} we know for all $R>0$ and $x\notin B(o,2R)$,
\begin{equation}\label{r4-1-3}
\begin{split}
&\mu_{\R^+}^x\Big(\sup_{s\in [0,1]}\rho\big(x,\gamma(s)\big)>\rho(o,x)-R\Big)\\
&\le \exp\Big(n+c_3(1+\rho(o,x)^\alpha)-\kappa(1)(\rho(o,x)-R)^2\Big)\\
&\le \exp\Big(n+c_3(1+\rho(o,x)^\alpha)-\frac{\kappa(1)}{4}\rho(o,x)^2\Big)\\
&\le c_4e^{-c_5\rho(o,x)^2},
\end{split}
\end{equation}
where $c_4$, $c_5$ are positive constants independent of $x\in M$ and $R>0$.
Denote by ${\rm Cut}(o)$ the cut-locus of $o$ in $M$ and the exponential map from $o\in M$ by ${\rm exp}_o:T_oM \rightarrow M$.
It is well known that $\exp_o^{-1}: M\setminus {\rm Cut}(o) \rightarrow
\exp_o^{-1}\big(M\setminus {\rm Cut}(o)\big) \subset T_o M \simeq \R^+\times \mathbb{S}^{n-1}$ is a diffeomorphism, which induces
the geodesic spherical coordinates of $M$ (see e.g. \cite[Section III.1]{Cha}
for details). Let $(r,\theta)\in \R^+\times \mathbb{S}^{n-1}$ be the element in geodesic spherical coordinates,
then for every $f\in C_c(M)$ we have (see e.g. \cite[Theorem III 3.1]{Cha})
\begin{align*}
\int_M f(x)dx=\int_{\R^+}\int_{\mathbb{S}^{n-1}}f\big((r,\theta)\big)\big||\mathscr{A}(r,\theta)|\big|\dd r \dd\theta,
\end{align*}
where $\mathscr{A}(r,\theta)$ is a $n\times n$ matrix, $|\mathscr{A}|$ denotes the determinant of
$\mathscr{A}$, and $\mathscr{A}$ satisfying the following equation
\begin{align*}
\mathscr{A}^{''}(t,\theta)+\mathscr{R}(t,\theta)\mathscr{A}(t,\theta)=0,\ \ \mathscr{A}(0,\theta)=0,\
\mathscr{A}'(0,\theta)=\mathbf{I}.
\end{align*}
Here $\mathscr{R}(t,\theta)\in L(\R^n;\R^n)\simeq \R^{n\times n}$ and
$\mathscr{R}(t,\theta)\xi:=U_t^{-1}{\rm R}\big(\gamma'_\theta(t),U_t\xi\big)\gamma'_\theta(t)$ for all $\xi\in \R^n$
with $\gamma_\theta(t)=\exp_o\big((t,\theta)\big)$, $U_t:\R^n \rightarrow T_{\gamma_\theta(t)}M$ is the parallel translation
along geodesic $\gamma_{\theta}(\cdot)$, ${\rm R}$ denotes the Riemannian curvature operator on $M$.

Moreover, we have the following estimates for $|\mathscr{A}|$ (see e.g. \cite[Theorem III 4.3]{Cha}),
\begin{equation}\label{r4-1-4}
\big||\mathscr{A}(r,\theta)|\big|\le \Big(\sqrt{\frac{n-1}{K_1(r)}}{\rm sinh}
\Big(\sqrt{\frac{K_1(r)}{n-1}}r\Big)\Big)^{n-1}\le c_6e^{c_7r^{1+\alpha/2}},\ r>0,
\end{equation}
where $c_6,c_7$ are positive constants independent of $r$, $K_1(r)$ is the function in \eqref{r4-1-1a} and
the last step is due to $\frac{\sinh a}{a}\leq \cosh a$ and \eqref{r4-1-1}.

Combining  \eqref{r4-1-2a}, \eqref{r4-1-3} and \eqref{r4-1-4} yields
\begin{equation}\label{r4-1-5}
\begin{split}
&\int_{B(o,2R)^c}\mu_{\R^+}^x\Big(\sup_{s\in [0,T]}\rho\big(x,\gamma(s)\big)>\rho(o,x)-R\Big)\nu(\dd x)\\
&\le c_8\int_{2R}^\infty\int_{\mathbb{S}^{n-1}} \exp\big(C_3r^\beta+c_7r^{1+\alpha/2}-c_5r^2\big)\dd\theta \dd r\\
&\le c_9\int_{2R}^\infty e^{-c_{10}r^2}\dd r\le c_{11}e^{-c_{12}R^2}.
\end{split}
\end{equation}
Here in the second step of inequality we have applied the fact $\alpha\in (0,2)$ and $\beta\in (0,2)$.

Based on this estimate we could obtain \eqref{l4-1-1} immediately.
\end{proof}


\begin{Remark}
By Lemma \ref{r4-1} we know that under curvature condition \eqref{r4-1-1}, the property
\eqref{l4-1-1} holds if $\nu$ is the volume measure of $M$.
\end{Remark}

\begin{Remark}\label{r4-3}
For the free loop measure $\tilde \mu^{\nu}$ defined by \eqref{e 3-2} on ${\bf L}(M)$, by carefully tracking the proof
of Lemma \ref{l4-1} we know if for every $R>0$,
\begin{equation}\label{r4-3-1}
\int_{M}\tilde \mu^x\Big(\sup_{s\in [0,\frac{1}{2}]}\rho\big(x,\gamma(s)\big)>\rho(o,x)-R\Big)\nu(\dd x)<\infty,
\end{equation}
then $\F C_c$ is dense in $L^2({\bf L}(M); \tilde \mu^\nu)$.

When we choose $\nu(\dd x)= p_1(x,x)\dd x$, it holds.
\begin{equation}\label{est}\aligned
&\int_{M}\tilde \mu^x\Big(\sup_{s\in [0,\frac{1}{2}]}\rho\big(x,\gamma(s)\big)>\rho(o,x)-R\Big)\nu(\dd x)\\
&=\int_{M}\mu^x_{\R^+}\Big(1_{\{\sup_{s\in [0,\frac{1}{2}]}\rho\big(x,\gamma(s)\big)>\rho(o,x)-R\}}(\gamma)
p_{\frac{1}{2}}\big(\gamma(\frac{1}{2}),x\big)\Big)\dd x\\
&\le \int_{M} \sqrt{\mu^x_{\R^+}\Big(\sup_{s\in [0,\frac{1}{2}]}\rho\big(x,\gamma(s)\big)>\rho(o,x)-R\Big)}
\sqrt{\int_{M}p_{\frac{1}{2}}(x,y)^3 \dd y}\dd x,\endaligned
\end{equation}
where the last step is due to Cauchy-Schwartz inequality.

If the curvature condition \eqref{r4-1-1} holds, then we know that \eqref{r4-1-3} is true. Moreover, suppose \eqref{r4-1-1} holds and
the following lower bound of volume \eqref{r4-3-2} is satisfied (, which could be viewed as an local volume non-collapsed condition)
\begin{equation}\label{r4-3-2}
\inf_{x; \rho(o,x)\le R}m\Big(B(x,\frac{1}{2})\Big)\ge C_1e^{-C_2 R^\beta},\ \forall R>1,
\end{equation}
where $\beta\in (0,2), C_1, C_2$ are  constants, $m$ denotes the volume measure on $M$ and $B(x,r):=\{y\in M; \rho(y,x)\le r\}$ is the geodesic ball
on $M$. Then according to the proof of  \cite[Corollary 3]{ATW06} (here our curvature condition \eqref{r4-1-1} is a little different from \cite{ATW06}), we have
\begin{align*}
p_{\frac{1}{2}}(x,y)\le C_3\exp\Big(-C_4\rho(x,y)^2+C_5\big(\rho(o,x)^{\max(\alpha,\beta)}+\rho(o,y)^{\max(\alpha,\beta)}\big)\Big).
\end{align*}
Putting this into \eqref{est} and by the same arguments as in the proof of Lemma \ref{r4-1} we could prove \eqref{r4-3-1}.

As a result, combining all the above estimates we could prove that if \eqref{r4-1-1} and \eqref{r4-3-2} are true, then
$\F C_c$ is dense in $L^2({\bf L}(M); \tilde \mu^\nu)$.
\end{Remark}

For  $F \in \F C_c$, we still define the directional derivative $D_h F(\gamma)$
along  $h \in \HH:=L^2(\R\rightarrow \R^n;\dd s)$ and the gradient operator $DF\in \HH$ as in \eqref{eq3.03} and \eqref{eq2.03}, respectively. Here as explained before  Lemma \ref{l2.8} we know that $D_hF$ and $DF$ are well-defined for $\mu^\nu_\R$-a.e. $\gamma$.

Now for the fixed $o\in M$, as in Lemma \ref{l2.5} (although here the initial point will not be fixed, see e.g. \cite{TW98} or \cite{CLW17})
we could construct a series of relatively compact
subset $\{D_m\}_{m=1}^{\infty}$ of $M$ (with $o\in D_m$ for all $m$), and a series of adapted vector fields
$\{l_{m,T}\}_{m,T=1}^{\infty}$ such that $l_{m,T}: [0,\infty)\times W_{\R^+}(M)\rightarrow [0,1]$,
items (1)-(2) in Lemma \ref{l2.5} and the following estimates hold
\begin{equation*}
\sup_{x \in D_{m}}\int_{{\bf E}^x_{\R^+}(M)}\int_0^t |l'_{k,T}(s,\gamma)|^p \dd s \mu^x_{\R^+}(\dd \gamma)<\infty,\ \ k>m,\ p>0.
\end{equation*}
In particular, by (1) in Lemma \ref{l2.5} we have
\begin{equation*}
l_{m,T}(s,\gamma)\equiv 0,\ \mu^x_{\R^+}-a.s.\ \gamma\in {\bf E}_{\R^+}^x(M) \ \ \textrm{if}\ x\notin D_{m}.
\end{equation*}
As before, we split  $\gamma\in {\bf E}_{\R}(M)$ into $\tilde{\gamma}, \bar{\gamma}\in {\bf E}_{\R^+}(M)$ by
$$\tilde{\gamma}(s):=\gamma(s), s\geq0,\quad \bar{\gamma}(s):=\gamma(-s), s\geq0,$$
and following the procedures of \eqref{eq2.05-1} we could extend $l_{m,T}$ to
an adapted vector field $\hat l_{m,T}:\R \times {\bf E}_{\R}(M)\rightarrow [0,1]$. Moreover, it holds that
\begin{equation*}
\hat l_{m,T}(s,\gamma)\equiv 0,\ \mu^x_{\R}-a.s.\ \gamma\in {\bf E}_{\R}^x(M) \ \ \textrm{if}\ x\notin D_{m}.
\end{equation*}
\begin{equation}\label{e3-2a}
\sup_{x \in D_{m}}\int_{{\bf E}^x_{\R}(M)}\int_0^t |\hat l'_{k,T}(s,\gamma)|^p \dd s \mu^x_{\R}(\dd \gamma)<\infty,\ \ \forall\ k>m,\ p>0.
\end{equation}
By the proof of Theorem \ref{l2.7} in \cite{CLW17} (see also appendix), 
\eqref{eq2.04} holds for $\mu^x_{\R}$ with every $x\in D_q$
with $q<m$, which yields immediately
for every  $F\in \F C_c,\ h \in \H^{\infty},\ m,k,T\in \mathbb{N}^+$ with $k>m$ (note that
$h(0)=0$ for every $h\in \H^{\infty}$),
\begin{equation}\label{e3-2}
\int_{D_m}\int_{{\bf E}_\mathbb{R}^x(M)} \langle DF, \hat l_{k,T}h\rangle_{\HH}
\dd \mu^{x}_{\R}\nu(\dd x)=
\int_{D_m}\int_{{\bf E}_{\R}^x(M)} F\Theta_h^{k,T} \dd \mu^{x}_{\R}\nu(\dd x),\
\end{equation}
where $\Theta_h^{k,T}$ is defined by \eqref{eq2.04-1}.

Fix a sequence of elements $\{h_k\}\subset \H^{\infty}$ such that it is an orthonormal basis in $\HH$, we define
the following symmetric quadratic form
$$\E^{\nu}_{\R}(F,G):=\frac{1}{2}\int_{{\bf E}_{\R}(M)}
\langle DF, DG\rangle_{\HH}\dd\mu^{\nu}_{\R}=
\frac{1}{2}\sum_{k=1}^\infty\int_{{\bf E}_{\R}(M)} D_{h_k}F D_{h_k}G  \dd\mu^{\nu}_{\R}; \quad F,G\in\F C_c.$$
In particular, by  the same  arguments as in the proof of Lemma \ref{l4-1}, we know $\E^{\nu}_{\R}(F,F)<\infty$ for every
$F\in \F C_c$.

Since the reference measure $\mu^{\nu}_{\R}$ has infinite mass, we use a cut-off technique to prove the quasi-regularity of the associated $L^2$-Dirichlet form.

\beg{thm}\label{t 4-1}
Suppose that \eqref{l4-1-1} holds. Then the quadratic form $(\E^{\nu}_{\R}, \F C_c)$
is closable and its closure $(\E^{\nu}_{\R},\D(\E^{\nu}_{\R}))$ is a quasi-regular Dirichlet form on
$L^2({\bf E}_{\R}(M);\mu^{\nu}_{\R})$.
\end{thm}
\begin{proof}
{\bf$(a)$ Closablity:} The proof is similar to that of Theorem \ref{T 3.1}. Suppose $\{F_k\}_{k=1}^{\infty}\subseteq \F C_c$ is a sequence of cylinder functions with
\begin{equation}\label{t4 -1-1}
\begin{split}
\lim_{m \rightarrow \infty}\mu^{\nu}_{\R}\left( F_m^2\right)=0,\ \
\lim_{k,m \rightarrow \infty}\E^{\nu}_{\R}\left(F_k-F_m,F_k-F_m\right)=0.
\end{split}
\end{equation}
Thus $\{D F_m\}_{m=1}^{\infty}$ is a Cauchy sequence in
$L^2\left({\bf E}_{\R}(M)\rightarrow \HH;\mu^{\nu}_{\R}\right)$ for which there exists a limit $\Phi$. It suffices to prove that $\Phi=0$.

Combining \eqref{t4 -1-1} with \eqref{e3-2a} and \eqref{e3-2} yields that for all $m,k,T\in \mathbb{N}^+$, $G\in \F C_c$ and
the orthonormal basis $\{h_i\}_{i=1}^{\infty}\subset \H^{\infty}$ of $\HH$ with
$k>m$,
\begin{equation*}
\int_{D_m}\int_{{\bf E}_\mathbb{R}^x(M)} G\langle \Phi, \hat l_{k,T}h_i\rangle_{\HH}
\dd \mu^{x}_{\R}\nu(\dd x)=0,
\end{equation*}
which ensures the existence of a $\mu_{\R}^{\nu}$-null set $\Delta_i$ such that for all
$m,k,T\in \mathbb{N}^+$ with $k>m$,
\begin{equation}\label{t 4-1-2}
\hat l_{k,T}(\gamma)\langle \Phi(\gamma),h_i\rangle_{\HH}=0,\ \ \forall\ \gamma\notin \Delta_i, \gamma(0)\in D_m.
\end{equation}

For a fixed $h_i\in \H^{\infty}$, we could find  $T_i\in \mathbb{N}^+$ (which may depend on $h_i$) satisfying
$\supp h_k\subset [-T_i,T_i]$. Since $\gamma(\cdot)$ is non-explosive,
for every $\gamma \notin \Delta_0$ with some $\mu^o_{\R}$-null set $\Delta_0$, there exist $m_i,k_i\in \mathbb{Z}_+$
(which may depend on $\gamma$), such that
$k_i>m_i$, $\gamma(0)\in D_{m_i}$, $ \gamma(t)\in D_{k_i-1}$  for all $t \in [-T_i,T_i]$,
hence $\hat l_{k_i,T_i}(t,\gamma)=1$ for all $t \in [-T_i,T_i]$.
By this  and \eqref{t 4-1-2} we know
$$\langle \Phi(\gamma), h_i\rangle_{\HH}=0,\quad  i\ge 1,
\gamma \notin \Delta_0\cup\Delta_i,$$
which implies that
$\Phi(\gamma)=0$, $\forall\ \gamma \notin \Delta:=\cup_{i=0}^{\infty}\Delta_i$. So $\Phi=0$, a.s., and
$(\E_{\R}^\nu,\F C_c)$ is closable.
By  standard methods, we show easily that its closure $(\E^{\nu}_{\R},\D(\E^{\nu}_{\R}))$ is a Dirichlet form.

{\bf$(b)$ Quasi-Regularity:}

We first verify
(i) of  \cite[Definition IV-3.1]{MR92}:
Since the metric space $({\bf E}_{\R}(M); \tilde d)$ ($\tilde d$ is defined by \eqref{eq2.12-a}) is separable,
we can choose a fixed countable dense subset $\{\xi_m|m\in\mathbb{N}^+\}\subset W_{\R}(M)$.
Let $\varphi\in C_b^\infty(\mathbb{R})$ such that
$\varphi$ is an increasing function  satisfying
$$\varphi(t)=t,\quad \forall~t\in[-1,1]~~\text{and}~~\|\varphi'\|_{\infty}\leq 1.$$
Let $\phi_R\in C_c^{\infty}(\R)$ such that $\|\phi_R'\|_{\infty}\le 2$ and
\begin{equation*}
\phi_R(x)=
\begin{cases}
& 1,\ \ \ \ \ \ \ \ \text{if}\ |x|\le R,\\
& \in (0,1),\ \ \ \ \text{if}\ R<|x|\le R+1,\\
& 0,\ \ \ \ \ \ \ \ \text{if}\ |x|>R+1.
\end{cases}
\end{equation*}

For  fixed $o\in M$ and each $m,R\in \mathbb{N}^+$, we define $v_{m,R}:{\bf E}_{\R}(M)\rightarrow\mathbb{R}$ by
$$v_{m,R}(\gamma)=\phi_R\Big(\int_0^{1} \rho(o,\gamma(s))\dd s\Big)\varphi(\tilde{d}(\gamma,\xi_m)),\quad\gamma\in {\bf E}_{\R}(M)
.$$
Then by similar argument as in the proof of Lemma \ref{l2.8}, it is easy to see that $v_{m,R}\in  \D(\E^{\nu}_{\R})$.

Define for closed set $A\subset {\bf E}_{\R}(M)$
$$\D_A(\E^{\nu}_{\R}):=\{u\in \D(\E^{\nu}_{\R})|u=0 \quad \mu_{\R}^{\nu}-\textrm{ a.e. on } A^c\}, $$
which is a closed subspace of $\D(\E^{\nu}_{\R})$. This implies that $(\E^{\nu}_{\R}, \D_A(\E^{\nu}_{\R}))$ is a Dirichlet form.
Now we have $v_{m,R}\in \D_{B_{R+1}}(\E^{\nu}_{\R})$,
with $B_{R}:=\{\gamma\in {\bf E}_{\R}(M), \int_0^{1}\rho(o,\gamma(s))\dd s\le R\}$.

Still according to the same procedures as that in the proof of Lemma \ref{l2.8} (2) we have for every $m,R\in \mathbb{N}^+$,
\begin{align*}
&Dv_{m,R}(\gamma)(s)\\
=&\phi_R\Big(\int_0^{1} \rho(o,\gamma(s))\dd s\Big)\varphi'(\tilde d(\gamma,\xi_m))\cdot\Bigg(\sum_{k=1}^{\infty}\frac{1}{2^k}\Big(
 U_s^{-1}(\tilde \gamma)\nabla_1 \tilde{\rho}(\tilde \gamma(s),\xi_m(s))1_{(k-1,k]}(s)\\
&+U_{-s}^{-1}(\bar \gamma)\nabla_1 \tilde{\rho}(\bar \gamma(-s),\xi_m(s))1_{[-k,-k+1)}(s)\Big)\Bigg)\\
&+\phi_R'\Big(\int_0^{1} \rho(\gamma(s),o)\dd s\Big)\varphi(\tilde d(\gamma,\xi_m))\Big(U_s^{-1}(\tilde \gamma)
\nabla_1 \rho(\tilde \gamma(s),o)\Big)1_{(0,1]}(s)
\end{align*}
for $\dd s\times \mu^\nu_{\R}-a.s.
  (s,\gamma)\in \R\times {\bf E}^o_{\R}(M).$
Such expression yields that for every fixed $R\in\mathbb{N}^+$,
\begin{equation*}
\int \sup_{m\ge 1}|Dv_{m,R}|_{\bf H}^2\dd \mu^\nu_\R<\infty.
\end{equation*}
 Based on this and \cite[Lemma I-2.12, Proposition III-3.5, Lemma IV-4.1]{MR92} we obtain that  for every fixed $R>0$,
\begin{equation}\label{eq2.13}
w_{k,R}:=\inf_{m\leq k}v_{m,R} \textrm{ converges } \E^\nu_{\R} -\text{quasi-uniformly to zero on} ~{\bf E}_{\R}(M).
\end{equation}
Therefore for each $R,N\in \mathbb{N}^+$ there exists a closed set $\hat F_{N,R}\subset {\bf E}_{\R}(M)$ with
\begin{equation}\label{Cap1}
\text{Cap}((\hat F_{N,R})^c)<\frac{1}{N},
\end{equation}
 and
$w_{k,R}$ converges  uniformly on $\hat F_{N,R}$ to zero as $k \rightarrow \infty$.
Here $\text{Cap}$ denotes the capacity associated to the Dirichlet form
$(\E^{\nu}_{\R},\D(\E^{\nu}_{\R}))$. In particular, for every open set $U\subset {\bf E}_{\R}(M)$
\begin{equation*}
\text{Cap}(U):=\inf\{\E^{\nu}_{\R,1}(w,w)| w\in \D(\E^{\nu}_{\R}), w\ge G_1\psi\quad \mu^\nu_\R\textrm{-a.e. on }U \},
\end{equation*}
where  $\psi\in L^2({\bf E}_{\R}(M),\mu^\nu_\R)$ with $\psi>0$ is arbitrarily chosen,
for $\beta\in \R^+, w\in \D(\E^{\nu}_{\R})$, $\E^{\nu}_{\R,\beta}(w,w):=\E^{\nu}_{\R}(w,w)+\beta\mu^\nu_\R(w^2)$ and $(G_\alpha)_{\alpha>0}$ is the resolvent associated to the Dirichlet form $(\E^{\nu}_{\R},\D(\E^{\nu}_{\R}))$
(we refer readers to \cite[Chapter III. Defi. 2.4]{MR92} for more details).

Set ${F}_{N,R}:=\hat{F}_{N,R}\cap B_R$. Since by definition of $\F C_c$ and $\D(\E_{\R}^\nu)$, it is easy to verify that
 $\F C_c\subset \bigcup_{R=1}^\infty\D_{B_R}(\E_{\R}^\nu)
\subset \D(\E_{\R}^\nu)$, which implies that $ \bigcup_{R=1}^\infty\D_{B_R}(\E_{\R}^\nu)$ is dense in
$\D(\E_{\R}^\nu)$ (with respect to $\E^{\nu}_{\R,1}$ norm). Then according to \cite[Theorem III-2.11]{MR92}
we obtain
\begin{equation}\label{t4-1-3}
\lim_{R \rightarrow \infty}\text{Cap}(B_R^c)=0.
\end{equation}
Note that by \eqref{Cap1}
$$\aligned\text{Cap}((F_{N,R})^c)\le &
\text{Cap}((\hat F_{N,R})^c)+\text{Cap}(B_R^c)\\
\le& \frac{1}{N}+\text{Cap}(B_R^c).\endaligned$$
Combining this with \eqref{t4-1-3} yields
\begin{equation}\label{t4-1-4}
\lim_{N,R\rightarrow \infty}\text{Cap}((F_{N,R})^c)=0.
\end{equation}

Moreover, we have $w_{k,R}\rightarrow0$ uniformly on $F_{N,R}\subset B_{R}$ as $k\rightarrow \infty$
and $\phi_R(\int_0^{1}\rho(o,\gamma(s))\dd s)=1$ on $B_{R}$, therefore due to the definition of
  $w_{k,R}$ it is not difficult to verify for every fixed $N,R\in \mathbb{N}^+$,
\begin{align*}
\lim_{k \rightarrow \infty}\sup_{\gamma \in F_{N,R}}\inf_{m\le k}\varphi\big(\tilde d(\gamma,\xi_m)\big)=0.
\end{align*}
Hence for every $0<\varepsilon<1$ there exists $k\in\mathbb{N}^+$ such that $w_{k,R}<\varepsilon$ on $F_{N,R}$, which implies that
$F_{N,R}\subset \cup_{m=1}^k B(\xi_m,\varepsilon)$, where $B(\xi_m,\varepsilon):=\{\gamma\in {\bf E}_{\R}(M);
\tilde d(\xi_m,\gamma)<\varepsilon\}$ denotes the ball in $({\bf E}_{\R}(M),\tilde d)$.
Consequently, for every $N,R\in \mathbb{N}^+$, $F_{N,R}$ is totally bounded, hence compact.

By now we have shown that $\{F_{N,R}\}_{N,R=1}^{\infty}$ is a compact $\E$-nest. So
(i) of \cite[Definition IV-3.1]{MR92} holds

{For any $\gamma, \eta\in {\bf E}_{\R}(M)$ with $\varepsilon:=\tilde{d}(\gamma,\eta)>0$}, then there exist $R\in\mathbb{N}$ and  certain $\xi_M$ such that $\tilde{d}(\xi_M,\eta)<\frac{\varepsilon}{4}$ and $\tilde{d}(\xi_M,\gamma)>\frac{\varepsilon}{4}$.
Taking a $R$ large enough such that $\phi_R\Big(\int_0^{1} \rho(\gamma(s),o)\dd s\Big)=\phi_R\Big(\int_0^{1} \rho(\eta(s),o)\dd s\Big)=1$, then
it is easy to see $v_{M,R}(\gamma)\neq v_{M,R}(\eta)$.
Hence $\{v_{m,R}(\gamma),m, R\in \mathbb{N}^+\}$ separate points and
(iii) of \cite[Definition IV-3.1]{MR92} follows.
Following the same procedures as in the proof of Theorem \ref{T2.1} and Theorem \ref{T3.1} above,  we could
check (ii) in \cite[Definition IV-3.1]{MR92}. By now we have finished the proof.
\end{proof}

By using the theory of Dirichlet form (refer to \cite{MR92}), we obtain the following associated diffusion process.
Furthermore, we also obtain that the process is conservative in the sense that the lifetime of the process is infinity. If the reference measure is finite, it is easy to see $1\in \D(\E^\nu_\R)$ and $\E^\nu_\R(1,1)=0$, which implies the processes are conservative and recurrent.  However, in this case $1\notin \D(\E^\nu_\R)$. Motivated by \cite{Dav} for the finite dimensional case, we construct suitable approximation functions and obtain that the processes are conservative under mild assumptions.

 \beg{thm}\label{t4-2a}
  Suppose that \eqref{l4-1-1} holds. There exists a 
 (Markov) diffusion process
 $\mathbf{M}=(\Omega,\F,\M_t,$ $(X(t))_{t\geq0},(\mathbf{P}^z)_{z\in {\bf E}_{\R}(M)})$ on ${\bf E}_{\R}(M)$
 properly associated with $(\E_{\R}^\nu,\D(\E_{\R}^\nu))$, i.e. for $u\in L^2({\bf E}_{\R}(M);\mu^\nu_{\R})\cap\B_b({\bf E}_{\R}(M))$, the transition semigroup $P_tu(z):=\EE^z[u(X(t))]$ is an $\E_{\R}^\nu$-quasi-continuous version of $T_tu$ for all $t >0$, where $T_t$ is the semigroup associated with $(\E^\nu_{\R},\D(\E^\nu_{\R}))$. Moreover, the results in Theorem \ref{T2.4} also hold in this case.

 Moreover, if conditions \eqref{r4-1-1} and \eqref{r4-1-2a} hold, then the diffusion process
 $\mathbf{M}=(\Omega,\F,\M_t,$ $(X(t))_{t\geq0},(\mathbf{P}^z)_{z\in {\bf E}_{\R}(M)})$ is  conservative in the sense that $T_t1=1$ $\mu_\R^\nu$-a.e. for all $t>0$ (c.f. \cite[Section 1.6 P56]{FOT94}).

In particular, for $M=\R$, $\nu$ being Lebesgue measure, the diffusion process
 $\mathbf{M}=(\Omega,\F,\M_t,$ $(X(t))_{t\geq0},(\mathbf{P}^z)_{z\in {\bf E}_{\R}(M)})$ is recurrent in the sense that $Gf=0$ or $\infty$ $\mu_\R^\nu$-a.e. with
$f\in L^1({\bf E}_{\R}(M);\mu^\nu_{\R})$, $f\geq0$ (c.f. \cite[Section 1.6 P56]{FOT94}). Here $Gf=\int_0^\infty T_tf\dd t$.
 \end{thm}

\begin{proof}
The existence of a diffusion process is the same as that
for Theorem \ref{T2.4} (due to quasi-regularity of $(\E_{\R}^\nu,\D(\E_{\R}^\nu))$), so we omit it here.

{\bf Step (1)}
We first prove that the process is conservative.

Choose $\phi_R\in C_c^\infty(\R)$ to
be the same function as that in the proof of Theorem \ref{t 4-1}.  For every
$R>0$, we define $\Phi_R(\gamma):=\phi_R\Big(\int_{0}^{1} \rho(o,\gamma(s))\dd s\Big)$.
For $N>0$, choose $F\in L^2({\bf E}_{\R}(M), \mu_{\R}^\nu), F\geq0$ with $F(\gamma)=\phi_N(\int_{0}^{1} \rho(o,\gamma(s))\dd s)$.
Let $(L,\D(L))$ denote the infinitesimal generator associated with
$(\E_{\R}^\nu, \D(\E_{\R}^\nu))$, then  it holds that
$u_t:=T_t F\in \D(L)$ for all $t>0$.


Note that
\begin{equation*}
\begin{split}
&D\Phi_{R}(\gamma)(s)
=\phi_R'\Big(\int_0^{1} \rho(\gamma(s),o)\dd s\Big)\Big(U_s^{-1}(\gamma)
\nabla_1 \rho(\gamma(s),o)\Big)1_{[0,1]}(s).
\end{split}
\end{equation*}
Since $D\Phi_R(\gamma)=0$ for all $\gamma$ satisfying $\inf_{t\in[0,1]}\rho(\gamma(s),o)> R+1$, by \eqref{r4-1-4} and
\eqref{r4-1-5} we obtain for all $R>1$,
\begin{equation}\label{bod}\aligned
&\E_{\R}^\nu(\Phi_R)=\int_M\int_{{\bf E}_{\R}^x(M)}
|D\Phi_R(\gamma)|_{\HH}^2\dd \mu^x_{\R}\nu(\dd x)\\
&\le \int_{B(o,2R)}\int_{{\bf E}_{\R}^x(M)}\dd\mu^x_{\R}\nu(\dd x)+\int_{B(o,2R)^c}\mu^x_{\R^+}\Big(
\sup_{s\in [0,1]}\rho(x,\gamma(s))\ge \rho(o,x)-R-1\Big)\nu(\dd x)\\&\leq \nu(B(o,2R))+c_1\exp(-c_2R^2)\leq
c_3\exp(c_4R^{\zeta}),
\endaligned\end{equation}
where $c_1-c_4$ are positive constants independent of $R$, $\zeta:=\max\{1+\frac{\alpha}{2},\beta\}<2$ with
$\alpha, \beta\in (0,1)$ being the constants in \eqref{r4-1-1} and \eqref{r4-1-2a}.
Then we have
\begin{equation}\label{diff}\aligned\mu_\R^\nu( F\Phi_R)-\mu_\R^\nu(  u_t\Phi_R)=&-\int_0^t\frac{d}{ds}\mu_\R^\nu(  u_s\Phi_R) \dd s=-\int_0^t\mu_\R^\nu( Lu_s\Phi_R) \dd s
\\=&\int_0^t\int\langle Du_s,D\Phi_R\rangle_{\HH}\dd\mu^\nu_\R \dd s=\int_0^t\int\langle \varphi_{N,R}Du_s,\varphi_{N,R}^{-1}D\Phi_R\rangle_{\HH}\dd\mu^\nu_\R \dd s
\\\leq& \bigg(\int_0^t \int|\varphi_{N,R} Du_s|^2_{\HH}\dd \mu^\nu_\R\dd s\bigg)^{1/2}\bigg(\int_0^t\int|\varphi_{N,R}^{-1} D\Phi_R|^2_{\HH}\dd \mu^\nu_\R\dd s\bigg)^{1/2},\endaligned\end{equation}
where the operator $D$ on $u_s$ is the closure of $D$ defined in \eqref{eq2.03} and
$$\varphi_{N,R}(\gamma):=\exp\Big(\theta \psi_{N,R}(\int_{0}^{1} \rho(\gamma(s),o)\dd s)\Big),$$
 for some $\theta>0$, $R>2(N+1)$ and $\psi_{N,R}\in C_b^1(\mathbb{R}^+)$ satisfies $\|\psi_{N,R}'\|_\infty\leq 2$,
$\psi_{N,R}(t)=t$ for $t\in [R,R+1]$ and $\psi_{N,R}(t)=0$ for $t\in [0,N+1]$. Define $\varphi_{N,R,M}:=\varphi_{N,R}\Phi_M$. It is obvious that $\varphi_{N,R,M}\in \F C_c$ and
$\lim_{M\rightarrow\infty}\varphi_{N,R,M}=\varphi_{N,R}$ $\mu_\R^\nu$-a.s. $\gamma$. By \cite[Corollary I-4.15]{MR92} we know $\varphi_{N,R,M}^2 u_t\in \D(\E^\nu_\R)$.
Furthermore we have
$$\aligned&\frac{\partial}{\partial t}\mu^\nu_\R(\varphi_{N,R,M}^2 u_t^2)=2\mu^\nu_\R( \varphi_{N,R,M}^2 Lu_t\cdot u_t)=
-2\int \langle Du_t, D(\varphi_{N,R,M}^2 u_t)\rangle_{\HH}\dd \mu^\nu_{\R}
\\=&-2\int\langle Du_t, 2u_t\varphi_{N,R,M}D\varphi_{N,R,M}+\varphi_{N,R,M}^2Du_t\rangle_{\HH}\dd \mu^\nu_{\R}\\
\le&-2\int|\varphi_{N,R,M} Du_t|_{\HH}^2\dd \mu^\nu_{\R}+2\Big(\lambda^{-1}\int|\varphi_{N,R,M} Du_t|_{\HH}^2\dd \mu^\nu_{\R}+\lambda\int|u_t D\varphi_{N,R,M} |_{\HH}^2\dd \mu^\nu_{\R}\Big)\\
\leq&-2\int|\varphi_{N,R,M} Du_t|_{\HH}^2\dd \mu^\nu_{\R}+2\Big(\lambda^{-1}\int|\varphi_{N,R,M} Du_t|_{\HH}^2\dd \mu^\nu_{\R}+8\lambda\theta^2\mu^\nu_\R(\varphi_{N,R,M}^2 u_t^2)\\&+2\lambda\mu^\nu_\R(\varphi_{N,R}^2 |D\Phi_M|_{\HH}^2 u_t^2)\Big).\endaligned$$
Here the last step is due to the property $|D\varphi_{N,R,M}|_{\HH}^2\le 8\theta^2\varphi_{N,R,M}^2+2\varphi_{N,R}^2 |D\Phi_M|_{\HH}^2$.

Choosing $\lambda=1$ and  using Gronwall's Lemma we obtain that
 $$\mu^\nu_\R(\varphi_{N,R,M}^2 u_t^2)\leq \exp(16\theta^2t)\Big(\mu^\nu_{\R}(\varphi_{N,R,M}^2F^2)+\frac{1}{4\theta^2}\lambda\mu^\nu_\R(\varphi_{N,R}^2 |D\Phi_M|_{\HH}^2 u_t^2)\Big).$$
 By the dominated convergence theorem we know  $\lim_{M\rightarrow\infty}\mu^\nu_\R(\varphi_{N,R}^2 |D\Phi_M|_{\HH}^2 u_t^2)=0$.
Based on this, choosing $\lambda=2$ and letting $M\rightarrow\infty$ we have
\begin{equation}\label{bod1}
\int_0^t|\varphi_{N,R} Du_s|_{\HH}^2\dd s\leq 2e^{16\theta^2t}\mu^\nu_{\R}(\varphi_{N,R}^2F^2).
\end{equation}

 For $\gamma$ with $D\Phi_R(\gamma)\neq0$ (i.e. $R\le \int_0^1\rho(o,\gamma(s))\dd s\le R+1$) it is easy to see $\varphi_{N,R}(\gamma)^{-1}\leq e^{-\theta R}$. Now combining \eqref{bod}, \eqref{diff} and \eqref{bod1} yields
 $$\aligned\mu_\R^\nu( F\Phi_R)-\mu_\R^\nu(  u_t\Phi_R)
\leq \bigg[2c_3e^{16\theta^2t}\mu^\nu_{\R}(\varphi_{N,R}^2F^2)e^{-2\theta R}te^{c_4 R^{\zeta}}\bigg]^{1/2},\endaligned$$
 Choosing $\theta=\frac{R}{16t}$ we have
  \begin{equation}\label{e3-1-1}
  \mu_\R^\nu( F\Phi_R)-\mu_\R^\nu(  u_t\Phi_R)
\leq \bigg[c_5\mu^\nu_{\R}(\varphi_{N,R}^2F^2)te^{-\frac{R^2}{16t}+c_4 R^{\zeta}}\bigg]^{1/2},
\end{equation}
where $ c_4,c_5$ are independent of $F$, $N$ and $R$.

We arrive at for all $R>2(N+1)$
\begin{align*}
 \mu_\R^\nu( \Phi_N\Phi_R)-\mu_\R^\nu(  T_t(\Phi_N)\Phi_R)
&\leq \bigg[c_5\mu^\nu_{\R}(\varphi_{N,R}^2\Phi_N^2)te^{-\frac{R^2}{16t}+c_4 R^{\zeta}}\bigg]^{1/2}\\
&=\bigg[c_5\mu^\nu_{\R}(\Phi_N^2)te^{-\frac{R^2}{16t}+c_4 R^{\zeta}}\bigg]^{1/2},
\end{align*}
where the last equality is due to the fact $\Phi_N(\gamma)\neq 0$ only if $\varphi_{N,R}(\gamma)=1$ since
$R>2(N+1)$. Hence letting $R\rightarrow \infty$ we derive for every $N>0$ and $t>0$ (note that $\zeta<2$ here)
$$\int \Phi_N \dd \mu_{\R}^\nu- \int \Phi_N T_t 1\dd \mu_{\R}^\nu=\int \Phi_N \dd \mu_{\R}^\nu- \int T_t(\Phi_N)\dd \mu_{\R}^\nu\leq0.$$
Since it always hold $T_t 1\le 1$, the above inequality implies that
$T_t1(\gamma)=1$ for all $\gamma\in {\bf E}_{\R}(M)$
satisfying $\int_{0}^{1}\rho\big(\gamma(s),o\big)\dd s\leq N$.
Also note that $N$ is arbitrary, we obtain $T_t1(\gamma)=1$ for $\mu^\nu_{\mathbb{R}}$-a.e. $\gamma\in
{\bf E}_{\R}(M)$ immediately, therefore the process ${\bf M}$ is conservative.

{\bf Step (ii)} Now we prove the recurrence property.
 Choosing $\tilde{\phi}_R\in C_c^\infty(\R^+)$ satisfying
  \begin{equation*}
\tilde{\phi}_R(x)=
\begin{cases}
& 1,\ \ \ \ \ \ \ \ \text{if}\ x\le R,\\
& \in (0,1),\ \ \ \ \text{if}\ R<x<2R,\\
&0.\ \ \ \ \ \ \ \  \text{if}\  x>2R,
\end{cases}
\end{equation*}
 and $\|\phi'_R\|_\infty\leq \frac{1}{R}$.
 We define
$\tilde{\Phi}_R(\gamma):=\tilde{\phi}_R\Big(\int_{0}^{1} \rho(o,\gamma(s))\dd s\Big)$.
Then we have
\begin{equation*}
\begin{split}
&D\tilde{\Phi}_{R}(\gamma)(s)
=\phi_R'\Big(\int_0^{1} \rho(\gamma(s),o)\dd s\Big)\Big(U_s^{-1}(\gamma)
\nabla_1 \rho(\gamma(s),o)\Big)1_{[0,1]}(s).
\end{split}
\end{equation*}
Now it holds $|D\tilde{\Phi}_{R}|_{\HH}\leq \frac{1}{R}$ and
$D\tilde{\Phi}_R(\gamma)=0$ all $\gamma$ satisfying $\inf_{t\in[0,1]}\rho(\gamma(s),o)> 2R$, then still according to \eqref{r4-1-5} we get
\begin{align*}
&\E_{\R}^\nu(\tilde{\Phi}_R)=\int_M\int_{{\bf E}_{\R}^x(M)}
|D\tilde{\Phi}_R(\gamma)|_{\HH}^2\dd \mu^x_{\R}\nu(\dd x)\\
&\le \frac{1}{R^2}\int_{B(o,3R)}\nu(\dd x)
+\int_{B(o,3R)^c}\int_{{\bf E}_{\R^+}^x(M)}\mu^x_{\R^+}\Big(
\sup_{s\in [0,1]}\rho(x,\gamma(s))\ge \rho(o,x)-2R-1\Big)\nu(\dd x)\\
&\leq \frac{c_6}{R}+c_6\exp(-c_7R^2)\rightarrow 0, R\rightarrow\infty.
\end{align*}
Therefore we have found as series of $\tilde{\Phi}_R$ such that
$\tilde{\Phi}_R\rightarrow1$ $\mu^\nu_\R$-a.e. as $R\rightarrow\infty$ and
$\E_{\R}^\nu(\tilde{\Phi}_R)\rightarrow 0$ as $R \rightarrow \infty$,  so the recurrence follows by
\cite[Theorem 1.6.5]{FOT94}.


\end{proof}

\begin{Remark}\label{loop}
By integration by parts formula obtained in \cite{CLWFL} and carefully tracking the proof of Theorem \ref{t 4-1} and \ref{t4-2a}, we could verify that if
\eqref{r4-1-1} and \eqref{r4-3-2} are true, then the conclusions of  Theorems \ref{t 4-1} and \ref{t4-2a}
still hold for $(\tilde \E_{\R}^\nu, \mathscr{D}(\tilde \E_{R}^\nu))$ with $\nu(\dd x)=p_1(x,x)\dd x$. Here $(\tilde \E_{\R}^\nu, \mathscr{D}(\tilde \E_{R}^\nu))$ is defined in Remark \ref{l o}.

 Furthermore  a similar argument implies that the results in Theorems \ref{t 4-1} and \ref{t4-2a} also hold for  the reference measure given by  $e^{c\int_0^1\mathrm{Scal}(\gamma(s))\dd s}\tilde \mu^\nu(\gamma)$, if
\eqref{r4-1-1} and \eqref{r4-3-2} are true. Here $c\in \mathbb{R}$ and $\mathrm{Scal}$ denotes the scalar curvature.
\end{Remark}

\beg{Remark}\label{Finite volume}{\bf (Finite Volume Case for the line)} For each $A_1,A_2\in[0,\infty)$, we could also construct Wiener measure on $C([-A_2,A_1],M)$. In this case the above results also hold.
\end{Remark}

\section{Ergodicity/ Non-ergodicity}\label{sect4}
\subsection{Half line}\label{sect4}
In this section, we  study the long time behavior of  the Markov process  $X(t), t\geq0,$  and the $L^2$-Dirichlet form
$(\E_{\R^+}^o,\D(\E_{\R^+}^o))$ constructed in Section 2. In fact, we establish some functional inequalities associated with $(\E_{\R^+}^o,\D(\E_{\R^+}^o))$, which gives ergodicity or non-ergodicity of the corresponding Markov process $X(t), t\geq0$.


\subsubsection{$M$ has strictly positive Ricci curvature}

 \beg{thm}\label{T3.1}[Log-Sobolev inequality and Poincar\'e inequality]
\begin{itemize}
\item [(1)]  Suppose that $\Ric\geq K$ for $K>0$, then the log-Sobolev inequality holds
\begin{equation}\label{eq3.2}
\mu_{\R^+}^o(F^2\log F^2)\le 2C(K) \E_{\R^+}^o(F,F),\ \ \ \ F \in \F C^1_{b},
\ \mu_{\R^+}^o(F^2)=1,\end{equation}
where $C(K):=\frac{4}{K^2}$.

\item[(2)] Suppose that $M$ is compact and there exists $\varepsilon \in (0,1)$ such that
\begin{equation}\label{t3-1}
\delta_\varepsilon:=\sup_{T\in [0,\infty)}\delta_{\varepsilon}(T)<\infty,
\end{equation}
where
\begin{equation}\label{t3-1a}
\begin{split}
&\delta_{\varepsilon}(T):=\varepsilon^{-1}\big(1-e^{-\varepsilon T}\big)\int_0^T e^{\varepsilon s}\eta(s)\dd s,
\ \eta(s):=\sup_{x\in M}\mu_{\R^+}^x\Big[\exp\Big(-\int_0^s K(\gamma(r))\dd r\Big)\Big],
\end{split}
\end{equation}
and $K(x):=\inf\{{\rm Ric}_x(X,X); X\in T_x M, |X|=1\}$, $x\in M$. Then the following Poincar\'e inequality holds,
\begin{equation}\label{t3-2}
\mu_{\R^+}^o(F^2)-\mu_{\R^+}^o(F)^2 \le \delta_\varepsilon \E_{\R^+}^o(F,F),\ \ \ \ F \in \F C^1_{b},
\end{equation}
where $\delta_\varepsilon$ is defined by \eqref{t3-1}.
\end{itemize}
\end{thm}

\begin{Remark}\label{r3-1}
Obviously if
\begin{equation}\label{r3-1-1}
\limsup_{t\uparrow \infty}\frac{1}{t}\sup_{x\in M}\log \mu_{\R^+}^x\Big[\exp\Big(-\int_0^s K(\gamma(r))\dd r\Big)\Big]<0,
\end{equation}
then condition \eqref{t3-1} holds.

Moreover, as explained in \cite{ELR93,W99}, condition \eqref{r3-1-1} is equivalent to the spectral positivity of
the operator $L_0=-\Delta+K$ (here $L_0f(x):=\Delta f(x)+K(x)f(x)$). In particular, if ${\rm Ric}\ge K$ for some constant
$K>0$, then \eqref{t3-1} holds.
\end{Remark}

\beg{Remark}\label{r3.2}  \beg{enumerate}

\item[$(i)$] According to \cite{W05}, the log-Sobolev inequality implies hypercontractivity of the associated semigroup $P_t$ and Poincar\'e inequality, which derives the $L^2$-exponential ergodicity of the process:
$\|P_tF-\int F d\mu\|_{L^2}\leq e^{-t/C(K)}\|F\|_{L^2}.$

\item[$(ii)$] Poincar\'e inequality also implies the irreducibility of the Dirichlet form $(\E_{\R^+}^o,\D(\E_{\R^+}^o))$. It is obvious that the Dirichlet form $(\E_{\R^+}^o,\D(\E_{\R^+}^o))$ is recurrent. Combining these two results, by \cite[Theorem 4.7.1]{FOT94}, for any nearly Borel non-exceptional set $B$,
$$\mathbf{P}^z(\sigma_B\circ\theta_n<\infty,\forall n\geq0)=1, \quad \textrm{ for q.e. } z\in {\bf E}_{\R^+}^o(M).$$
Here $\sigma_B=\inf\{t>0:X(t)\in B\}$, $\theta$ is the shift operator for the Markov process $X$, and for the definition of any nearly Borel non-exceptional set we refer to \cite{FOT94}. Moreover by \cite[Theorem 4.7.3]{FOT94} we obtain the following strong law of large numbers: for $f\in L^1({\bf E}_{\R^+}^o(M),\mu_{\R^+}^o)$
$$\lim_{t\rightarrow\infty}\frac{1}{t}\int_0^tf(X(s))\dd s=
\int f\dd\mu^o_{\R^+}, \quad \mathbf{P}^{z}-a.s.,$$
for q.e. $z\in {\bf E}_{\R^+}^o(M)$.

\end{enumerate}
\end{Remark}

\ \newline\emph{\bf Proof of Theorem \ref{T3.1}.}
{\bf Step (1)} By the standard method and the technique in \cite{FW17}(See also \cite{GW06} and \cite{RWZZ17} and references therein), it is not difficult to prove \eqref{eq3.2}. For the reader's convenience, in the following we  give a detailed proof.

By \cite{GW06} we have the martingale representation theorem, that is, for $F\in \F C_b^1$ with the form
\begin{equation}\label{t3-1-1}
\aligned
F(\gamma)=f\left(\int_0^{T_1}  g_1(s,\gamma(s)) \dd s,\int_0^{T_2}   g_2(s,\gamma(s)) \dd s,...,\int_0^{T_m}   g_m(s,\gamma(s)) \dd s\right),\quad \gamma\in {\bf E}^o_{\R^+}(M),\endaligned\end{equation} we have
\begin{equation}\label{eq3.5}F=\mu_{\R^+}^o(F)+\int^T_0\langle H_s^F, \dd \beta_s\rangle ,\end{equation}
where $T=\max{T_i}$, $\beta_s$ is the anti-development of canonical path $\gamma(\cdot)$ (whose distribution
is an $\R^n$-valued Brownian motion under $\mu_{\R^+}^o$) and \begin{equation}\label{eq3.7} H_s^F=\mu_{\R^+}^o\left[M_s^{-1}\int_s^TM_r(DF(r))\dd r\bigg|\F_s\right].\end{equation}
Here and in the following
$(\F_t)$ is the natural filtration generated by $\gamma(\cdot)$, $\mu_{\R^+}^o\left[\cdot|\F_t\right]$ denotes the
 conditional expectation under $\mu_{\R^+}^o$ and $M_t$ is the solution of the equation
\begin{equation}\label{eq3.4}
\frac{\dd}{\dd t}M_t+\frac{1}{2}M_t\Ric_{U_t}=0,\quad M_0=I.\end{equation}
Let $F=G^2$ for $G\in \F C_b^1$ being strictly positive and with the form \eqref{t3-1-1}, consider the continuous version of the martingale $N_s=\EE[F|\F_s]$.
By the lower bound of the Ricci curvature it is easy to verify for every $0\le s\le r<\infty$
\begin{equation}\label{eq3.6}
\|M_s^{-1}M_r\|\leq \exp\Big(-\frac{1}{2}\int_s^r K(\gamma(t))\dd t\Big)\le \exp\Big(-\frac{K(r-s)}{2}\Big),
\end{equation}
where $\|\cdot\|$ denotes the matrix norm.
Then we can take the conditional expectation $\mu_{\R^+}^o[\cdot|\F_s]$ in \eqref{eq3.5} to obtain
\begin{equation}\label{t3-3}
N_s=\mu_{\R^+}^o[F]+\int_0^s\langle H_r^F,\dd \beta_r\rangle.
\end{equation}
Now applying It\^{o}'s formula to $N_s\log N_s$, we have
\begin{equation}\label{eq3.8}
\begin{split}
&\quad \mu_{\R^+}^o\left( G^2\log G^2\right)-\mu_{\R^+}^o(G^2)\log \mu_{\R^+}^o(G^2)\\
&=\mu_{\R^+}^o\left( N_T\log N_T\right)-\mu_{\R^+}^o\left(N_0\log N_0\right)=\frac{1}{2}
\mu_{\R^+}^o\left[\int_0^TN_s^{-1}|H_s^F|^2\dd s\right].
\end{split}
\end{equation}
Here and in the following we use $|\cdot|$ to denote the norm in $\mathbb{R}^d$.
Note that
$$DF=D(G^2)=2GDG.$$
Using this relation in the explicit formula \eqref{eq3.7} for $H^F$, we have
\begin{equation}\label{eq3.9}H_s^F=2\mu_{\R^+}^o\left[GM_s^{-1}\int_s^TM_r DG(r) \dd r\bigg|\F_s\right].\end{equation}

By Cauchy-Schwarz inequality in \eqref{eq3.9} and \eqref{eq3.6}, we have
$$|H_s^F|^2\leq 4 \mu_{\R^+}^o[G^2|\F_s]\mu_{\R^+}^o\bigg[\bigg(\int_s^Te^{-K(r-s)/2}|D G(r)|\dd r\bigg)^2\bigg|\F_s\bigg].$$
Thus the right hand side of \eqref{eq3.8} can be controlled by
\begin{equation}\label{eq3.1}\aligned 2\mu_{\R^+}^o
\bigg[\int_0^T\bigg(\int_s^Te^{-K(r-s)/2}|DG(r)|\dd r\bigg)^2\dd s\bigg].\endaligned\end{equation}
By H\"{o}lder's inequality we have
$$\aligned \bigg(\int_s^Te^{-K(r-s)/2}|D G(r)|\dd r\bigg)^2\leq \int_s^Te^{-K(r-s)/2}\dd r\int_s^Te^{-K(r-s)/2}|D G(r)|^2\dd r.\endaligned$$
Then  changing the order of integration we obtain
$$\aligned \mu_{\R^+}^o\bigg(\int_0^T\bigg(\int_s^Te^{-K(r-s)/2}|D G(r)|\dd\tau\bigg)^2\dd s\bigg)
\leq \mu_{\R^+}^o\bigg(\int_0^TJ_1(s,T)|DG(s)|^2\dd s\bigg),\endaligned$$
where
$$\aligned J_1(s,T):=&\int_0^s\frac{2}{K}\big(1-e^{-K(T-t)/2}\big)e^{-K(s-t)/2}\dd t\\=&\frac{2}{K^2}\bigg[2(1-e^{-\frac{Ks}{2}})-e^{-\frac{K(T-s)}{2}}
+e^{-\frac{K(T+s)}{2}}\bigg]\leq \frac{4}{K^2},\ \ \forall\ s\in [0,T]\endaligned$$
Hence
$$\aligned &\mu_{\R^+}^o\bigg(\int_0^T\bigg[\bigg(\int_s^Te^{-K(r-s)/2}|DG(r)|\dd r\bigg)^2\bigg]\dd s\bigg)
\leq 
\frac{4}{K^2}\E_{\R^+}^o(G,G).\endaligned$$
Combining all above estimates into \eqref{eq3.8}, we complete the proof for \eqref{eq3.2}.

{\bf Step (2)} Some proof in this step is inspired by that of \cite[Theorem 1]{W99}. Still applying It\^o formula to $N_s^2$ (where $N_s=\mu_{\R^+}^o[F|\F_s]$ and $F\in \F C_b^1$ with the form \eqref{t3-1-1} )
we arrive at
\begin{equation}\label{t3-1-4}
\begin{split}
&\quad \mu_{\R^+}^o(F^2)-\mu_{\R^+}^o(F)^2\\
&=\mu_{\R^+}^o(N_T^2)-\mu_{\R^+}^o(N_0^2)=\mu_{\R^+}^o\Big(\int_0^T |H_s^F|^2 \dd s\Big).
\end{split}
\end{equation}
By \eqref{eq3.7}, \eqref{eq3.6}, Markov property and Cauchy-Schwartz inequality we obtain
\begin{equation*}
\begin{split}
&\left|H_s^F\right|^2 \le \mu_{\R^+}^o\left[\int_s^T \exp\left(-\int_s^r K(\gamma(t))\dd t\right)e^{-\varepsilon(T-r)}\dd r\Big|\F_s\right]
\mu_{\R^+}^o\left[\int_s^T e^{\varepsilon(T-r)}|DF(r)|^2\dd r\Big|\F_s\right] \\
&\le \left(\int_s^T \sup_{x\in M}\mu_{\R^+}^x\left[\exp\left(-\int_0^{r-s}K(\gamma(t))\dd t\right)\right]e^{-\varepsilon(T-r)}\dd r\right)
\mu_{\R^+}^o\left[\int_s^T e^{\varepsilon(T-r)}|DF(r)|^2\dd r\Big|\F_s\right]\\
&=\left(\int_s^T \eta(r-s)e^{-\varepsilon(T-r)}\dd r\right)\mu_{\R^+}^o\left[\int_s^T e^{\varepsilon(T-r)}|DF(r)|^2\dd r\Big|\F_s\right],
\end{split}
\end{equation*}
where in the second inequality we used the Markov property of the canonical process $\gamma(\cdot)$ and $\eta(t)$ is defined
by \eqref{t3-1a}.
Therefore let $\phi(t):=\int_0^t\Big(\int_s^T \eta(r-s)e^{-\varepsilon(T-r)}\dd r\Big)\dd s$, $t\in [0,T]$ it holds
\begin{align*}
\mu_{\R^+}^o\Big(\int_0^T|H_s^F|^2\dd s\Big)&\le \int_0^T
\Big(\int_s^T \eta(r-s)e^{-\varepsilon(T-r)}\dd r\Big)\Big(\int_s^T e^{\varepsilon(T-r)}\mu_{\R^+}^o(|DF(r)|^2)\dd r\Big)\dd s\\
&=\int_0^T \phi'(s)\Big(\int_s^T e^{\varepsilon(T-r)}\mu_{\R^+}^o(|DF(r)|^2)\dd r\Big)\dd s\\
&=\mu_{\R^+}^o\Big(\int_0^T \phi(r)e^{\varepsilon (T-r)}|DF(r)|^2 \dd r\Big).
\end{align*}
Since by elementary calculation it is easy to check $\sup_{r\in [0,T]}\phi(r)e^{\varepsilon (T-r)}\le \delta_{\varepsilon}(T)$,
combining all the estimates into \eqref{t3-1-4} yields \eqref{t3-2}.
$\hfill\square$

\subsubsection{$M=\R^n$}

 In this subsection we  consider the case that $M=\R^n$ and $o=0\in \R^n$ and we use $X_t$ to denote $X(t)$ for simplicity. As mentioned in the introduction, it is easy to see that the Markov process $(X_t)_{t\ge 0}$ associated with $(\E_{\R^+}^o, \D(\E_{\R^+}^o))$ is the
unique solution to the following stochastic heat equations on $\R^+\times \R^+$
\begin{equation}\label{e4-3}
\begin{split}
\partial_t X_t=&\frac{1}{2}\Delta X_t+ \xi,\ t>0,\\
X_t(0)=&0,\ t>0,\\
X_0(\cdot)=&\gamma(\cdot)\in {\bf E}^o_{\R^+}(\R^n)
\end{split}
\end{equation}
where $\xi$ denotes an standard  $\R^n$-valued space-time white noise on $\R^+\times \R^+$ (on some probability $(\Omega,\mathscr{F},{\bf P})$).
In the Euclidean space, we have the following ergodicity results. In this case,  the exponential ergodicity does not hold any more, which implies that the $L^2$-spectral gap is zero.

\begin{thm}\label{t4-1}
Suppose $M=\R^n$, then the following statements hold
\begin{itemize}
\item [(1)] For every $F\in L^2({\bf E}^o_{\R^+}(\R^n); \mu_{\R^+}^o)$ 
we have 
\begin{equation}\label{t4-1-1}
\lim_{t \rightarrow \infty}\mu_{\R^+}^o\Big(\big|P_tF(\gamma)-\mu^o_{\R^+}(F)\big|^2\Big)=0,\
\end{equation}
where $P_tF(\gamma):={\bf E}\big[F(X_t^{\gamma})\big]$, $(X_t^{\gamma})_{t\ge 0}$ is the solution to \eqref{e4-3} with initial value $X_0(\cdot)=\gamma$.

\item[(2)] 
The  Poincar\'e inequality does not hold, i.e.  for any $C>0$, there exists $F\in \D(\E_{\R^+}^o)$ such that
\begin{equation}\label{t4-1-2}
\mu_{\R^+}^o(F^2)-\mu_{\R^+}^o(F)^2\geq C\E_{\R^+}^o(F,F).
\end{equation}
In particular, the spectral gap
$$C_{\R^+}(SG):=\inf_{F\neq \textrm{const}, F\in \D(\E_{\R^+}^o)}\frac{\E_{\R^+}^o(F,F)}{\mu_{\R^+}^o(F^2)-\mu_{\R^+}^o(F)^2}=0,$$
and the exponential ergodicity does not hold in this case.
\end{itemize}
\end{thm}
\begin{proof}
{\bf Step (1)} As explained in \cite[Page 315]{FX}, the solution $X_t$ to \eqref{e4-3} with initial value
$X_0(\cdot)=\gamma$ has the following expression,
\begin{equation*}
\begin{split}
X_t^{\gamma}(x)&=\int_{\R^+} p(t,x,y)\gamma(y)\dd y+\int_0^t\int_{\R^+}p(t-s,x,y)\xi(\dd s,\dd y)\\
&:=U_1(t,x)+U_2(t,x),
\end{split}
\end{equation*}
where $p(t,x,y)$ is the Dirichlet heat kernel on $\R^+$ with the following expression
\begin{equation*}
p(t,x,y)=\frac{1}{\sqrt{2\pi t}}\Big[\exp\big(-\frac{(x-y)^2}{2t}\big)-
\exp\big(-\frac{(x+y)^2}{2t}\big)\Big],\ \ x,y\in \R^+,\ t>0.
\end{equation*}

By \cite[Lemma 4.3]{FX} and the law of iterated logarithm (which implies $\lim_{y \rightarrow +\infty}
\frac{\gamma(y)}{y}=0$ for $\mu^o_{\R^+}$-a.s. $\gamma\in {\bf E}^o_{\R^+}(\R^n)$), it is easy to verify
that for $\mu_{\R^+}^o$-a.s. $\gamma \in {\bf E}^o_{\R^+}(\R^n)$ and every $x\in \R^+$,
$$ \lim_{t \rightarrow +\infty}U_1(t,x)=0.$$

Note that $U_2(t,\cdot)=\big(U_2^1(t,\cdot),\dots, U_2^n(t,\cdot)\big)$ is a centered Gaussian vector on $L^2(\R^+; e^{-rx}\dd x)$,
and for every
$x,y \in \R^+$ it holds
\begin{align*}
&\lim_{t \uparrow \infty}{\bf E}[U_2^i(t,x)U_2^j(t,y)]\\
&=
\lim_{t\uparrow \infty}{\bf E}\Big[\Big(\int_0^t \int_{\R^+}p(t-s,x,z)\xi^i(\dd s,\dd z)\Big)
\Big(\int_0^t \int_{\R^+}p(t-s,y,z)\xi^j(\dd s,\dd z)\Big)\Big]\\
&=\delta_i^j\lim_{t\uparrow \infty}\int_0^t\int_{\R^+}p(t-s,x,z)p(t-s,y,z)\dd z \dd s\\
&=\delta_i^j\lim_{t\uparrow \infty}\int_0^t p(2(t-s),x,y)\dd s=
\delta_i^j\lim_{t\uparrow \infty}\frac{1}{2}\int_0^{2t} p(s,x,y)\dd s=
\delta_i^j(x\wedge y),\ 1\le i,j\le n,
\end{align*}
where the last calculation can be found in \cite[Section 2.3]{CZ}, $\delta_i^j=1$ when $i=j$ and $\delta_i^j=0$ when $i\neq j$.

This implies that $U_2(t,\cdot)$ converges weakly in $L^2(\R^+; e^{-rx}\dd x)$ as $t \uparrow \infty$
to a Gaussian random vector whose distribution is $\mu^o_{\R^+}$. Combining all the estimates above we know that for
$\mu_{\R^+}^o$-a.s. $\gamma \in {\bf E}_{\R^+}^o(\R^n)$,
$X_t^\gamma(\cdot)$ converges weakly  on $L^2(\R^+; e^{-rx}\dd x)$ as $t \uparrow \infty$
to a Gaussian random vector whose distribution is $\mu^o_{\R^+}$. Thus for
$\mu_{\R^+}^o$-a.s. $\gamma \in {\bf E}_{\R^+}^o(\R^n)$ and every $F \in \F C_b^1$ we have
\begin{equation*}
\lim_{t \rightarrow \infty} P_t F(\gamma)=\mu^o_{\R^+}(F).
\end{equation*}
By this and the dominated convergence theorem
we obtain \eqref{t4-1-1} holds for $F\in \F C_b^1$ immediately. By approximations we can easily check that \eqref{t4-1-1} holds for $F\in L^2({\bf E}^o_{\R^+}(\R^n); \mu_{\R^+}^o)$, which implies that $\mu_{\R^+}^o$ is ergodic.

{\bf Step (2)} We first suppose the Poincar\'{e} inequality holds, i.e. for $F\in \D(\E^o_{\R^+})$
\begin{equation}\label{Poeq}\mu_{\R^+}^o(F^2)-\mu_{\R^+}^o(F)^2\leq C\E_{\R^+}^o(F,F)\end{equation}
 for some $C>0$.
For a fixed $T>0$, let $F_T(\gamma):=\int_0^T \gamma_1(s)\dd s$, where $\gamma_1(s)$ denotes the first
coordinate of process $\gamma(s):=(\gamma_1(s),\cdots ,\gamma_n(s))$. By the proof of Lemma \ref{l2.8}, it is not difficult to verify that
$F_T\in \D(\E_{\R^+}^o)$.

At the same time, we have for $o=0\in \R^n$
\begin{align*}
& \mu_{\R^+}^o(F_T^2)=\mu_{\R^+}^o\Big(\int_0^T\int_0^T\gamma_1(s)\gamma_1(t)\dd s \dd t\Big)\\
&=\int_0^T\int_0^T \mu_{\R^+}^o\Big(\gamma_1(s)\gamma_1(t)\Big)\dd s \dd t=\int_0^T\int_0^T (s\wedge t) \dd s \dd t\\
&\ge \int_0^T \int_0^t s\dd s \dd t\ge \frac{T^3}{6},
\end{align*}
$$\mu_{\R^+}^o(F_T)=\int_0^T \mu_{\R^+}^o(\gamma_1(s))\dd s= 0,
$$
and
$$\E_{\R^+}^o(F_T)=\int_{{\bf E}_{\R^+}^o(\R^n)}|DF_T(\gamma)|_{\HH_+}^2 \dd \mu_{\R^+}^o\le T.
$$
Here we have applied the property that $|DF_T(\gamma)(s)|\le 1_{[0,T]}(s)$. Combining all the estimates above and putting
$F_T$ into \eqref{Poeq} we arrive at
$\frac{T^3}{6}\le CT$. Then letting $T\rightarrow \infty$ we get $C=+\infty$ and there is a contradiction. So
\eqref{Poeq} does not hold for any $C>0$. The results for spectral gap follow from \cite{W05}.
\end{proof}

\begin{Remark}
By carefully tracking the proof of Theorem \ref{t4-1}, it is not difficult to verify that
the conclusion of Theorem \ref{t4-1} still holds for every initial point $o\in \R^n$, not only
$o=0$.
\end{Remark}

\subsubsection{$M$ is not a Liouville manifold}
In this subsection, we  prove that when $M$ is not a Liouville manifold,  $(\E_{\R^+}^o,\D(\E_{\R^+}^o))$ is reducible, which by \cite[Propsition 2.1.6]{CF12} implies that the Markov semigroup $(P_t)_{t\geq0}$ constructed in Theorem \ref{T2.2} is non-ergodic in the sense that there exists a non-constant function $F\in \D(\E_{\R^+}^o)$ such that $P_t F=F$ $\mu^o_{\R^+}$-a.s..

Recall that we call a connected Riemannian manifold $M$ a Louville manifold, if there does not exist
a non-constant bounded harmonic function on $M$. In particular, if $M$ is not a Liouville manifold, then there exists
a bounded harmonic function $u:M\rightarrow \R$ which is not a constant.

\begin{thm}\label{t4-2}
If $M$ is not a Liouville manifold, then $(\E_{\R^+}^o,\D(\E_{\R^+}^o))$ is reducible. Hence $\mu_{\R^+}^o$ is not
ergodic for the Markov process associated with $(\E_{\R^+}^o,\D(\E_{\R^+}^o))$.
\end{thm}
\begin{proof}
The following argument follows essentially from  \cite[Theorem 4.3]{A961} and \cite[Theorem 1.5]{W04}. Since
$M$ is not Louville manifold, we could find a non-constant harmonic function $u: M\rightarrow \R$. For
every fixed $T>0$, we define $F_{T}:=\frac{1}{T}\int_0^T u(\gamma(t))\dd t$. Since $u$ is harmonic, by It\^o's formula
we obtain
\begin{equation}\label{t4-2-0}
u(\gamma(t))-u(o)=\int_0^t \langle \nabla u(\gamma(s)), U_s(\gamma)\dd \beta_s\rangle_{T_{\gamma(s)}M},
\end{equation}
where $\beta_s$ denotes the anti-development of $\gamma(\cdot)$, whose law is an $\R^n$-valued Brownian motion
under $\mu_{\R^+}^o$. Thus $N_t:=u(\gamma(t))-u(o)$ is a bounded martingale, according to the martingale
convergence theorem, there is a non-constant random variable $N_{\infty}$ such that
\begin{equation*}
\lim_{t\uparrow \infty}\mu_{\R^+}^o\big(\big|N_t-N_{\infty}\big|^2\big)=0,
\end{equation*}
which implies immediately
\begin{equation}\label{t4-2-1}
\lim_{T\uparrow \infty}\mu_{\R^+}^o\big(\big|F_T-N_{\infty}\big|^2\big)=0.
\end{equation}
On the other hand, set $F_T^R:=\frac{1}{T}\int_0^T\phi_R(\rho(o,\gamma(s)))u(\gamma(s))\dd s$, where $o\in M$, $\phi_R$ is defined as in the proof of Theorem \ref{t 4-1}.
Then by Lemma \ref{l2.8} it is easy to see that $F_T^R\in \D(\E^o_{\R^+})$ for $R, T>0$.
Note that for fixed $T>0$,  $F_T^R\rightarrow F_T$ in $L^2({\bf{E}}_{\R^+}^o(M), \mu_{\R^+}^o)$, as $R\rightarrow\infty$. We also have
$$\aligned\E_{\R^+}^o(F_T^R,F_T^R)\leq& \frac{1}{T^2}\int_0^T
\mu_{\R^+}^o\big(|\nabla u(\gamma(s))|^2\big)\dd s+\frac{4}{T^2}\int_0^T
\mu_{\R^+}^o\big(|u(\gamma(s))|^2\big)\dd s\\
\leq&\frac{1}{T^2}
\mu_{\R^+}^o\Big(\big|u(\gamma(T))-u(o)\big|^2\Big)+C_1\leq C,\endaligned$$
where $C, C_1$ are  constants independent of $R$ and  the second inequality follows from \eqref{t4-2-0}. This by \cite[Lemma I-2.12]{MR92} implies that $F_T\in \D(\E^o_{\R^+})$ and
$$DF_T(\gamma)(s)=\frac{1}{T}\big(U_s(\gamma)^{-1}
\nabla u(\gamma(s))\big)1_{[0,T]}(s),$$
 hence
\begin{equation}\label{t4-2-2}
\begin{split}
\lim_{T \uparrow \infty}\E_{\R^+}^o(F_T,F_T)&=\lim_{T \uparrow \infty}\frac{1}{T^2}\int_0^T
\mu_{\R^+}^o\big(|\nabla u(\gamma(s))|^2\big)\dd s\\
&=\lim_{T \uparrow \infty}\frac{1}{T^2}
\mu_{\R^+}^o\Big(\big|u(\gamma(T))-u(o)\big|^2\Big)\le
\lim_{T \uparrow \infty}\frac{4\|u\|_{\infty}}{T^2}=0,
\end{split}
\end{equation}
where the second equality follows from \eqref{t4-2-0}.

Combining \eqref{t4-2-1}, \eqref{t4-2-2} with the closbility of $(\E_{\R^+}^o,\D(\E_{\R^+}^o))$ yields that
$N_{\infty}$ is not a constant, $N_{\infty}\in \D(\E_{\R^+}^o)$ and $\E_{\R^+}^o(N_{\infty},N_{\infty})=0$. So
$(\E_{\R^+}^o,\D(\E_{\R^+}^o))$  is reducible.
\end{proof}

Note that if $M$ is a Cartan-Hadamard manifold with section curvature $-c_1(\rho(o,x)\vee1)^2\le {\rm Sec}_x(X_1,X_2)\le
-c_2(\rho(o,x)\vee1)^{-2}$ for some $c_1,c_2>0$ and every $x\in M$, $X_1,X_2\in T_xM$ with $|X_1|=|X_2|=1$, then $M$ is not a Louville manifold
(where ${\rm Sec}_x$ denotes the sectional Curvature tensor at $x\in M$).
So we have the following
result immediately.

\begin{cor}
If  $M$ is a Cartan-Hadamard manifold with section curvature $-c_1(\rho(o,x)\vee1)^2\le {\rm Sec}_x(X_1,X_2)\le
-c_2(\rho(o,x)\vee1)^{-2}$ for some $c_1,c_2>0$ and every $x\in M$, $X_1,X_2\in T_x M$ with $|X_1|=|X_2|=1$, then $(\E_{\R^+}^o,\D(\E_{\R^+}^o))$ is reducible. 
Hence $\mu_{\R^+}^o$ is not
ergodic for the Markov process associated with $(\E_{\R^+}^o,\D(\E_{\R^+}^o))$ constructed in Theorem \ref{T2.2}.
\end{cor}

\subsection{The whole line}\label{sect4}
In this section, we will study the functional inequality and ergodic property for the Dirichlet form
$(\E_{\R}^\nu,\D(\E_{\R}^\nu))$ constructed in Section \ref{sect3}, where $\nu(\dd x)=\nu(x) \dd x$ is a probability measure on
$M$ which is absolutely continuous with respect to volume (Lebesgue) measure on $M$. The case for $(\E_{\R}^o,\D(\E_{\R}^o))$
is similar and we omit the details here.

As in Section \ref{sect3}, for  $\gamma \in {\bf E}_{\R}(M)$,
we could decompose $\gamma=(\tilde \gamma, \bar \gamma)$ with
$$ \tilde \gamma(s):=\gamma(s),\ \bar \gamma(s):=\gamma(-s),\ s\ge 0$$. We also set
\begin{align*}
M_s(\gamma):=
\begin{cases}
& \hat M_s(\tilde \gamma),\ \ \ \ s\ge 0,\\
& \hat M_{-s}(\bar \gamma),\ \ \ s<0.
\end{cases}
\end{align*}
Here $\hat M_t(\gamma)$ denotes the solution to \eqref{eq3.4} with $\gamma \in {\bf E}_{\R^+}(M)$.
\begin{lem}\label{l4-2}
Suppose $M$ is compact, for every $F\in \F C_b$ with the form \eqref{eq3.01} we have
\begin{equation}\label{l4-2-1}
\begin{split}
\nabla_x \mu^x_{\R}(F)&=\sum_{j=1}^m\mu^x_{\R}\Big[\int_0^{T_j}\hat \partial_j f(\gamma)M_s(\gamma)U_s(\gamma)^{-1}\nabla g_j
\big(s,\gamma(s)\big)\dd s\Big]\\
&+\sum_{j=1}^k\mu^x_{\R}\Big[\int_{-\bar T_j}^0\hat \partial_{m+j} f(\gamma)M_s(\gamma)U_s(\gamma)^{-1}\nabla \bar g_j
\big(s,\gamma(s)\big)\dd s\Big],
\end{split}
\end{equation}
where $\hat \partial_j f(\gamma)$ denotes the same item as that in \eqref{eq3.03} and $U_s(\gamma)$ is defined in \eqref{eq3.01-1}.

\end{lem}
\begin{proof}
For simplicity, we only prove  \eqref{l4-2-1} for $F=f\Big(\int_0^T g(s,\gamma(s))\dd s\Big)$ for some
$f\in C_b^1(\R)$ and $g\in C_b^{0,1}([0,\infty)\times M)$. Other cases could be tackled similarly (by decomposing
into $\gamma=(\tilde \gamma, \bar \gamma)$).

For each $k\in \mathbb{N}^+$, let $F_k(\gamma):=f\Big(\sum_{i=1}^k \frac{T}{k}g\big(t_i,\gamma(t_i)\big)\Big)$ with
$t_i=\frac{iT}{k}$, $1\le i \le k$. Then  applying \cite[Lemma 3.3]{FW05} we  obtain that
\begin{equation*}
\begin{split}
\nabla_x \mu^x_{\R}(F_k)=\mu^x_{\R}\Big[\sum_{i=1}^k \frac{T}{k}\hat \partial f_k(\gamma)M_{t_i}U_{t_i}(\gamma)^{-1}
\nabla g\big(t_i,\gamma(t_i)\big)\Big],
\end{split}
\end{equation*}
where  $\hat \partial f_k(\gamma)=f'\Big(
\sum_{i=1}^k \frac{T}{k}g\big(t_i,\gamma(t_i)\big)\Big)$.

Based on such expression it is easy to verify that
\begin{align*}
& \lim_{k \rightarrow \infty}\int_M \Big|\nabla_x \mu^x_{\R}(F_k)-
\mu^x_{\R}\Big[\int_0^T\hat \partial f(\gamma)M_{s}U_{s}(\gamma)^{-1}
\nabla g\big(s,\gamma(s)\big)\dd s\Big]\Big|^2 \dd x=0,\\
& \lim_{k \rightarrow \infty}\int_M \big|\mu^x_{\R}(F_k)-\mu^x_{\R}(F)\big|^2
\dd x=0,
\end{align*}
where $\hat \partial f(\gamma):=f'\Big(\int_0^T g\big(s,\gamma(s)\big)\dd s\Big)$. According to this we could prove
for every smooth vector fields $V \in C^{\infty}(TM)$,
\begin{equation*}
\int_M \Big\langle \mu^x_{\R}\Big[\int_0^T\hat \partial f(\gamma)M_{s}U_{s}(\gamma)^{-1}
\nabla g\big(s,\gamma(s)\big)\dd s\Big] ,V(x)\Big\rangle_{T_x M}\dd x=-\int_M \mu^x_{\R}(F){\rm div}V(x)\dd x,
\end{equation*}
which means
$$\nabla_x  \mu^x_{\R}(F)= \mu^x_{\R}\Big[\int_0^T\hat \partial f(\gamma)M_{s}U_{s}(\gamma)^{-1}
\nabla g\big(s,\gamma(s)\big)\dd s\Big].$$
Thus \eqref{l4-2-1} holds for $F=f\Big(\int_0^T g(s,\gamma(s))\dd s\Big)$ and we have finished the proof.
\end{proof}



 \beg{thm}\label{t4-3}[Log-Sobolev inequality and Poincar\'e inequality]
\begin{itemize}
 \item [(1)]  Suppose that $\Ric\geq K$ for $K>0$ and the following log-Sobolev inequality holds for $\nu$ (on $M$) 
 \begin{equation}\label{t4-3-0}
\nu\big(f^2\log f^2\big)-\nu(f^2)\log \nu(f^2)\le C_1\int_M |\nabla f(x)|^2\nu(\dd x),\ \ \forall\ f\in C^1(M).
\end{equation}
Then the log-Sobolev inequality holds
\begin{equation}\label{t4-3-1}
\mu_{\R}^\nu(F^2\log F^2)\le \Big(\frac{8}{K^2}+\frac{2C_1}{K}\Big) \E_{\R}^\nu(F,F),\ \ \ \ F \in \F C_{c}.
\ \mu_{\R}^\nu(F^2)=1.\end{equation}

\item[(2)] Suppose $M$ is compact and the following Poincar\'e inequality holds
\begin{equation}\label{t4-3-0a}
\nu(f^2)-\nu(f)^2\le C_2\int_M |\nabla f(x)|^2\nu(\dd x),\ \ \forall\ f\in C^1(M),
\end{equation}
and there exists  $\varepsilon \in (0,1)$ such that
\begin{equation}\label{t4-3-2}
\delta_\varepsilon:=\sup_{T\in [0,\infty)}\delta_{\varepsilon}(T)<\infty,
\end{equation}
and
$$ C_0:=\int_0^{\infty}\eta(s)\dd s<\infty,$$
where $\delta_{\varepsilon}(T)$, $\eta(s)$ are defined by \eqref{t3-1a}.
Then the following Poincar\'e inequality holds,
\begin{equation}\label{t4-3-3a}
\mu_{\R}^\nu(F^2)-\mu_{\R}^\nu(F)^2 \le \big(\delta_\varepsilon+C_0C_2\big) \E_{\R}^\nu(F,F),\ \ \ \ F \in \F C_{c}.
\end{equation}
\end{itemize}
\end{thm}

\begin{Remark}
As explained by \cite[Chapter 5]{W05}, if $M$ is compact and $\nu(\dd x)=\nu(x)\dd x$ is a probability measure
such that $\inf_{x\in M}\nu(x)>0$, then the log-Sobolev inequality \eqref{t4-3-0} and Poincar\'e inequality
\eqref{t4-3-0a} hold. In particular, \eqref{t4-3-0} and \eqref{t4-3-0a} hold for the normalized volume measure when
$M$ is compact.
\end{Remark}


\ \newline\emph{\bf Proof of Theorem \ref{t4-3}.}
{\bf Step (1)}
Let $$G(\bar \gamma):=\sqrt{
\int_{{\bf E}_{\R^+}^x(M)}F^2(\tilde \gamma,\bar \gamma)\mu_{\R^+}^x (\dd \tilde \gamma)},$$
$$g(x):=\sqrt{\int_{{\bf E}_{\R^+}^x(M)}\int_{{\bf E}_{\R^+}^x(M)}F^2(\tilde \gamma, \bar \gamma)
\mu_{\R^+}^x (\dd \tilde \gamma)\mu_{\R^+}^x (\dd \bar \gamma)}=
\sqrt{\int_{{\bf E}_{\R}(M)}F^2(\gamma)\mu_{\R}^x(\dd \gamma)}.$$
Then we have for every $F \in \F C_c$ with form \eqref{eq3.01},
\begin{equation}\label{t4-3-3}
\begin{split}
& \quad \int_{{\bf E}_{\R}(M)} F^2(\gamma)\log F^2(\gamma) \mu_{\R}^\nu(\dd \gamma)\\
&=\int_M\int_{{\bf E}_{\R^+}^x(M)}\int_{{\bf E}_{\R^+}^x(M)}F^2(\tilde \gamma,\bar \gamma)
\log F^2(\tilde \gamma,\bar \gamma) \mu_{\R^+}^x(\dd \tilde \gamma)
 \mu_{\R^+}^x(\dd \bar \gamma)\nu(\dd x)\\
&\le 2C(K)\int_M\int_{{\bf E}_{\R^+}^x(M)}\int_{{\bf E}_{\R^+}^x(M)}
|\tilde DF(\tilde \gamma, \bar \gamma)|_{\HH_+}^2\mu_{\R^+}^x(\dd \tilde \gamma)
 \mu_{\R^+}^x(\dd \bar \gamma)\nu(\dd x)\\
 &\quad +\int_M\int_{{\bf E}_{\R^+}^x(M)}G^2(\bar \gamma)\log G^2(\bar \gamma)\mu_{\R^+}^x(\dd \bar \gamma)\nu(\dd x)\\
 &\le 2C(K)\int_M\int_{{\bf E}_{\R^+}^x(M)}\int_{{\bf E}_{\R^+}^x(M)}
|\tilde DF(\tilde \gamma, \bar \gamma)|_{\HH_+}^2\mu_{\R^+}^x(\dd \tilde \gamma)
 \mu_{\R^+}^x(\dd \bar \gamma)\nu(\dd x)\\
 &\quad + 2C(K)\int_M\int_{{\bf E}_{\R^+}^x(M)}|\bar D G(\bar \gamma)|_{\HH_+}^2
  \mu_{\R^+}^x(\dd \bar \gamma)\nu(\dd x)+\int_M g^2(x)\log g^2(x) \nu(\dd x)\\
  &\le 2C(K)\int_M\int_{{\bf E}_{\R^+}^x(M)}\int_{{\bf E}_{\R^+}^x(M)}
|\tilde DF(\tilde \gamma, \bar \gamma)|_{\HH_+}^2\mu_{\R^+}^x(\dd \tilde \gamma)
 \mu_{\R^+}^x(\dd \bar \gamma)\nu(\dd x)\\
 &\quad + 2C(K)\int_M\int_{{\bf E}_{\R^+}^x(M)}|\bar D G(\bar \gamma)|_{\HH_+}^2
  \mu_{\R^+}^x(\dd \bar \gamma)\nu(\dd x)+C_1\int_M |\nabla g(x)|^2 \nu(\dd x)+
  \mu_{\R}^\nu(F^2)\log \mu_{\R}^\nu(F^2).
\end{split}
\end{equation}
 Here in the second step we  applied \eqref{eq3.2} to $F(\cdot,\bar \gamma)$ (with $\bar \gamma$ fixed) with
 $ \tilde DF(\tilde \gamma, \bar \gamma)$ denoting the $L^2$ gradient with respect to the variable $\tilde \gamma\in
 {\bf E}_{\R^+}(M)$; in the third step we  applied \eqref{eq3.2} to $G(\bar \gamma)$ with $\bar DG(\bar \gamma)$ denoting
 $L^2$ gradient with respect to the variable $\bar \gamma\in
 {\bf E}_{\R^+}(M)$ and the property $\int_{{\bf E}_{\R^+}^o(M)}G^2(\bar \gamma)\mu_{\R^+}^x(\dd \bar \gamma)=
 g^2(x)$ ; in the last step we  applied \eqref{t4-3-0} to $g(x)$ and the property
 $\int_M g^2(x)\nu(\dd x)=\mu_{\R}^\nu(F^2)$. At the same time, it holds
 \begin{align*}
& \big|\tilde D F(\tilde \gamma, \bar \gamma)\big|_{\HH_+}^2+
 \big|\bar D F(\tilde \gamma, \bar \gamma)\big|_{\HH_+}^2=\big|DF( \gamma)\big|_{\HH}^2,\\
&\big|\bar D G(\bar \gamma)\big|_{\HH_+}^2=\frac{\Big|\int_{{\bf E}_{\R^+}^x(M)}
F(\tilde \gamma, \bar \gamma)\bar DF(\tilde \gamma, \bar \gamma) \mu_{\R^+}^x(\dd \tilde \gamma)\Big|_{\HH_+}^2}
{\int_{{\bf E}_{\R^+}^x(M)}F^2(\tilde \gamma, \bar \gamma)\mu_{\R^+}^x(\dd \tilde \gamma)}
\le \int_{{\bf E}_{\R^+}^x(M)}
\big|\bar DF(\tilde \gamma, \bar \gamma)\big|_{\HH_+}^2 \mu_{\R^+}^x(\dd \tilde \gamma).
 \end{align*}
 Meanwhile by \eqref{l4-2-1},
 \begin{align*}
 & |\nabla g(x)|^2=\frac{\Big|\mu_{\R}^x\Big[F(\gamma)J(\gamma)\Big]\Big|^2}
{\int_{{\bf E}_{\R}^x(M)}F^2(\gamma)\mu_{\R}^x(\dd \gamma)}\le \mu_{\R}^x[|J(\gamma)|^2],
 \end{align*}
 where
 \begin{equation}\label{t4-3-4}
 \begin{split}
J(\gamma)&=\sum_{j=1}^m \int_0^{T_j}\hat \partial_j f(\gamma)M_s(\gamma)U_s(\gamma)^{-1}\nabla g_j
\big(s,\gamma(s)\big)\dd s
\\&+\sum_{j=1}^k\int_{-\bar T_j}^0\hat \partial_{m+j} f(\gamma)M_s(\gamma)U_s(\gamma)^{-1}\nabla \bar g_j
\big(s,\gamma(s)\big)\dd s\\
&=\int_{-T}^T M_s(\gamma) DF(\gamma)(s)\dd s
 \end{split}
 \end{equation}
 with $T:=\max\{\max_{1\le j \le m}T_j,\max_{1\le j \le k}\bar T_j\}$. Based on the expression of $J(\gamma)$ above we arrive at
 \begin{align*}
|\nabla g(x)|^2&\le \mu_{\R}^x \Big[\Big(\int_{-T}^T \|M_s(\gamma)\|^2 \dd s\Big)\cdot
\Big(\int_{-T}^T |DF(\gamma)|^2(s)\dd s\Big)\Big]\\
&\le 2\Big(\int_0^\infty e^{-Ks}\dd s\Big)\mu_{\R}^x\Big[|DF(\gamma)|^2_{\HH}\Big],
 \end{align*}
 where the last step follow from the estimates $\|M_s(\gamma)\|\le e^{-\frac{K|s|}{2}}$ for all $s\in \R$.

Finally, combining all the estimates above into \eqref{t4-3-3} yields \eqref{t4-3-1}.

{\bf Step (2)} Similar as \eqref{t4-3-3} (and apply \eqref{t3-2}) we obtain
\begin{equation}\label{t4-3-5}
\begin{split}
& \quad \int_{{\bf E}_{\R}(M)} F^2(\gamma)\mu_{\R}^\nu(\dd \gamma)-
\Big(\int_{{\bf E}_{\R}(M)} F(\gamma)\mu_{\R}^\nu(\dd \gamma)\Big)^2\\
&\le \delta_\varepsilon \int_M\int_{{\bf E}_{\R^+}^x(M)}\int_{{\bf E}_{\R^+}^x(M)}
|\tilde DF(\tilde \gamma, \bar \gamma)|_{\HH_+}^2\mu_{\R^+}^x(\dd \tilde \gamma)
 \mu_{\R^+}^x(\dd \bar \gamma)\nu(\dd x)\\
 &+\delta_\varepsilon \int_M\int_{{\bf E}_{\R^+}^x(M)}|\bar D Q(\bar \gamma)|_{\HH_+}^2
\mu_{\R^+}^x(\dd \bar \gamma)\nu(\dd x)+C_2\int_M |\nabla q(x)|^2 \nu(\dd x),
\end{split}
\end{equation}
where
$$Q(\bar \gamma):=\int_{{\bf E}_{\R^+}^x(M)}F(\tilde \gamma,\bar \gamma)\mu_{\R^+}^x(\dd \tilde \gamma),\ \
q(x):=\int_{{\bf E}_{\R}^x(M)}F(\gamma)\mu_{\R}^x(\dd \gamma).$$
Still by the same arguments in {\bf Step (1)} we could show
\begin{align*}
& |\bar D Q(\bar \gamma)|_{\HH_+}^2=\Big|\int_{{\bf E}_{\R^+}^x(M)} \bar
D F(\tilde \gamma,\bar \gamma)\mu_{\R^+}^x(\dd \tilde \gamma)\Big|_{\HH_+}^2
\le \int_{{\bf E}_{\R^+}^x(M)} \big|\bar
D F(\tilde \gamma,\bar \gamma)\big|_{\HH_+}^2\mu_{\R^+}^x(\dd \tilde \gamma),\\
& |\nabla q(x)|^2\le \big|\mu_{\R}^x\big[J(\gamma)\big]\big|^2
\le \mu_{\R}^x\Big[\int_{-T}^T \|M_s\|^2 \dd s\Big]\cdot
\mu_{\R}^x\Big[\int_{-T}^T |DF(s)|^2 \dd s\Big]\\
&\le 2\Big(\int_0^{\infty}\eta(s)\dd s\Big)\cdot \mu^x_{\R}\Big[\big|DF(\gamma)\big|_{\HH}^2\Big],
\end{align*}
where $J(\gamma)$ is defined by \eqref{t4-3-4} and the last step above is due to
\begin{align*}
&\quad \mu_{\R}^x\Big[\int_{-T}^T \|M_s\|^2 \dd s\Big]=2\mu_{\R^+}^x\Big[\int_0^T \|M_s\|^2 \dd s\Big]\\
&\le 2\mu_{\R^+}^x\Big[\int_0^T \exp(-\int_0^s K(\gamma(r))\dd r) \dd s\Big]\le
2\int_0^\infty \eta(s)\dd s.
\end{align*}
Then combining all the above estimates into \eqref{t4-3-5} yields \eqref{t4-3-3a}.
$\hfill\square$
\vskip.10in

When $M=\R^n$, the Markov process constructed in Section 3 corresponds to the solutions to the stochastic heat equations.  The most interesting case is that $\nu$ is given by Lebesgue measure, which is related the the stochastic heat equations without any boundary condition. In this case the reference measure has  infinite mass. So we do not investigate the long time behavior here.

Following the same procedure in Theorem \ref{t4-2} we can still get the following result. So we omit the proof here.
\begin{thm}\label{t4-5}
If $M$ is not a Liouville manifold and $\nu$ is a probability measure,
then $(\E_{\R}^\nu,\D(\E_{\R}^\nu))$ is reducible. Hence $\mu_{\R}^\nu$ is not
ergodic for the Markov process associated with $(\E_{\R}^\nu,\D(\E_{\R}^\nu))$.
\end{thm}

\section*{Acknowledgments}
We would like to thank Professor Tadahisa Funaki for helpful suggestions.

\beg{thebibliography}{99}

\leftskip=-2mm
\parskip=-1mm

\bibitem{A96} S. Aida, \emph{Logarithmic Sobolev inequalities on loop spaces over compact Riemannian
manifolds}, in ``Proc. Fifth Gregynog Symp., Stoch. Anal. Appl. ''(I. M. Davies, A. Truman,and K. D. Elworthy, Eds.), pp. 1-15, World Scientific, 1996

\bibitem{A961} S. Aida, \emph{Gradient estimates of harmonic functions and the asymptotics of spectral gaps on path spaces},
 Interdiscip. Inform. Sci. 2 (1996), 75--84.

\bibitem{AE95}S. Aida and D. Elworthy, \emph{Differential calculus on Path and Loop spaces 1: Logarithmic
Sobolev inequalities on Path spaces,} C.R. Acad. Sci. Paris Ser. I Math. 321 (1995),
97--102.

\bibitem{ALR93}S. Albeverio, R. L\'{e}andre and M. R\"{o}ckner, \emph{Construction of rotational invariant diffusion
on the free loop space,} C. R. Acad. Paris Ser. 1 316 (1993) 287--292

\bibitem{AR91}S. Albeverio, M. R\"{o}ckner, \emph{Stochastic differential equations in infinite
dimensions: Solutions via Dirichlet forms}, Probab. Theory Related Field 89 (1991) 347-386.

\bibitem{AD99}  L. Andersson,  B. K. Driver, \emph{Finite-dimensional approximations to
Wiener measure and path integral formulas on manifolds}, J. Funct. Anal, 165, no. 2, (1999), 430¨C498.

\bibitem{ATW06} M. Arnaudon, A. Thalmaier, F.-Y. Wang, \emph{
Harnack inequality and heat kernel estimates on manifolds with curvature unbounded below},
Bull. Sci. Math. 130 (2006), 223--233.

\bibitem{BL}D. Barden,
H. LE, \emph{Some consequences of the nature of the distance function on cut locus in a Riemannian manifold,}
J.London Math.Soc. (2)56 (1997) 369¡À383

\bibitem{BGHZ19} Y. Bruned, F. Gabriel, M. Hairer, L. Zambotti, \emph{Geometric stochastic heat equations,}
arXiv: 1902.02884.

\bibitem{BHZ16}Y. Bruned, M. Hairer, and L. Zambotti, \emph{Algebraic renormalisation of regularity structures.
arXiv:1610.08468,} pages 1-84, 2016.

\bibitem{CF12}Z.-Q. Chen and M. Fukushima, Symmetric Markov processes, time change, and
boundary theory, vol. 35 of London Mathematical Society Monographs Series, Princeton
University Press, Princeton, NJ, 2012

\bibitem{CHL97}B. Capitaine, E. P. Hsu and M. Ledoux, \emph{Martingale representation and a sim-
ple proof of logarithmic Sobolev inequalities on path spaces,}  Elect. Comm. Probab.
2(1997), 71-81.

\bibitem{CH16}A. Chandra and M. Hairer. \emph{An analytic BPHZ theorem for regularity structures}
arXiv:1612.08138, pages 1-113, 2016.

\bibitem{Cha} I. Chavel, \emph{Riemannian geometry. A modern introduction.}
Cambridge Studies in Advanced Mathematics, 98. Cambridge University Press, Cambridge, 2006.

\bibitem{CLW11}X. Chen, X.-M. Li and B. Wu, \emph{A concrete estimate for the weak Poincar¡äe inequality on
loop space,} Probab.Theory Relat. Fields 151 (2011), no.3-4, 559--590.

\bibitem{CLW17}X. Chen, X.-M. Li and B. Wu, \emph{Small time gradient and Hessian estimates for logarithmic
heat kernel on a general complete manifold,} Preprint.

\bibitem{CLW18} X. Chen, X.-M. Li and B. Wu, \emph{Stochastic analysis on loop space over general Riemannian
manifold,} Preprint.

\bibitem{CLWFL}X. Chen, X.-M. Li and B. Wu, \emph{Analysis on Free Riemannian Loop Space,} Preprint.

\bibitem{CW14} X. Chen and B. Wu,  \emph{Functional inequality on path space over a non-compact Riemannian manifold,} J. Funct. Anal. 266 (2014), no. 12, 6753--6779

\bibitem{CZ} K.L.   Chung, Z. Zhao, \emph{From Brownian motion to Schr\"{o}dinger¡¯s equation,} A series of
Comprehensive Studies in Mathematics 312, Springer, 1985

\bibitem{Dav} E. B. Davies, \emph{Heat kernel bounds, conservation of probability and the feller property,} Journal d'Analyse Math\'{e}matique, 58, Issue 1,  (1992)  99--119

\bibitem{DR92}B. K. Driver and M. R\"{o}ckner,\emph{Construction of diffusions on path and loop spaces of
compact Riemannian manifolds,} C. R. Acad. Sci. Paris Ser. I 315 (1992) 603--608

\bibitem{Dri92}B. K. Driver, \emph{A Cameron-Martin quasi-invariance theorem for Brownian motion on
a compact Riemannian manifolds,} J. Funct. Anal. 110 (1992) 273--376.

\bibitem{Dri94} B. K. Driver, \emph{A Cameron-Martin type quasi-invariance theorem for
pinned Brownian motion on a compact Riemannian manifold} Trans. Amer.
Math. Soc. 342, no. 1, (1994), 375--395.

\bibitem{EL}  K. D. Elworthy, Xue-Mei Li.: A class of Integration by parts formulae in stochastic analysis I,
\emph{``It\^o's Stochastic Calculus and Probability Theory''} (dedicated to It\^o on the occasion of his eightieth birthday),
S. Watanabe, K. D. Elworthy  eds., (1996), Springer.

\bibitem{ELL97}  K. D. Elworthy, Y. Le Jan and X.-M. Li, \emph{
On the geometry of diffusion operators and stochastic flows,}
Lecture Notes in Mathematics, 1720. Springer-Verlag, Berlin, 1999.

\bibitem{ELR93}  K. D. Elworthy, X.-M. Li and S. Rosenberg, \emph{Curvature and topology: spectral positivity. Methods and applications of global analysis,} 45--60, 156, Novoe Global. Anal., Voronezh. Univ. Press, Voronezh, (1993).

\bibitem{FM} S. Fang P. Malliavin, \emph{
Stochastic analysis on the path space of a Riemannian manifold. I. Markovian stochastic calculus,}
J. Funct. Anal. 118 (1993), 249--274.

\bibitem{FW05} S. Fang and  F.-Y. Wang, \emph{Analysis on free Riemannian
path spaces,} Bull. Sci. Math. 129 (2005), 339--355.

\bibitem{FWW} S. Fang, F.-Y. Wang and B. Wu, \emph{Transportation-cost inequality on path spaces with uniform distance,}
Stochastic Process. Appl. 118 (2008), 2181--2197.

\bibitem{FW17} S. Fang and B. Wu, \emph{Remarks on spectral gaps on the Riemannian path space,}
Electron. Commun. Probab.
22(2017), no. 19, 1--13.

\bibitem{Fun88} T. Funaki, \emph{On diffusive motion of closed curves,} In Probability theory
and mathematical statistics (Kyoto, 1986), vol. 1299 of Lecture Notes
in Math., 86-94. Springer, Berlin, 1988.

\bibitem{Fun92} T. Funaki, \emph{A stochastic partial differential equation with values in a manifold},
J. Funct. Anal. 109, no. 2, (1992), 257--288.

\bibitem{FM} T. Funaki and H. Masato, \emph{A coupled KPZ equation, its two types of approximations and existence of global solutions},
J. Funct. Anal. 273, (2017), 1165--1204.

\bibitem{FQ} T. Funaki and J. Quastel, \emph{KPZ equation, its renormalization and invariant measures},
Stoch. Partial Differ. Equ. Anal. Comput. 3 (2015), 159--220.

\bibitem{FX} T. Funaki and B. Xie, \emph{A stochastic heat equation with the
distributions of L\'evy processes as its invariant measures}, Stochastic Process. Appl. 119 (2009), 307--326.

\bibitem{FOT94}M. Fukushima, Y. Oshima, and M. Takeda, \emph{Dirichlet Forms and Symmetric
Markov Processes,} de Gruyter, Berlin (1994)

\bibitem{FW05}S. Fang and F.-Y. Wang, \emph{Analysis on free Riemannian path spaces,} Bull. Sci. Math.
129 (2005) 339--355.


\bibitem{GWu} R. E. Greene and H.X. Wu,  \emph{Function Theory on Manifolds Which Possess a Pole,}
Lecture Notes in Math.  699, Springer-Verlag, (1979).

\bibitem{GW06} M. Gourcy,  L. Wu, \emph{Logarithmic Sobolev inequalities of diffusions for the L2-metric}. Potential
 Anal. 25 77--102 (2006)

\bibitem{Hai14}M. Hairer. \emph{A theory of regularity structures} Invent. Math., 198(2):269--504, 2014

\bibitem{Hai16}M. Hairer. \emph{The motion of a random string,} arXiv:1605.02192, pages 1--20, 2016.

\bibitem{H97}E.P. Hsu, \emph{Logarithmic Sobolev inequalities on path spaces over Riemannian manifolds,} Comm. Math.
Phys. 189 (1997) 9--16.

\bibitem{Hsu97} E. P. Hsu.: Integration by parts in loop spaces,
\emph{Math. Ann.} \textbf{309} (1997), 331--339.

\bibitem{Hsu2} E. P. Hsu.: Stochastic Analysis on Manifold, \emph{American Mathematical Society,} 2002.

\bibitem{IM85} A. Inoue, Y. Maeda, \emph{On integral transformations associated with a
certain Lagrangian as a prototype of quantization,} J. Math. Soc. Japan 37,
no. 2, (1985), 219--244.

\bibitem{L93}R. L\'eandre,  \emph{Integration by parts formulas and rotationally invariant Sobolev calculus
on free loop spaces,}  J. Geom. Phys. 11 (1993) 517¡V528.
\bibitem{L97}R. L\'eandre, \emph{Invariant Sobolev calculus on the free loop space,} Acta Appl. Math. 46
(1997) 267¡V350.
\bibitem{LN77} R. L\'eandre and J. Norris,  \emph{Integration by parts and Cameron¡VMartin formulas for
the free path space of a compact Riemannian manifold,} S?em. Probab. XXXI (1977)
16¡V23.

\bibitem{L04} J.-U. L\"{o}bus, \emph{A class of processes on the path space
over a compact Riemannian manifold with unbounded diffusion,} Tran.
Ame. Math. Soc. (2004), 1--17.

\bibitem{MR92}Z. M. Ma and M. R\"{o}ckner, \emph{Introduction to the Theory of (Non-Symmetric) Dirichlet
Forms ,}(Springer-Verlag, Berlin, Heidelberg, New York, 1992.

\bibitem{MR00}Z. M. Ma and M. R\"{o}ckner, \emph{Construction of diffusions on configuration spaces,} Osaka
J. Math. 37 (2000) 273--314.

\bibitem{N} A. Naber, Characterizations of bounded Ricci curvature on smooth and nonsmooth spaces, {\it arXiv: 1306.6512v4}.

\bibitem{Nor98} J. R. Norris,\emph{Ornstein-Uhlenbeck processes indexed by the circle}, Ann.
Probab. 26, no. 2, (1998), 465-478.

\bibitem{RWZZ17} M. R\"{o}ckner, B. Wu, R.C. Zhu and R.X. Zhu,\emph{Stochastic Heat Equations with Values in a
Manifold via Dirichlet Forms}, arxiv: 1711.09570.

\bibitem{S00} D.W. Stroock, \emph{An introduction to the analysis of paths on a Riemannian manifold,} Mathematical Surveys and Monographs, 74. American Mathematical Society, Providence, RI, 2000.

\bibitem{T97} A. Thalmaier, \emph{On the differentiation of heat semigroups and poisson integrals}, Stochastics
and Stochastic Reports 61 (1997) 297--321.

 \bibitem{TW98} A. Thalmaier and F.-Y. Wang, \emph{Gradient estimates for harmonic functions on regular domains
in Riemannian manifolds}, J. Funct. Anal. 155 (1998) 109--124.

\bibitem{W99} F.-Y. Wang, \emph{Spectral gap on path spaces with infinite time interval},
Sciences in China. 42, (1999), 600-604.

\bibitem{W04} F.- Y. Wang, \emph{Weak poincar\'{e} Inequalities on path
spaces,} Int. Math. Res. Not. (2004), 90--108.

\bibitem{W05} F.Y. Wang, \emph{Functional Inequalities, Markov Semigroup and Spectral Theory}. Chinese Sciences
Press, Beijing (2005)

\bibitem{WW08}F.-Y. Wang and B. Wu, \emph{Quasi-regular Dirichlet forms on path and loop spaces,} Forum
Math. 20 (2008) 1085-1096.

\bibitem{WW09} F. -Y. Wang and B. Wu, \emph{Quasi-Regular Dirichlet Forms on Free Riemannian Path and Loop Spaces,} Inf. Dimen. Anal. Quantum Probab.
and Rel. Topics 2(2009) 251--267.

\bibitem{WW16}F.- Y. Wang and B. Wu, \emph{Pointwise Characterizations of Curvature and Second
Fundamental Form on Riemannian Manifolds,} arXiv:1605.02447.

\bibitem{W1} B. Wu, \emph{Characterizations of the upper bound of Bakry-Emery curvature,} arXiv:1612.03714.


\end{thebibliography}
\end{document}